\documentclass[cag,a4paper,onlinefirst]{ipart_v2}

\Year{2026}
\Vol{34}
\Issue{3}
\Doi{10.4310/CAG.260817233214}
\firstpage{1}
\Pubonline{August 18, 2026}

\usepackage[english]{babel}
\usepackage[utf8]{inputenc}

\usepackage{amssymb}
\usepackage{amsthm}
\usepackage{amsmath}
\usepackage[only,Yup]{stmaryrd} %%% PARA O FIBRADO DE MOLINO
\usepackage{mathrsfs}    %%% PARA PSEUDOGRUPOS
\usepackage[new]{old-arrows}   %%% PARA SETAS DE INJEÇÕES
\usepackage{comment}
\usepackage{soul}
\usepackage{cancel}

\usepackage[all,cmtip]{xy}

\usepackage{enumerate}

\usepackage{multicol}

\DeclareMathOperator{\diam}{diam}
\DeclareMathOperator{\codim}{codim}

\DeclareMathOperator{\ric}{Ric}

\mathchardef\ordinarycolon\mathcode`\:
\mathcode`\:=\string"8000
\begingroup \catcode`\:=\active
  \gdef:{\mathrel{\mathop\ordinarycolon}}
\endgroup

\newenvironment{proofoutline}{\proof[Proof outline]}{\endproof}
\newenvironment{proofcomment}{\proof[Comment on the proof]}{\endproof}
\theoremstyle{definition}
\newtheorem{example}{Example}[section]
\newtheorem{definition}[example]{Definition}
\newtheorem{remark}[example]{Remark}
\theoremstyle{plain}
\newtheorem{proposition}[example]{Proposition}
\newtheorem{theorem}[example]{Theorem}
\newtheorem{lemma}[example]{Lemma}
\newtheorem{corollary}[example]{Corollary}

\newtheorem{thmx}{Theorem}
\newtheorem{corx}[thmx]{Corollary}
\newcommand{\dif}[0]{\mathrm{d}}

\newcommand{\f}[0]{\mathcal{F}}
\newcommand{\g}[0]{\mathcal{G}}

\begin{document}

\title[Transverse geometry of Lorentzian foliations]{Transverse geometry of Lorentzian foliations with applications to\\ Lorentzian orbifolds}

\author[F. C.~Caramello Jr., et al.]{Francisco C.~Caramello Jr., H. A. Puel Martins,\\ and Ivan P. Costa e Silva}

\begin{abstract}
\vspace{-1.5em}
We prove a transverse diameter theorem in the context of Lorentzian foliations, which can be interpreted as a Hawking--Penrose-type singularity theorem for timelike geodesics transverse to the foliation. In order to develop the necessary machinery we introduce and study a novel causality structure on the leaf space via the transverse Lorentzian geometry on the foliated manifold. We describe the initial rungs of a transverse causal ladder and relate them to their standard counterparts on an underlying foliated spacetime. We show how these results can be interpreted as doing Lorentzian (and more generally semi-Riemannian) geometry on low-regularity spaces that can be realized as leaf spaces of foliations. Accordingly, we discuss how all of these concepts and results apply to Lorentzian orbifolds, insofar as these can be seen as leaf spaces of a specific class of Lorentzian foliations. In particular, we derive an associated Lorentzian timelike diameter theorem on orbifolds. 
\vspace{-1.5em}
\end{abstract}

\maketitle
\tableofcontents

\section{Introduction}

Low-regularity analogues of concepts and results in Lorentzian geometry have attracted considerable interest for the better part of the last decade. This has taken place either via the analysis of geometric objects of low differentiability on an otherwise smooth (at least $C^2$) background manifold, or even by relaxing the smoothness requirement on the underlying manifold itself. Some highlights of the former trend include the causality of spacetimes with $C^0$ metric \cite{chrusciellowcausal1,chrusciellowcausal2}, Penrose and Hawking singularity theorems for $C^{1,1}$-metrics \cite{c11metric1,c11metric2,c11metric3}, event horizons of low regularity and applications \cite{area}, and a causality theory for cone structures on manifolds \cite{conefields1,conefields2}. The study of Lorentz-geometric features in the absence of differentiability and even of a manifold structure has been pursued at an especially energetic pace in the theory of Lorentzian length spaces (see, e.g., \cite{length1,length2,length3} and references therein). These extensions have been prompted in part by required applications in analytical and/or physical contexts where assuming smoothness may be too restrictive \cite{analysis1,analysis2}, or simply by the challenge of developing Lorentzian analogues of concepts first considered in the context of (synthetic) Riemannian geometry, such as limits of Gromov-Hausdorff convergence or synthetic bounds on the Ricci tensor \cite{mondino,ms}. In a more speculative vein, physicists have long entertained the idea that the manifold structure of spacetime in classical geometric theories of gravity should cease to apply in the context of a quantum gravity theory, and to put that notion on a sound basis provides yet another motivation for these recent developments.  

Of course, the idea of extending geometric notions initially defined in a smooth setting to situations in which such smoothness is relaxed or modified is not new, and indeed has been a recurring theme in many areas of geometry. For instance, the synthetic metric geometry of Alexandrov spaces arising as Gromov-Hausdorff limits of Riemannian manifolds that inspire their Lorentzian counterparts indicated above has been an already established research area for many decades (see. e.g. \cite{burago,harvey} and references therein), while the algebraic varieties and schemes, which are the stock-in-trade of algebraic geometry and as such nearly as old as the field itself, are often quite singular from a differential-geometric perspective \cite{alggeom}. Even within Lorentzian geometry there have been a number of earlier attempts to develop abstract aspects of causality theory that might dispense with the need for an underlying manifold structure (see, e.g., \cite{harris,causalspaces}).

Another very well-known ``irregular'' geometric structure especially pertinent for us here is given by \textit{orbifolds}. These are a fruitful generalization of the notion of manifold introduced by I. Satake \cite{satake} in 1956\footnote{Satake named them \textit{V-manifolds}. The dissemination of the term \textit{orbifold} is due to W. Thurston \cite{thurston}.}. Orbifolds appear naturally in many areas of mathematics as orbit spaces of locally free Lie group actions. They also play a prominent role in the theory of 3-manifolds, as quotients of Seifert fibrations. In physics they arise, for instance, from symplectic reductions in Hamiltonian mechanics, and as phase spaces after one formally quotients infinite-dimensional spaces of fields by groups of symmetries of a system, such as gauge transformations in Yang-Mills theory and diffeomorphisms in general relativity \cite{emmrich}. They are also particularly important in string theory, or in field theories with a finite group of global symmetries \cite{adem,chenruan,strings1}. Orbifolds have a rich geometric theory which includes intrinsically defined notions such as semi-Riemannian metrics, geodesics and curvature with many analogues of standard results, for instance an orbifold version of the Bonnet-Myers' finite diameter theorem (see, e.g., \cite{caramello3} for a didactic introduction with a wealth of references and a list of such geometric results). 

Now, the low-regularity geometric generalizations mentioned above have often entailed the development of clever strategies and new specific techniques. While some of these are fairly natural, there are others arguably less so. Thus, for example, in the theory of Kunzinger--Sämann Lorentzian length spaces the definition of geodesics may not coincide with the standard one on manifolds when these notions intersect, and the generalization of a convex normal neighborhood is given via the rather technical new notion of ``localizability'' (conf. Definition 3.16 in \cite{length3}). The mere definition of ``smooth function'' in the orbifold sense is rather involved (see, e.g., Section~`1.4 of \cite{caramello3}). That entails frequent hurdles in each attempt to come up with low-regularity generalizations of well-known results.   

However, there is a rather natural means of systematically applying standard, high-regularity differential-geometric techniques to low-regularity spaces, provided these that can be obtained as \textit{leaf spaces} of (regular) foliations on a manifold. Leaf spaces form a rather broad class of geometrically interesting, low-regularity topological spaces that includes all orbifolds (conf. Prop. 4.1.4 in \cite{caramello3}).  Leaf spaces can have a notoriously ``bad'' structure from a differential-geometric standpoint, often with a non-Hausdorff topology. Indeed, in the preface of his celebrated book on noncommutative geometry \cite{connes} A. Connes elects these spaces as one of his examples of topological spaces where standard geometric and analytic techniques break down and his particular brand of ``algebraic desingularization'' applies. However, in foliation theory it has long been acknowledged that the structure of these spaces can alternatively be studied by using fairly standard differential-geometric techniques, using the so-called \textit{transverse geometry} of the foliation. This ``geometric desingularization'' method will underscore our approach here. 

One general goal of this work is to introduce and illustrate this ``transverse-geometry-equals-leaf-space-geometry'' philosophy to an audience of mathematical relativists and Lorentzian geometers interested in low-regularity geometry, as well as foliation theorists with an interest in physical applications of their \textit{m\'{e}tier}, via a systematic study of the transverse geometry of the so-called \textit{Lorentzian foliations} and their leaf spaces. We shall see that any spacetime, and more generally all the so-called \text{Lorentzian orbifolds} can be realized as leaf spaces of Lorentzian foliations in a natural way. However, let us hasten to emphasize that this is \textit{not} a review article, and while we will recall a number of concepts and results which are well-known in the foliation theory community --- we shall indicate clearly when that is the case --- we also explore here a number of notions and concepts that to the best of our knowledge are new. 

We also have two more specific goals here. The first is to initiate the development of the elements of a (to our knowledge entirely novel) \textit{causality theory on leaf spaces} by means of systematic study of transversal causal notions on Lorentzian foliations. It is very important to distinguish at the outset these \textit{Lorentzian foliations} from the more mundane notion of foliation in spacetimes. Although these notions intersect, they should not be not be conflated, and relating them appropriately will take up a sizable part of our efforts here.  We relate the transverse causal notions with standard one and also include some --- though not all --- steps on a \textit{transverse causal ladder} analogous to that of spacetimes, and their most basic features. It will be clear to the knowledgeable reader that there is much room for development along these two tiers, which we nevertheless defer to future work. 

Another goal, which will also provide a concrete application illustrating the broad concepts and results outlined above, is to discuss and prove the following novel \textit{transverse timelike diameter theorem}:

\begin{thmx}\label{teointrotransversediameter}
Let $(M,\mathcal{F}, g_\intercal)$ be a codimension $q$ transversely globally hyperbolic Lorentzian foliation with compact leaves and such that the leaf space $M/\f$ is Hausdorff. If the \emph{transverse Ricci tensor} obeys the inequality $\ric_\intercal\geq (q-1)C>0$ on unit transversely timelike vectors for some constant $C>0$, then the \emph{the transverse timelike diameter} satisfies
$$\diam_\intercal(M,\f, g_\intercal))\leq\frac{\pi}{\sqrt{C}}.$$
\end{thmx}

Notice that Theorem \ref{teointrotransversediameter} restricts to the usual diameter theorem when $\f$ is the trivial foliation by points of $M$. Also, since any orbifold can be realized as the leaf space of a foliation, it has an immediate consequence for the so-called \textit{Lorentzian orbifolds}:

\begin{corx}
Let $(\mathcal{O}, g)$ be an $n$-dimensional globally hyperbolic Lorentzian orbifold. If $\ric_\mathcal{O}\geq (n-1)C>0$ for unit timelike vectors, then the \emph{timelike diameter} satisfies
$$\diam(\mathcal{O}, g)\leq\frac{\pi}{\sqrt{C}}.$$
\end{corx}

Just as in the standard timelike diameter theorem on spacetimes \cite[Thm. 11.9]{beem} the finiteness of the timelike diameter should be interpreted --- unlike its analogue on complete Riemannian manifolds --- as a (Hawking--Penrose type) \textit{singularity theorem} for orbifolds, since it entails that all timelike geodesics are incomplete.

Still in the same vein, the techniques developed here also lead to a generalized timelike diameter theorem which, in contrast with its standard version that requires controlling the Ricci curvature of $M$ on timelike vectors, involves a bound on the transverse Ricci curvature of some Lorentzian foliation on $M$. Precisely:

\begin{thmx}\label{theointrodiameter}
Let $(M, g)$ be a globally hyperbolic spacetime. Let $\mathcal{F}$ be a codimension $q\geq2$ foliation on $M$ whose leaves are all future causally complete, spacelike submanifolds with respect to $g$, and such that $g$ is bundle-like with respect to $\f$. Suppose that $\ric_\intercal\geq (q-1)C>0$ on $g_\intercal$-unit transversely timelike vectors, for the associated transverse metric $g_\intercal$. Then
$$\diam(M, g)\leq\frac{\pi}{\sqrt{C}}.$$
\end{thmx}

Again, the standard timelike diameter theorem also corresponds to the case that $\f$ is the trivial foliation by points in Theorem \ref{theointrodiameter}. However, as we discuss in Section \ref{section: constrasting ricci and transverse ricci}, in the nontrivial cases the hypothesis on $\ric_\intercal$ is independent of bounds on $\ric_g$.

The paper is organized as follows. Section \ref{section:preliminaries} gives a relatively self-contained description of the main concepts and results on foliations, including subsections on their transverse geometry and orbifolds. We also give many illustrative basic examples. This material is fairly standard and well-known to foliation theorists, and those familiar with it may skim over it just to get familiar with our notation. 

Section \ref{section: SR foliations} briefly reviews the so-called \textit{semi-Riemannian foliations}. Though perhaps not as well-known in this generality even in the foliation community, we believe the main notions there are easily accessible to most geometers, and so we have endeavored to keep it as brief as possible and refer to the specialized literature for further details. 

Sections \ref{section: Lorentz fol1} and \ref{section: Lorentzfol2} lay out a detailed description of Lorentzian foliations, including a novel causality theory on leaf space and a sketch of some rungs of the corresponding causal ladder. It has been long recognized by Lorentzian geometers that the causal and chronological relations can be studied abstractly, without an explicit reference to an underlying manifold. Accordingly, we directly introduce these relations on leaf spaces and study them in the spirit of Kronheimer and Penrose's \textit{causal spaces} introduced in Ref. \cite{causalspaces}, but we also provide corresponding differential-geometric transverse notions, in order to further illustrate how they work nicely in tandem with the intrinsically defined structures on the leaf space. The development here is intended to have interesting features both for Lorentzian geometers and foliation theorists.

Finally, in the final Section~\ref{mainsec} we discuss and prove the transverse timelike diameter theorem and its corollary for Lorentzian orbifolds stated above, as well as Theorem \ref{theointrodiameter}. 

\section{Preliminaries}\label{section:preliminaries}

Throughout this paper $M$ always denotes a connected (real, Hausdorff second-countable) $n$-dimensional manifold with $n\geq 2$ and we assume for simplicity that $M$ and all geometric objects defined thereon and elsewhere are smooth (i.e., of class $\mathcal{C}^\infty$) unless otherwise stated. 

\subsection{Metrics on vector bundles}\label{subsection: bundles}
We begin by recalling a few elementary structures on vector bundles, largely to establish terminology. Let $\pi_E\colon E\rightarrow M$ be a rank $k$ smooth (real) vector bundle over $M$ and let $s\in \mathbb{N}$. We denote as $\Gamma(E)$ the $C^\infty(M)$-module of smooth sections of the vector bundle $E$, but adopt the standard special notation $\mathfrak{X}(M):= \Gamma(TM)$ for the tangent bundle $TM$.

%By a $(0,s)$-type {\em (covariant) tensor field on $E$} we mean a map $T$ that assigns to each $p \in M$ an $\mathbb{R}$-multilinear map $T_x: E_x^s \rightarrow \mathbb{R}$ on the fiber $E_x := \pi_E^{-1}(x)$ over $x$. As usual, we say that such a $(0,s)$-type tensor field $T$ is {\em smooth} if for any local frame of smooth Sections $s_1, \ldots, s_k: U\subset M \rightarrow E$ defined on an open set $U$ the real-valued functions $T(s_{i_1}, \ldots, s_{i_s})$ are all smooth. Again, when we refer to tensor fields on a given vector bundle $E$ below we shall assume they all are smooth unless otherwise stated. 

%Clearly, any smooth $(0,s)$-type tensor field $T$ can be equivalently defined as a $C^\infty(M)$-multilinear map $T:\Gamma(E)^s \rightarrow C^\infty(M)$, and the set of all tensor fields of a fixed type has a natural, pointwisely defined structure of a $C^\infty(M)$-module. 

Now, let $\mu \in \{0,1,\ldots,k\}$. A {\em semi-Riemannian metric (tensor) of index $\mu$} on $E$ is a $(0,2)$-type tensor field $\sigma$ that associates to each $x \in M$ a symmetric nondegenerate bilinear form $\sigma_x\colon E_x\times E_x \rightarrow \mathbb{R}$ of index $\mu$ on $E_x$. If $\mu=0$ [resp. $\mu=1$ with $k\geq 2$], then the semi-Riemannian metric $\sigma$ is said to be {\it Riemannian} [resp. {\it Lorentzian}]. Of course, the usual definition of semi-Riemannian [resp. Riemannian, Lorentzian] metric {\it on the manifold $M$} is recovered by just taking $E=TM$.

\subsection{Basics of foliation theory}\label{subsection: foliations}

Given $p\in\mathbb{Z}$ with $0\leq p \leq n$, a \textit{(regular\footnote{Here, ``regular'' refers to the fact that all leaves have the same dimension, as opposed to the so-called {\em singular} foliations, where this requirement is dropped. We deal exclusively with regular foliations in this paper.}) $p$-dimensional foliation} on (or of) $M$ is a partition $\f$ of $M$ into $p$-dimensional, connected, immersed submanifolds, the \textit{leaves} (of $\f$), such that, for each leaf $L\in\f$ and each $x\in L$, there are smooth vector fields $X_1,\ldots,X_p\in\mathfrak{X}(M)$ everywhere tangent to the leaves whose values at $x$ form a basis for $T_xL$. The smooth distribution given by the tangent spaces of the leaves is denoted by $T\f$, and the leaf containing $x$ by $L_x$. We say that a vector field $V\in \mathfrak{X}(U)$ is \textit{tangent} (to $\f$) or \textit{vertical} (with respect to $\f$) if $V_x \in T_xL_x=:T_x\f$ for all $x\in U$. We denote by $\mathfrak{X}(\f|_{U})$ the $C^\infty(U)$-submodule of $\mathfrak{X}(U)$ of vertical vector fields, and $\mathfrak{X}(\f):=\mathfrak{X}(\f|_{M})$. The number $q=n-p$ is called the \textit{codimension} of $\f$. Following standard usage it will be referred to more often than the dimension of the leaves when discussing foliations.

There are several alternative albeit equivalent definitions of foliations (see for instance \cite[Section 1.2]{mrcun}), among which the following one will be most convenient for us here. For a fixed $q\in \{1,\ldots,n-1\}$, consider an open cover $\{U_i\}_{i\in I}$ of $M$, together with submersions $\pi_i\colon U_i\to S_i:= \pi_i(U_i)\subset\mathbb{R}^q$, and diffeomorphisms $\gamma_{ij}\colon\pi_j(U_i\cap U_j)\to\pi_i(U_i\cap U_j)$ satisfying
$$\gamma_{ij}\circ\pi_j|_{U_i\cap U_j}=\pi_i|_{U_i\cap U_j}$$
for all $i,j\in I$. The connected components of the fibers $\pi_i^{-1}(\overline{x})$ are called \textit{plaques}. It is not difficult, using the local form of submersions, to verify that the plaques glue together to form a partition of $M$ by connected immersed submanifolds of dimension $p=n-q$, thus recovering the original definition of a (codimension $q$) foliation $\f$ as described above. The collection $\Gamma=(U_i,\pi_i,\gamma_{ij})_{i,j\in I}$ is called a \textit{(Haefliger) cocycle} for the foliation $\f$. In this context, we will call each domain $U_i$ a \textit{simple open set}\footnote{Although that term has other meanings in the literature, we shall use it here exclusively to refer to the domains of local submersions as described. This usage is fairly common among foliation theorists: see, e.g., \cite[Introduction]{molino} \cite[Section~5.2]{alex}.} for $\f$, while the submersion $\pi_i\colon U_i \to \mathbb{R}^q$ is said to {\it locally define} $\f$. It is often useful to identify the images $S_i$ of these local submersions with submanifolds of $M$ transverse to the leaves, its pieces given by images of sections for the submersions $\pi_i$. In that case, they are referred to as \textit{local transversals} for $\f$. The disjoint union $S_\Gamma:=\bigsqcup_i S_i$ is called a \textit{total transversal} for $\f$. 

Let us recall a few elementary examples of foliations. 
\begin{example}[Submersions $\&$ Bundles]\label{exe: simple submersions and bundles}
   The simplest kind of foliation arises from submersions. Specifically, let $\pi\colon M^n\rightarrow B^q$ be any submersion with $n>q$. Then, the collection of all the connected components of the fibers $\pi^{-1}(y)$ for $y\in \pi(M)$ define the leaves of a codimension $q$ foliation of $M$. $M$ itself is a simple open set in the language introduced above. The leaves are of course not diffeomorphic to each other in general, but that happens in the important particular case when for some connected manifold $F$, $(M,B,\pi,F)$ is a fiber bundle with total space $M$, base $B$, bundle projection $\pi$ and fiber $F$. In this case $\pi$ is an onto submersion, and $\forall y \in B$ the fiber $\pi^{-1}(y)$ if diffeomorphic to $F$. 
\end{example}
%\begin{example}[Reeb foliation]\label{exe: reeb}
%\textcolor{red}{COLOCAR AQUI A FOLHEAÇÃO DE REEB COMO EXEMPLO DE UMA QUE NÃO VEM DE SUBMERSÃO.} \textcolor{orange}{alternativamente, para economizar, podemos citar o Exemplo \ref{example:kronecker} ilustrar isso, ao invés de introduzir Reeb}
%\end{example}
\begin{example}[Homogeneous foliations]\label{exe: foliated actions}
Let $\mu\colon G\times M\to M$ be a smooth left action of a Lie group $G$ on $M$. Given any $x\in M$, let $\mu_x\colon  g\in G \mapsto \mu(g,x) \in M$ be the associated \textit{orbit map}, whose image is the orbit $Gx =\mu_x(G)$. The stabilizer group of $x$ is then $G_x:=\mu_x^{-1}(x)$. As it is well-known, $G_x$ is a closed Lie subgroup of $G$, and the map $G/G_x\to M$ induced by the orbit map is an injective immersion whose image is precisely $Gx$ (see, for instance, \cite[Proposition 3.14]{alex}). In other words, each orbit $Gx$ is an immersed submanifold of $M$ of dimension $\text{dim}(G)-\text{dim}(G_x)$. Hence, if  $\dim(G_x)$ is a constant function (of $x$), then the connected components of the orbits partition $M$ into immersed submanifolds all of the same dimension, defining a foliation $\f_G$. To see this, let $E\subset T_eG$ (with $e$ denoting the identity of $G$) be any vector subspace such that $E\oplus T_eG_x = T_e G$. Then recall that for each $v\in E$, one has an associated fundamental vector field $V^\#\in\mathfrak{X}(M)$ induced by the action. Since $T_x(Gx)=\dif (\mu_x)_{e}(E)$, it follows that the $V^\#$  generate $T\f_G$. Entirely analogous considerations apply to smooth right actions of $G$ on $M$. A foliation $\f_G$ given by any (left or right) action of a Lie group $G$ with these properties is usually called \textit{homogeneous}.
\end{example}

\begin{example}[Integrable distributions] 
Perhaps the best known way in which foliations arise is via integrable distributions. Recall that a {\it $p$-dimensional distribution} on $M$ for $p \in \{1,\ldots, n-1\}$ is a rank-$p$ vector subbundle $D$ of the tangent bundle $TM$. $D$ is said to be {\it involutive} if the $C^{\infty}(M)$-submodule $\Gamma(D)$ of its sections is also a Lie subalgebra of $\mathfrak{X}(M)$. The (global) Frobenius' theorem \cite[Theorem 19.12]{Lee} then establishes that a $p$-dimensional distribution $D\subset TM$ is involutive if and only if there exists a (unique) $p$-dimensional foliation $\f$ of $M$ such that $T\f =D$. In this context, the leaf $L_x$ at $x \in M$ is said to be the {\it (maximal) integral submanifold} of $D$ through $x$ and $D$ is said to be {\it integrable}. In particular, any 1-dimensional distribution on $M$ (provided it exists) is integrable.
\end{example}

A useful source of elaborate examples is the following standard construction \cite{haefliger4} (see also \cite[Ch. 5]{camacho}). 

\begin{example}[Suspensions]\label{example: suspensions}
Fix connected manifolds $B^p$ and $S^q$, a base point $x_0 \in B$, and let $h\colon \pi_1(B,x_0)\to\mathrm{Diff}(S)$ be a group homomorphism. Denote by $\rho\colon \hat{B}\to B$ the canonical projection of the universal covering of $B$, and consider $\tilde{M}:=\hat{B}\times S$. The fibers of the projection onto the second factor determine [conf. Example \ref{exe: simple submersions and bundles}] a codimension $q$ foliation $\tilde{\f}$ of $\tilde{M}$. Let $\pi_1(B,x_0)$ act (on the right) on $\tilde{M}$ by 
$$(\hat{b},s)\!\cdot\![\gamma]=\left(\hat{b}\!\cdot\![\gamma],h\left([\gamma]^{-1}\right)(s)\right),$$
where $\hat{b}\!\cdot\![\gamma]$ is the image of $\hat{b}$ under the deck transformation determined by $[\gamma]\in \pi_1(B,x_0)$. This action is free and properly discontinuous; hence $M=\tilde{M}/\pi_1(B,x_0)$ is an $n$-manifold (with $n=p+q$) and the orbit projection $\pi\colon \tilde{M}\to M$ is a smooth covering map \cite[p. 28]{molino}. Moreover, one verifies that $\tau\colon M\to B$ given by $\tau(\pi(\tilde{b},t))=\rho(\tilde{b})$, is the projection of a fiber bundle with total space $M$, base $B$, fiber $S$ and structural group $h(\pi_1(B,x_0))$. The foliation $\tilde{\f}$ is invariant under the $\pi_1(B,x_0)$-action (i.e., the action of each element sends leaves to leaves), hence projects through $\pi$ to a foliation $\f$ of $M$, with $\codim(\f)=q=\dim(S)$. The foliation $\f$ is said to be given by the \textit{suspension of the homomorphism} $h$. An important particular use of suspensions is when we fix a diffeomorphism $\phi\in \mathrm{Diff}(S)$, and take $B=\mathbb{S}^1$ and $h(m) = \phi^m$ for all $m\in \mathbb{Z}\simeq \pi_1(\mathbb{S}^1)$.
\end{example}

\begin{example} [Pullbacks]\label{exe: pullbacks}
    Foliations can be pulled back by smooth maps, provided they are transverse to the leaves. More precisely, suppose $\f$ is a foliation of $M$ and let $f\colon N\to M$ be a smooth map which is transverse to each leaf of $\f$, i.e.
    $$df_x(T_xN) + T_{f(x)}\f = T_{f(x)}M, \quad \forall x\in N. $$ Then, if $(U_i,\pi_i,\gamma_{ij})$ is a cocycle for $\f$, one obtains a cocycle $(V_i,\pi_i', \gamma_{ij})$, where $V_i=f^{-1}(U_i)$ and $\pi_i'=\pi_i\circ f|_{V_i}$, defining a foliation $f^*(\f)$ on $N$. It is clear that $Tf^*(\f)=(\dif f)^{-1}(T\f)$ and $\codim(f^*(\f))=\codim(\f)$.
\end{example}

For the rest of this subsection, we fix a codimension $q$ foliation $\mathcal{F}$ on $M$, and sometimes we shall abbreviate these data via the pair $(M,\f)$ (called a {\it foliated manifold}). We can naturally define an equivalence relation on $M$ whereby two points are equivalent if and only if they are on the same leaf of $\f$. The quotient of $M$ by this equivalence relation is denoted by $M/\f$ and called {\it leaf space}. \textit{Here and hereafter the leaf space is always understood to be endowed with the quotient topology}, although we may on occasion repeat this for emphasis. Leaf spaces may be rather unwieldy from a differential-geometric standpoint, as the next quintessential example shows.

\begin{example}\label{example:kronecker}
Consider the flat torus $\mathbb{T}^2=\mathbb{R}^2/\mathbb{Z}^2$. For any $(x,y)\in \mathbb{R}^2$ we shall denote its correspondent element in the quotient by $[x,y] \in \mathbb{R}^2/\mathbb{Z}^2$. For each $\lambda\in(0,+\infty)$, consider the smooth $\mathbb{R}$-action
$$
\begin{array}{rcl}
\mathbb{R}\times\mathbb{T}^2& \longrightarrow &\mathbb{T}^2\\
(t,[x,y]) & \longmapsto & [x+t,y+\lambda t]
\end{array}.$$
It is clear that for each $[x,y]\in\mathbb{T}^2$, its $\mathbb{R}$-orbit is an immersed one dimensional submanifold. The corresponding homogeneous foliation $\f(\lambda)$ is usually called a \textit{$\lambda$-Kronecker foliation}. It is not difficult to verify that when $\lambda$ is irrational, each leaf is dense in $\mathbb{T}^2$. The quotient topology on the leaf space, hence, is the indiscrete topology $\{\emptyset,\mathbb{T}^2/\f(\lambda)\}$. However, when $\lambda\in\mathbb{Q}$ the structure of the leaf space is completely different: all leaves are closed and the leaf space is naturally a manifold, diffeomorphic to $\mathbb{S}^1$.
\end{example}

An important ingredient in foliation theory is to try and understand the structure of the leaf space in a given situation. A very helpful fact is the following \cite[Theorem III.1]{camacho}. 
\begin{proposition}\label{prop: proj is open}
  For any foliation $\f$ on $M$, the standard projection $\pi_{\f}\colon M\rightarrow M/\f$ is an open map.  
\end{proposition}
We also have good information about the leaf space of a submersion, as the next result shows. (Although we believe it is quite well-known, we could not find a detailed proof anywhere, and thus we include it here.)
\begin{proposition}\label{prop: submersion leaf space}
 Let $\pi\colon M\rightarrow B$ be an onto submersion. The following statements hold for the foliation $\f$ whose leaves are the connected components of the fibers $\pi^{-1}(b)$ for each $b\in B$.
 \begin{itemize}
     \item[1)] The uniquely defined map $\overline{\pi}\colon  M/\f \rightarrow B$ that makes the diagram 
     $$
     \xymatrix{
 & M  \ar[dr]^-{\pi} \ar[dl]_{\pi_{\f}} & \\
M/\f \ar[rr]_{\overline{\pi}} & & B}
$$
commutative is a homeomorphism if and only if each fiber $\pi^{-1}(b)$ is connected. 
     \item[2)] The leaf space $M/\f$ is a $T_1$ topological space. 
 \end{itemize}
\end{proposition}
\begin{proof}  
 $(1)$\\
 Observe first that $\overline{\pi}\colon  M/\f \rightarrow B$ is well-defined and continuous by elementary properties of the quotient map $\pi_{\f}\colon  M \rightarrow M/\f$. Because $\pi$ is onto, then so is $\overline{\pi}$. Also observe that given any open set $\widehat{O}\subset M/\f$ we have 
 $$\overline{\pi}(\widehat{O}) = \pi(\pi^{-1}_\f(\widehat{O})).$$
 Since any submersion is in particular an open map this shows that $\overline{\pi}$ is an open map. Therefore, $\overline{\pi}$ is a homeomorphism if and only if it is one-to-one. However, since $\forall x,y\in M$ we have 
 $$\overline{\pi}(\pi_\f(x)) =\overline{\pi}(\pi_\f(y))\Longleftrightarrow \pi(x)=\pi(y),$$
 this implies that $\overline{\pi}$ is injective if and only if to be on the same fiber is the same as being on the same leaf, i.e., if and only if each fiber is connected. \\
 $(2)$\\
 Let $L\neq L'\in \f$ be any two distinct leaves. If $b:=\pi(L) \neq \pi(L')=:b'$ in $B$, then since $B$ is a manifold, and in particular Hausdorff, we can pick disjoint open sets $U\ni b$ and $U'\ni b'$; hence, $\pi^{-1}(U)$ and $\pi^{-1}(U')$ are two disjoint open sets in $M$ containing $L$ and $L'$ respectively, and project via $\pi_\f$ [conf.~\ref{prop: proj is open}] onto disjoint open sets in $M/\f$ containing $L$ and $L'$ (viewed now as elements in $M/\f$). In other words, $L$ and $L'$ are Hausfdorff-separated in~$M/\f$. 

 If we now have $b=b'$, then $L,L'$ are two distinct connected components of the fiber $\pi^{-1}(b)$; since the latter set is closed in $M$ and $L,L'$ are closed in $\pi^{-1}(b)$, they are also closed subsets of $M$. Pick any open set $U\ni b$ in $B$. Then $\pi^{-1}(U)\setminus L'$ is an open set in $M$ which thus projects via $\pi_\f$ onto an open set in $M/\f$ containing $L$ but not containing $L'$. 
\end{proof}
We end this subsection by pointing out that a foliation produces a new topology on $M$ which will play a useful role later.
\begin{definition}
Let $(M,\mathcal{F})$ be a foliated manifold and $A\subset M$. We say that $A$ is \textit{satured with respect to $\mathcal{F}$} if $A$ is either empty or a union of leaves of $\mathcal{F}$, or, equivalently, if $A=\pi^{-1}_\f(\pi_\f(A))$.
\end{definition}
More generally, for any $A\subset M$, the \textit{saturation of $A$ with respect to $\mathcal{F}$} is the set $\pi^{-1}_\f(\pi_\f(A))\subset M$, which is just the union of all leaves that contain points of $A$.

If $\tau$ denotes the original (manifold) topology of $M$, then $$\tau_\f:=\{U\in\tau\, :\, U \text{ is saturated with respect to } \f\}$$ is also a topology on $M$, with countable basis, coarser than $\tau$ and always non-Hausdorff, except in the case that $\f$ is the trivial foliation by points of~$M$.

In the next two subsections we summarize a few standard notions used in understanding certain properties of leaf spaces.

\subsection{Holonomy}\label{section: holonomy}
Let $(M,\f)$ be a codimension $q$ foliated manifold, represented by the Haefliger cocycle $\Gamma=(U_i,\pi_i,\gamma_{ij})_{i,j\in I}$, and let $x,y\in M$ be in the same leaf $L\in \f$. Given a path $c\colon [0,1]\to L$ joining $x$ to $y$, let $0=t_1<\dots<t_{m+1}=1$ be a partition of $[0,1]$ such that $c([t_k,t_{k+1}])\subset U_{i_k}$ for some simple open set $U_{i_k}$, $k=1,\ldots,m$. Then, restricting the pertinent transition maps $\gamma_{ij}$ if needed, there is a diffeomorphism
$$\gamma_{i_mi_{(m-1)}}\circ\gamma_{i_{(m-1)}i_{(m-2)}}\circ\dots\circ\gamma_{i_2i_1}=\gamma_{i_mi_1}$$
between small enough neighborhoods of $\overline{x}=\pi_1(x)$ and $\overline{y}=\pi_m(y)$. If we think of this as a local diffeomorphism on the total transversal $S_\Gamma = \bigsqcup_{i\in I}S_i$ for $\f$ containing $\overline{x}$ and $\overline{y}$, this becomes the ``sliding along the leaves'' notion from \cite[Section 1.7]{molino}. Let us denote the germ of $\gamma_{i_mi_1}$ at $\overline{x}$ by $h_c$. This germ depends only on the $\partial[0,1]$-relative homotopy class of $c$. When $x=y$, the \textit{holonomy group} of $L$ at $x$, given by
$$\mathrm{Hol}_x(L)=\{h_c \,: \, c\colon [0,1]\to L\ \mbox{\ is a loop}\},$$
then depends only on the fundamental group $\pi_1(L,x)$ of $L$ at the base point $x$ and codifies the recurrence of leaves near $L$ around $x$. Changing the base point $x$ on the same leaf gives the ``same'' group up to conjugacy. Identifying these groups up to conjugacy  along the points of $L$ we often drop the suffix $x$ and speak of \textit{the} group $\mathrm{Hol}(L)$ as the holonomy group of the leaf without specifying a base point. The leaf $L\in\f$ is said to be \textit{without holonomy} [res. \textit{with finite holonomy}] when $\mathrm{Hol}_x(L)$ is trivial [resp. finite]. In particular, any simply-connected leaf [resp. leaf with finite homotopy group] has trivial [resp. finite] holonomy. 

\subsection{Orbifolds}\label{subsection: orbifolds}

Orbifolds are a generalization of manifolds that allow certain types of singularities: instead of being locally Euclidean, orbifolds are locally modeled by quotients of the Euclidean space by actions of finite groups. We describe them briefly here insofar as they relate to foliations, and refer to \cite{adem}, \cite{caramello3} and \cite{mrcun} for more detailed introductions.

Let $X$ be a topological space. An \textit{orbifold chart} of dimension $n\in\mathbb{N}$ on $X$ consists of a connected open subset $\widetilde{U}\subset\mathbb{R}^n$, a finite group $H$ acting smoothly and effectively on $\widetilde{U}$, and a continuous, $H$-invariant map $\phi\colon \widetilde{U}\to X$ inducing a homeomorphism $\widetilde{U}/H\cong U$ where $U$ is an open subset of $X$. We will denote such a chart by $(\widetilde{U},H,\phi)$.

Two charts $(\widetilde{U}_i,H_i,\phi_i)\in \mathcal{A}$ ($=1,2$) are compatible if for each given $x\in U_1\cap U_2$, there is an open subset $U_3\subset U_1\cap U_2$ containing $x$ charted by $(\widetilde{U}_3,H_3,\phi_3)$ and embeddings $\lambda_i\colon \widetilde{U}_3\hookrightarrow\widetilde{U}_i$ such that $\phi_i\circ\lambda=\phi_3$. Then one defines an orbifold atlas for $X$ as a cover by compatible $n$-dimensional charts, in the obvious sense, and an \textit{$n$-dimensional smooth orbifold} as a Hausdorff, second countable topological space $\mathcal{O}$ endowed with a smooth maximal orbifold atlas.

\begin{example}[Manifolds are orbifolds] Manifolds can be regarded as orbifolds (of the same dimension) for which all the finite groups $H_i$ appearing in local charts are trivial.
\end{example}

\begin{example}[Quotients by locally free actions]
Let $G$ by a Lie group acting properly on a manifold $M$. If the action is \textit{locally free} (i.e., if the induced infinitesimal action of the Lie algebra $\mathfrak{g}$ is free or, equivalently, if each stabilizer $G_x$ is finite), then $M/G$ is naturally an orbifold, as one easily establishes via an application of the tubular neighborhood theorem \cite[Theorem 3.57]{alex}. Indeed, these \textit{quotient orbifolds} in fact exhaust all possible orbifolds, as we shall see in Proposition \ref{proposition: every orbifold is an isometric quotient} below.
\end{example}

\begin{example}[Orbifolds arising as leaf spaces]\label{example: when leaf spaces are orbifolds} Let $\f$ be a smooth foliation of codimension $q$ on a manifold $M$ and suppose all leaves of $\f$ are compact and have finite holonomy. It is clear that $M/\f$ is second countable (recall Prop. \ref{prop: proj is open}). Moreover, the compactness of the leaves guarantee that $M/\f$ is Hausdorff, since then each leaf admits tubular neighborhoods. Given $[x]\in M/\f$, by the so-called generalized Reeb stability (cf. e.g., \cite[Theorem 2.4.3 and Theorem 3.1.5]{candel}), there is a a tubular neighborhood $\mathrm{Tub}(L_x)$ restricted to which $\f$ is given by the suspension of the holonomy map $h\colon \pi_1(L,x)\to \mathrm{Hol}_x(L)<\mathrm{Diff}(S)$ on a small transversal $S\ni x$. We can, without loss of generality, identify $S$ with an open ball in $\mathbb{R}^q$. Then one easily verifies that $(S,\mathrm{Hol}_x(L),\pi_\f|_{\mathrm{Tub}(L_x)} )$ is a $q$-dimensional orbifold chart for the neighborhood $\pi_\f(\mathrm{Tub}(L_x))\ni [x]$, and that these charts in fact define a $q$-dimensional orbifold structure for $M/\f$.
\end{example}

Many of the elementary objects and constructions from the differential topology and geometry of manifolds generalize without difficulty to orbifolds, as can be seen in \cite{kleiner}, \cite{chenruan} and \cite{galazgarcia}. For instance, one constructs the (analog of the) tangent bundle of $\mathcal{O}$ by taking quotients of the tangent bundles $TU_i$ of the local charts by the corresponding linearized actions of the chart groups $H_i$ and showing that these local quotients glue together to form an orbifold $T\mathcal{O}$. Tensor bundles of every rank are constructed similarly. Their sections --- tensor fields on $\mathcal{O}$ --- correspond locally, in terms of a chart $(\widetilde{U},H,\phi)$, to $H$-invariant tensor fields on $\widetilde{U}$. A semi-Riemannian metric of index $\mu$ on the orbifold $\mathcal{O}$ is then just a field $\mathrm{g}$ of symmetric, non-degenerate covariant $2$-tensors of index $\mu$ on $\mathcal{O}$. As usual, such a semi-Riemannian metric is \textit{Riemannian} when $\mu =0$ and \textit{Lorentzian} when $\mu =1$ and $\dim\mathcal{O}\geq2$. In particular, many results from Riemannian geometry also have a generalization for orbifolds, see for instance \cite[Chapter 4]{caramello3}. For example, one proves via partitions of unity that \textit{any orbifold admits a Riemannian metric \cite[Proposition 2.20]{mrcun}}. To our knowledge far fewer results are known for more general indices, including the Lorentzian case.

The \textit{orthonormal frame bundle} $F\mathcal{O}$ of a Riemannian orbifold $\mathcal{O}$ is constructed by the same procedure: by gluing together the quotients of the orthonormal frame bundles of the charts by the (extended) actions of the chart groups. However, now such an extended action is actually free, since an isometry fixing a frame on a connected manifold must be the identity. Therefore $F\mathcal{O}$ is a bona fide manifold. It inherits a natural $\mathrm{O}(q)$-action which is proper, effective and locally free. As in the case of manifolds, by taking the quotient $F\mathcal{O}/\mathrm{O}(q)$ one recovers $\mathcal{O}$ up to diffeomorphism. 

In the case when $\mathcal{O}$ is orientable, it can further be obtained as the quotient of the bundle of \textit{oriented} orthonormal frames by the action of $\mathrm{SO}(q)$. This is desirable since $\mathrm{SO}(q)$ is connected and, in particular, its orbits form a foliation. When $\mathcal{O}$ is not orientable one can complexify $F\mathcal{O}$ in order to obtain $\mathcal{O}$ as the orbit space of the corresponding group $\mathrm{U}(q)$ (and hence again as a leaf space). Summing up, in this way one obtains a converse to Example \ref{example: when leaf spaces are orbifolds}, showing in particular that \textit{every orbifold can be realized as a leaf space for a foliation with compact leaves}:

\begin{proposition}[{\cite[Proposition 2.23]{mrcun}}]\label{proposition: every orbifold is an isometric quotient}
Every orbifold $\mathcal{O}$ can be realized as the orbit space of a locally free action of a compact Lie group $G$. Moreover, the group $G$ can be assumed to also be connected. Therefore, any orbifold $\mathcal{O}$ can be viewed as the leaf space of a homogeneous foliation with compact leaves.
\end{proposition}

\subsection{Transverse structures} \label{section:transverse 1}

The {\it normal bundle} $\nu(\mathcal{F})$ of a codimension $q$ foliation $\mathcal{F}$ on $M$ is the rank $q$ vector bundle over $M$ whose fiber at each $x\in M$ is the quotient space $\nu(\mathcal{F})_x:=T_xM/T_x\mathcal{F}$. Given any $v\in T_xM$ we denote by $\overline{v}$ its equivalence class in $\nu(\mathcal{F})_x$. Similarly, for any $X\in \mathfrak{X}(M)$, $\overline{X} \in \Gamma(\nu(\mathcal{F}))$ if the section which takes the value $\overline{X_x}$ at each $x\in M$. One easily sees via a partition of unity argument that the $C^\infty(M)$-linear map 
\begin{equation}\label{thephi}
\Phi\colon  X\in \mathfrak{X}(M)\mapsto \overline{X} \in \Gamma(\nu(\mathcal{F}))
\end{equation}
is surjective. Thus the equation 
\begin{equation}\label{tangentaction}
    \mathcal{L}_V\overline{X} := \overline{[V,X]}, \quad \forall V\in \mathfrak{X}(\mathcal{F}), \forall X\in \mathfrak{X}(M)
\end{equation}
gives a well-defined ``infinitesimal action'' of tangent vector fields on the sections of the normal bundle. Indeed, the correspondence 
$$(V,\chi) \in \mathfrak{X}(\mathcal{F})\times \Gamma(\nu(\mathcal{F})) \mapsto \mathcal{L}_V\chi \in \Gamma(\nu(\mathcal{F}))$$
defines a kind of ``partial connection'' on $\nu\mathcal{F}$ defined for tangent vector fields, called a {\it Bott connection}. 

The elements of the set 
$$\mathfrak{l}(\mathcal{F}) :=\{\chi \in \Gamma(\nu(\mathcal{F}))\, :\, \mathcal{L}_V \chi =0, \forall V\in \mathfrak{X}(\mathcal{F})\}$$
are called {\it transverse vector fields}. The set $\mathfrak{l}(\mathcal{F})$ has an obvious structure of real vector space, but it is {\it not} a $C^\infty(M)$-submodule of $\Gamma(\nu(\mathcal{F}))$. However, it is naturally a module over the ring $\Omega^0(\f)$ of {\it $\mathcal{F}$-basic functions}, i.e., those $f\in C^\infty(M)$ such that $Vf=0$ for all $V\in \mathfrak{X}(\mathcal{F})$. (That is, a smooth real-valued function on $M$ is basic if and only if it is constant on the leaves of $\mathcal{F}$.) $\mathfrak{l}(\mathcal{F})$ is also naturally a Lie algebra with bracket given by
$$[\overline{X},\overline{Y}] := \overline{[X,Y]}, \quad X,Y \in \mathfrak{X}(M).$$

The Lie subalgebra of the {\it foliate} (or {\it projectable}) vector fields is 
$$\mathfrak{L}(\mathcal{F}):= \{X\in \mathfrak{X}(M)\, :\, [X,\mathfrak{X}(\mathcal{F})]\subset \mathfrak{X}(\mathcal{F})\}.$$
It is clear that $\mathfrak{L}(\mathcal{F})$ is also a $\Omega^0(\f)$-module and indeed that the map $\Phi$ in \eqref{thephi} restricts to an isomorphism of $\Omega^0(\f)$-modules and also a Lie algebra isomorphism between $\mathfrak{L}(\mathcal{F})$ and $\mathfrak{l}(\mathcal{F})$.

More generally, let $T\colon  \Gamma(\nu(\mathcal{F}))^s \rightarrow C^\infty(M)$ be a $(0,s)$-type tensor field on the normal bundle. Given $V\in \mathfrak{X}(\mathcal{F})$, we define the action of the Bott connection via the product rule 
\begin{equation}\label{productrule2}
    (\mathcal{L}_VT)(S_1,\ldots,S_s) = V(T(S_1,\ldots,S_s)) - \sum_{i=1}^sT(S_1,\ldots, \mathcal{L}_VS_i,\ldots, S_s), 
\end{equation}
$\forall S_1, \ldots, S_k\in \Gamma(\nu(\mathcal{F}))$. We then say that $T$ is {\it holonomy-invariant} if 
$$\mathcal{L}_VT \equiv 0, \quad \forall V \in \mathfrak{X}(\mathcal{F}).$$
The set $\mathfrak{T}(s,\mathcal{F})$ of all holonomy-invariant $(0,s)$-type tensor fields on $ \Gamma(\nu(\mathcal{F}))$ has the natural structure of a $\Omega^0(\f)$-module. 

A $(0,s)$-type tensor field $\xi$ on $M$ is {\it $\mathcal{F}$-basic} if $\xi(X_1, \ldots, X_s) =0$ whenever some $X_i\in \mathfrak{X}(\mathcal{F})$, and $\mathcal{L}_V\xi =0$ for all $V\in \mathfrak{X}(\mathcal{F})$. We denote the $\Omega^0(\f)$-module of $\mathcal{F}$-basic $(0,s)$-type tensor fields on $M$ by $\mathfrak{L}(s,\mathcal{F})$. It is then easy to show:
\begin{proposition}\label{ prop: fundamental iso}
    For each $s\in \mathbb{N}$, the map 
    $$\Phi_s\colon  \xi \in \mathfrak{L}(s,\mathcal{F}) \mapsto \overline{\xi} \in \mathfrak{T}(s,\mathcal{F})$$
    given, for each $\xi \in \mathfrak{L}(s,\mathcal{F})$ by 
    $$\overline{\xi}(\overline{X_1},\ldots, \overline{X_s}):= \xi(X_1,\ldots, X_s), \quad \forall X_i \in \mathfrak{X}(M),$$
    is an isomorphism of $\Omega^0(\f)$-modules.
\end{proposition}
%\qcd

The geometric meaning of being $\mathcal{F}$-basic/projectable is elucidated via the following result, which also vindicates a general fact in foliation theory that underscores the philosophy of this paper: {\it the study of the transverse geometry of $\f$ is a way to ``regularize'' (or ``desingularize'') the geometry of the leaf space}.

\begin{proposition}\label{prop:basic 1}
For the codimension-$q$ foliation $\mathcal{F}$ the following hold.
\begin{itemize}
    \item[i)] Let $\xi$ be $(0,s)$-type tensor field on $M$, $f\in C^\infty(M)$, and $X\in \mathfrak{X}(M)$. $\xi$ is $\mathcal{F}$-basic [resp. $f\in \Omega^0(\f)$ is $\mathcal{F}$-basic; $X$ is projectable] if and only if for any $\pi\colon U\subset M \rightarrow \mathbb{R}^q$ submersion locally defining $\mathcal{F}$ there exists some $(0,s)$-type tensor $\widehat{\xi}$ on $\pi(U)$ [resp. $\hat{f}\in C^\infty(\pi(U))$; $\widehat{X} \in \mathfrak{X}(\pi(U))$] such that $\pi^{\ast}\widehat{\xi} = \xi|_{U}$ [resp. $\hat{f}\circ \pi = f|_{U}$, $X|_{U}$ is $\pi$-related to $\widehat{X}$].
 \item[ii)] Let $(U_i,\pi_i,\gamma_{ij})_{i\in I}$ be a Haefliger cocycle for $\mathcal{F}$, and suppose there exist $(0,s)$-tensor fields $\xi_i$ [resp. smooth real-valued functions $f_i$, vector fields $X_i$] defined on $\pi_i(U_i)$ such that for any $i,j\in I$ for which $U_i\cap U_j \neq \emptyset$ we have 
$$\gamma_{ij}^{\ast }\xi_i = \xi_j$$
[resp. $f_i\circ \gamma_{ij} = f_j$, $X_j \stackrel{\gamma_{ij}}{\sim}X_i$].
Then there exists an $\mathcal{F}$-basic $(0,s)$-type tensor field $\xi$ on $M$ [resp. $f\in \Omega^0(\f)$; projectable $X\in \mathfrak{L}(\mathcal{F})$] such that $\pi_i^{\ast}\xi_i = \xi|_{U_i}$ [resp. $f_i\circ \pi_i = f|_{U_i}$, $X|_{U}$ if $\pi_i$-related to $X_i$] for every $i\in I$.
\end{itemize}
\end{proposition}
\begin{proof}
   See \cite[Propositions 2.1, 2.2, 2.3]{molino}
\end{proof}
In particular, applying the local form of submersions in the item $(i)$ of the previous proposition we obtain the following result, which is quite useful in practical computations. 
\begin{corollary}\label{cor: local projectable}
Given a codimension $q$ foliated manifold $(M,\f)$ and any $x\in M$ there exist a open set $U \ni x$ of $M$ and a (linearly independent) frame of local vector fields $$V_1,\ldots, V_{n-q}, X_1,\ldots, X_q \in \mathfrak{X}(U)$$ such that each $V_i \in \mathfrak{X}(\f|_U)$ for all $i\in \{1,\ldots, n-q\}$, and each $X_j $ is both projectable and everywhere transverse to $\f$ in $U$.
\end{corollary}
%\qcd

%\subsection{Molino's theory}\textcolor{teal}{...the leaf space of every Riemanniann foliation with closed leaves has Hausdorff topology...Unlike the Riemannian case, from the fact that the isometry group $O(n, n-\nu)$ is not compact for $\nu>0$, a semi-Riemannian foliation is allowed to have a non-Hausdorff leaf space, even if all of its leaves are closed, as is shown by the next example. }

\section{Semi-Riemannian foliations}\label{section: SR foliations}

\textit{Riemanniann foliations} have had a long history since their introduction by Reinhart \cite{reinhart} and further development especially by P. Molino \cite{molino2,molino3,molino} and A. Haefliger \cite{haefliger6,haefliger4}. Accordingly, the subject has accumulated a vast specialized literature (see, e.g., \cite{alex4,mrcun,molino,tondeur} for comprehensive treatments of the basic results and original references). The main objects studied in this paper, \textit{Lorentzian} foliations, have been to the best of our knowledge far less studied \cite{boubel,zhukova}. However, both Riemannian and Lorentzian foliations are naturally encompassed in a natural broader generalization, the so-called {\it semi-Riemannian foliations} (sometimes also called \textit{pseudo-Riemannian foliations}) \cite{russas,schafer}, whose general outline we give in this section. The advantage of this generality is that the description of local transverse geometry is basically unified in this larger context.

\begin{definition}[Semi-Riemannian foliations]\label{def: SR foliations}
Let $M$ be an $n$-dimensional manifold and let $\mu,q \in \mathbb{Z}$ with $0<q<n$ and $0\leq \mu\leq q$. A {\em codimension $q$ semi-Riemannian foliation of index $\mu$} on $M$ is a pair $(\f,g_\intercal)$ such that 
\begin{itemize}
    \item[i)] $\f$ is a codimension $q$ foliation on $M$;
    \item[ii)] $g_\intercal$ is an $\f$-basic symmetric $(0,2)$-type tensor field on $M$  which induces, via the isomorphism $\Phi_2$ in Proposition \ref{ prop: fundamental iso}, a holonomy-invariant semi-Riemannian metric $\overline{g_\intercal}$ of index $\mu$ on the normal bundle $\nu(\f)$. Such a tensor $g_\intercal$ is said to be a {\it semi-Riemannian transverse metric (tensor)} (of index $\mu$) for $\f$. 
\end{itemize}
In particular, a semi-Riemannian foliation of index $1$ and codimension $q\geq2$ is called a {\it Lorentzian foliation}. It is a {\it Riemannian foliation} if it has index zero. 
\end{definition}

We emphasize that a semi-Riemannian transverse metric for a given foliation $\f$ is {\it not} a semi-Riemannian metric, since it is in particular degenerate along directions tangent to the leaves of $\f$. Indeed, while any manifold possesses a Riemannian metric, it is certainly false that every foliated manifold $(M,\f)$ admits a Riemannian transverse metric, many obstructions for that being known (see the cited works and \cite{haefliger6,ghys,ghys2}). However, transverse metrics can be related to actual semi-Riemannian metrics in certain contexts. An immediate one is via the following basic characterization of existence of transverse metrics. 

\begin{proposition}\label{ prop: basic 2}
   Let $\f$ be a codimension $q$ foliation on the manifold $M$. The following statements are equivalent.
   \begin{itemize}
       \item[i)] There exists a semi-Riemannian transverse metric of index $\mu$ for $\f$.
       \item[ii)] There exist a Haefliger cocycle $(U_i,\pi_i, \gamma_{ij})_{i\in I}$ on $M$ together with semi-Riemannian metrics $h_i$ of index $\mu$ defined on $\pi_i(U_i)\subset \mathbb{R}^q$ for each $i\in I$ such that 
       $$\gamma_{ij}^\ast (h_i|_{\pi_i(U_i\cap U_j)}) = h_j|_{\pi_j(U_i\cap U_J)} \quad \forall i,j \in I,$$
   \end{itemize}
   that is, the transition maps $\gamma_{ij}$ are isometries. 
\end{proposition}
\begin{proofoutline}
    Use Proposition \ref{prop:basic 1}.
\end{proofoutline}
\begin{example}[Semi-Riemannian transverse metrics on submersions]\label{exe: submersions fol1}
    Let $\pi\colon M^{p+q}\rightarrow N^q$ be an onto submersion, and let $\f$ be the codimension $q$ foliation on $M$ whose leaves are the connected components of the fibers of $\pi$ (cf. Example \ref{exe: simple submersions and bundles}). Given any $\mu \in \mathbb{Z}$ with $0\leq \mu\leq q$ and any semi-Riemannian metric $h$ of index $\mu$ on $N$, $(\f,\pi^\ast h)$ is a semi-Riemannian foliation of same index. Moreover, any semi-Riemannian transverse metric of index $\mu$ for $\f$ arises in this way from a uniquely defined semi-Riemannian metric on $N$. In other words, there is a one-to-one correspondence between semi-Riemannian metrics on $N$ of a given index and semi-Riemannian transverse metrics of the same index for $\f$. In particular, any $\f$ arising in this way from submersions admits a Riemannian transverse metric. However, it is well-known (cf. e.g. \cite[Prop. 5.37]{oneillbook}) that $N$ admits a Lorentzian metric if and only if it admits an everywhere-nonzero vector field. Therefore, if no such vector field exists (say, if the Euler characteristic is not zero when $N$ is compact), then no foliation arising from a submersion onto $N$ can have an associated transverse Lorentzian foliation. For example, the standard Hopf fibration $\pi\colon \mathbb{S}^3 \rightarrow \mathbb{S}^2$ cannot possibly admit a transverse Lorentzian metric, though $\mathbb{S}^3$ itself (since it has Euler characteristic zero \cite[Prop. 5.37]{oneillbook}) does admit a Lorentzian metric. 
\end{example}
\begin{example}[Generators in lightlike immersions]
Let $(N,h)$ be a Lorentzian manifold. An immersion $\phi\colon  M\rightarrow N$ is \textit{lightlike} if the pullback $(0,2)$-tensor $g_\intercal:= \phi^\ast h$ is degenerate everywhere on $M$. In this case, it is well-known that due to the Lorentzian signature of $h$, for each $x\in M$ the kernel 
$$D_x := \{v\in T_xM\, :\, g_\intercal(v,w) = 0, \forall w\in T_xM\}$$
is 1-dimensional, and hence together they define an integrable distribution $D$ on $M$. If $\phi$ is totally geodesic, then the associated integral foliation $\f$ by one-dimensional submanifolds is a Riemannian foliation with transverse metric  $g_\intercal$ \cite{Kupeli,larz, minguzzi-horizon}.
\end{example}
\begin{example}[Pullbacks]\label{pullback-transv}
    By using their description via Haefliger cocycles (conf. Example \ref{exe: pullbacks}) and Proposition \ref{ prop: basic 2}, any pullback of a semi-Riemannian foliation is naturally a semi-Riemannian foliation of same index.
\end{example}
\begin{example}[Products of semi-Riemannian foliations]
    If $(M,\f,g_\intercal)$ and $(N,\g, h_\intercal)$ are semi-Riemannian foliations of indices $\mu$ and $\nu$ respectively, then $g_\intercal \oplus h_\intercal$ (via the isomorphism $T(M\times N)\simeq TM\oplus TN$) is a semi-Riemannian transverse metric of index $\mu +\nu$ on the foliation $\f\times \g$ on $M\times N$ whose leaves are $\{L_1\times L_2 \, : \, L_1\in \f, L_2\in \g\}$. 
\end{example}

The simplest and most important way to associate semi-Riemmanian metrics with foliations and thus obtain more examples is via the classic Reinhart's concept of \textit{bundle-like metrics}.

\begin{definition}\label{def bundle-like metric}
    Let $\f$ be a foliation on the manifold $M$. A semi-Riemannian metric $g$ on $M$ is {\it bundle-like (with respect to $\f$)} if the following conditions are satisfied.
    \begin{itemize}
        \item[i)] $g$ is \textit{adapted} to $\f$, \textit{i.e.}, we have a direct sum decomposition 
        $$T\f \oplus T\f^{\perp}.$$
        In particular, this induces a concomitant smooth decomposition of any vector field $X\in \mathfrak{X}(M)$ into a \textit{vertical part} $X^\top \in \mathfrak{X}(\f)$ and a \textit{horizontal part} $X^\perp \in \Gamma(T\f^\perp)$.
        \item[ii)] The $(0,2)$-type symmetric tensor field $g_\intercal$ on $M$ given by 
        \begin{equation}\label{a que define}
            g_\intercal (X,Y) := g(X^\perp, Y^\perp), \quad \forall X,Y\in \mathfrak{X}(M)
        \end{equation}
        is a semi-Riemannian transverse metric for $\f$, said to be \textit{associated with $g$}.
    \end{itemize}
    
\end{definition}
\begin{remark}\label{rmk: facts about bundle like}
 The following remarks about Definition \ref{def bundle-like metric} should be borne in mind. 
 \begin{itemize}
     \item[1)] Given a semi-Riemannian foliation $(\f,g_\intercal)$ with index $\mu$ there always exists a bundle-like semi-Riemannian metric $g$ with same index with associated transverse metric $g_\intercal$: just pick any Riemannian metric $h$ on $M$ and define 
     \begin{equation}\label{a que define 2}
     g(X,Y):= g_\intercal(X^\perp, Y^\perp) + h(X^\top, Y^\top), \quad \forall X,Y \in \mathfrak{X}(M),
     \end{equation}
     where the vertical and horizontal parts on the right-hand side are taken with respect to $h$. Then it is easily seen that $g$ is a bundle-like metric for $\f$ with index $\mu$ and the same decomposition of vectors into vertical and horizontal parts as $h$'s, so \eqref{a que define} holds by construction. 
     \item[2)] Given a semi-Riemannian bundle-like metric $g$ on $M$ with respect to $\f$, condition $(i)$ in Definition \ref{def bundle-like metric} implies that all the leaves are not just immersed submanifolds, but \textit{semi-Riemannian immersed submanifolds}, i.e., the metric induced on any leaf by $g$ is itself semi-Riemannian. This is of course trivial when $g$ is Riemannian, but not for indefinite $g$. %[conf. Example \ref{exe: simplezinho} below]. 
     In addition, since the induced metric on $\nu(\f)$ has a fixed index $\mu$, the induced metrics on the leaves $L\in \f$ must all have the same index 
     $$\text{ind} (L) = \text{ind} \, g - \mu.$$
 \end{itemize}
\end{remark}

\begin{example}[Warped products] \label{exe: warped products}
Let $(B^q,g_B)$, $(F^p,g_F)$ be connected semi-Riemannian manifolds, and $f \in C^\infty(B)$ a positive function. Then $M:= B\times F$ endowed the metric 
$$g := \pi_B^\ast g_B + (f\circ \pi_B)^2 \pi_F^\ast g_F$$
where $\pi_B,\pi_F$ are the canonical projections onto the first and second factors, respectively, is a semi-Riemannian manifold, the \textit{warped product} with \textit{warping function} $f$ of the given manifolds. The submanifolds of the form $\{b\}\times F$ are the leaves of a $p$-dimensional foliation for which $g$ is bundle-like and $\pi_B^\ast g_B$ is a transverse metric of same index as that of $(B,g_B)$ associated with $g$. In particular, this class of examples show that there exist semi-Riemannian foliations of any index, and that a bundle-like metric need not have the same index as that of the associated transverse metric. 
\end{example}
\begin{example}[Semi-Riemannian submersions] \label{exe: SR submersions}
This class of examples is very important and generalizes the warped products discussed in the previous example. Let $(M^n,g)$ and $(B^q,h)$ be semi-Riemannian manifolds with $n>q$. Recall that an onto submersion $\pi\colon M\rightarrow B$ is {\it semi-Riemannian} provided \cite[Def. 7.44]{oneillbook}
\begin{itemize}
    \item[i)] The fibers $\pi^{-1}(y)$, $y\in B$, are semi-Riemannian manifolds of $(M,g)$, and
    \item[ii)] the derivative $d\pi$ preserves scalar products of vectors normal to the fibers.
\end{itemize}
Then $g$ is bundle-like for the codimension $q$ foliation defined by the fibers (conf. Example \ref{exe: simple submersions and bundles}). Note, in particular, that the projection $\pi_B$ in Example~\ref{exe: warped products} above is a semi-Riemannian submersion. 
\end{example}
\begin{example}[Homogeneous semi-Riemannian foliations] 
 If $\f$ is a homogeneous foliation on $M$ given by the action of a Lie group $G$ (conf. Example \ref{exe: foliated actions}) and $g$ is a semi-Riemannian metric on $M$ such that the action of $G$ is $g$-isometric and all leaves are immersed semi-Riemannian submanifolds of the same index, then $g$ is bundle-like for $\f$ \cite[Remark 2.7(8)]{mrcun}. 
\end{example}
\begin{example}\label{exe: twice misner} 
Even if $\f$ is a homogeneous foliation on $M$ given by the action of a Lie group $G$ as in the previous and $g$ is a semi-Riemannian metric on $M$ such that the action of $G$ is $g$-isometric, $g$ need not be bundle-like for $\f$. For a simple and nice illustrative example, let $(M,g) = (\mathbb{R}^2\setminus \{(0,0)\},-dt^2+dx^2)$ be the $2d$-Minkowski spacetime minus the origin, and 
consider the isometric $\mathbb{R}$-action $h\colon \mathbb{R}\times M\rightarrow M$ given by $$
h(\tau,(t,x)) := \left(\begin{array}{cc}
\cosh(\tau)&\sinh(\tau)\\ \sinh(\tau)&\cosh(\tau)\end{array}\right)\left(\begin{array}{c}
t\\ x\end{array}\right), \quad \forall \tau,t,x \in \mathbb{R}, (t,x)\neq (0,0),$$ 
that is, the additive group $(\mathbb{R},+)$ is isomorphic to the \textit{orthochronous special Lorentz group} $SO_+(1,1)$ acting by pure boosts. The flat metric $g$ is \textit{not} bundle-like for the corresponding homogeneous $1d$-foliation because the orbits are not all semi-Riemannian submanifolds of the same index. In fact, the leaves are spacelike on the region $t^2>x^2$, and hence restricted to this region the foliation has index 1. The leaves are timelike on the region $t^2<x^2$, i.e., restricted to this region the foliation is Riemannian. But the leaves $t^2=x^2$ are null  lines, and hence not even semi-Riemannian submanifolds. 
\end{example}
\begin{example}[Semi-Riemannian foliations given by suspensions]\label{exe : SR suspensions }
    We look again at the suspensions of homomorphisms described in Example~\ref{example: suspensions}. Thus, using the notation in that example pick manifolds $B^p$ and $S^q$, but this time around assume we have a semi-Riemanian metric $g_0$ of index $\mu\in \{0,\ldots,q\}$ on $S$ and take an homomorphism $ h\colon \pi_1(B,x_0)\to\mathrm{Iso}(S,g_0)$ defining an action on $S$ via isometries. Fix any Riemannian metric $\eta$ on $B$, and let $\widehat{\eta}$ be the induced metric on the covering $\widehat{B}$. Finally, take the product metric $\widetilde{g}$ of $\widehat{\eta}$ and $g_0$ on $\widetilde{M}=\widehat{B}\times S$. Since the projection $\pi_2\colon  \widetilde{M}:=\widehat{B}\times S \rightarrow S$ onto the second factor is the one determining the foliation $\widetilde{\f}$, Example \ref{exe: warped products} above entails that $\pi_2^\ast g_0$ is a transverse metric of index $\mu$ for $\widetilde{\f}$ with bundle-like metric $\widetilde{g}$. The right-action of $\pi_1(B,x_0)$ on $\widetilde{M}$ is now isometric, so the smooth manifold $M=\widetilde{M}/\pi_1(B,x_0)$ inherits an induced semi-Riemannian metric $g$ which is bundle-like for the foliation $\f$ on $M$ given by the suspension of the homomorphism $h$.
\end{example}

Another useful, more geometric characterization of bundle-like metrics for semi-Riemannian foliations is given via their \textit{horizontal geodesics}, i.e., geodesics everywhere normal to the leaves. This is not too surprising since if $\pi\colon (M,g)\rightarrow (N,h)$ is a semi-Riemannian submersion, then only horizontal $g$-geodesics are always projected to $h$-geodesics on $N$. Moreover, in order to ascertain the bundle-like character of a given candidate metric it suffices to check its behavior on horizontal projectable vector fields. The next result summarizes these points. 

\begin{proposition}\label{bundlelike equivalences}
Let $(M, \mathcal{F})$ be a foliated manifold and let $g$ be a semi-Riemannian metric on $M$ adapted to $\f$ (in the sense of Def. \ref{def bundle-like metric}(i)). The following statements are equivalent.
\begin{enumerate}
\item [(i)] $g$ is bundle-like for $\mathcal{F}$.
\item[ii)] For any open set $U\subset M$ and any $X,Y \in \mathfrak{L}(\f|_U)\cap \Gamma(T\f^\perp|_U)$ we have $g(X,Y)\in \Omega^0(\f|_U)$.
\item [(iii)] If $\gamma\colon  [a,b]\rightarrow M$ is a geodesic such that $\gamma^\prime(a)\perp T_{\gamma(a)}\f$, then $\gamma$ is everywhere horizontal.
\end{enumerate}
\end{proposition}

\begin{proofcomment} The equivalence $(i)\Leftrightarrow (ii)$ is easily established by using Corollary \ref{cor: local projectable}. The equivalence $(i)\Leftrightarrow (iii)$ is rather more involved, see Ref. \cite[Thm. 1]{russas} for the details.    
\end{proofcomment}

Let $(M,\mathcal{F}, g_\intercal)$ be a semi-Riemannian foliation. There is a unique connection $\nabla^\intercal$ on the normal bundle $\nu(\f)$ which is torsion free and compatible with the induced semi-Riemannian metric $\sigma:=\overline{g_\intercal}$ (recall the overline here indicates manifold representatives of geometric objects defined on $\nu(\f)$). This is the so-called \textit{basic Levi-Civita connection}. The existence and uniqueness  $\nabla^\intercal$ follows essentially as in the Riemannian case \cite[Lemma 3.3]{molino}. Let us recall how to obtain it concretely \cite[Section 2, Eq. (3)]{schafer}. Pick any bundle-like metric $g$ associated with $g_\intercal$, and denote by $\nabla ^g$ its (usual) Levi-Civita connection. Then define, for any $\xi \in \Gamma(\nu(\f))$ and any $X \in \mathfrak{X}(M)$, 
$$\nabla^\intercal_X\xi:= \left\{\begin{array}{cc}
     \overline{[X,Y_\xi]},& \text{ if $X\in \mathfrak{X}(\f)$}, \\
     \overline{\nabla^g_X Y_\xi}&  \text{ if $X \in \Gamma(T\f^\perp)$},
\end{array}\right.,$$
 where $Y_\xi \in \mathfrak{X}(M)$ is any representative of $\xi$, i.e., $\overline{Y_\xi}=\xi$, and we extend by linearity in the first entry of $\nabla^\intercal$. We give the full statement below for future reference.

\begin{proposition}\label{prop: basic connection unique}
   Let $g_\intercal$ be a transverse semi-Riemannian metric on the normal bundle $\nu(\f)$ of a foliation $\f$ on $M$. There is exactly one connection $\nabla^\intercal$ on $\nu(\f)$ such that
   \begin{itemize}
       \item[1)] $\nabla^\intercal$ is {\em compatible with $g_\intercal$}, i.e.,
       $$X\sigma(\xi,\chi)= \sigma(\nabla^\intercal_X\xi,\chi) + \sigma(\xi, \nabla^\intercal_X\chi), \quad \forall X\in \mathfrak{X}(M), \forall \xi,\chi \in \nu(\f),$$
       where $\sigma = \overline{g_\intercal}$;
       \item[2)] $\nabla^\intercal$ is {\em torsion-free}, i.e., $$\nabla^\intercal _X\overline{Y} - \nabla^\intercal _Y \overline{X} = \overline{[X,Y]}.$$
   \end{itemize}
\end{proposition}
\begin{proofcomment}
    Having defined $\nabla^\intercal$ above, all that remains is to check uniqueness.  Let $X,Y,Z\in \mathfrak{X}(M)$. A computation entirely analogous to the one leading to the Koszul formula for the Levi-Civita connection for semi-Riemannian metrics (conf., e.g., the first paragraph of the proof of \cite[Thm. 3.11]{oneillbook}), that is, by expanding terms of the form $X\sigma(\overline{Y},\overline{Z})$ and its cyclic permutations and using $(1)$ and $(2)$ leads to the desired uniqueness.
\end{proofcomment}

We can define the \textit{transverse curvature} of $(M,\f,g_\intercal)$ by 
$$R_\intercal(X,Y)\xi := \nabla^\intercal_X\nabla^\intercal_Y\xi - \nabla^\intercal_Y\nabla^\intercal_X\xi -\nabla^\intercal_{[X,Y]} \xi,$$
for all $X,Y\in \mathfrak{X}(M)$ and all $\xi \in \Gamma(\nu(\f))$.
The following properties follow from straightforward computations [conf., e.g., discussion around \cite[Sect.~2, Prop.~1]{schafer}]. 
\begin{itemize}
    \item[1)] $\nabla^\intercal$ is holonomy-invariant, i.e., 
    $$(\mathcal{L}_V\nabla^\intercal)_X\xi := \mathcal{L}_V(\nabla^\intercal_X\xi) - \nabla^\intercal_{[V,X]}\xi- \nabla^\intercal_X(\mathcal{L}_V\xi) \equiv 0,$$
    $\forall V\in \mathfrak{X}(\f), \forall X\in \mathfrak{X}(M), \forall \xi \in \Gamma(\nu(\f))$.
    \item[2)] The transverse curvature is $\f$-basic, i.e., 
    \begin{enumerate}
        \item[2.1)] $R_\intercal(X,Y)\equiv 0$ if either $X$ or $Y$ are in $\mathfrak{X}(\f)$, and
        \item[2.2)] $\mathcal{L}_V R_\intercal =0$, $\forall V\in \mathfrak{X}(\f)$.
    \end{enumerate}
\end{itemize}
Finally, we define the \textit{transverse Ricci tensor}, given by
$$\ric_\intercal(X,Y) := \sum_{i=1}^q \varepsilon_i \sigma (R_\intercal(E_i,X)\overline{Y},\overline{E_i}),$$
for any $X,Y\in \mathfrak{X}(M)$, where $\{E_1,\ldots,E_q\}$ denotes a locally defined transverse set of vector fields on $M$ whose projections on $\nu(\f)$ define a $\sigma$-orthonormal frame thereon, and $\varepsilon_i= \sigma(\overline{E_i},\overline{E_i})$. We see from the properties of the transverse curvature listed above that the transverse Ricci tensor is well-defined independently of the local frame, and it is a symmetric $\f$-basic tensor.

\begin{remark}\label{remark:basicLCconnectiononquotients}
In Section \ref{mainsec} we will be interested in imposing certain bounds on the transverse Ricci tensor and relate them to bounds on the Ricci tensors of local quotients of $\f$. Therefore, we note here that if $\pi\colon U\to S$ is a submersion that defines $\f$ locally, and we endow $S$ with the metric $\pi_*(g_\intercal)$ and identify $(\nu\f)|_U\equiv \pi^*TS$, then $\nabla^\intercal$ is given, on $U$, by the pullback of the Levi-Civita connection $\nabla^{S}$ of $S$. Thus, on $U$ the (Ricci) curvature tensor of $\nabla^\intercal$ is just the pullback by $\pi$ of the (Ricci) curvature tensor of $\nabla^{S}$. In particular, any bound for the transverse Ricci curvature of $\f$ is essentially the same as a corresponding bound on the Ricci curvature of $S$. 
\end{remark}

\section{Elements of transverse causality}\label{section: Lorentz fol1}

We start with a very brief review of Lorentzian geometry terminology, but we shall mostly assume the reader is familiar with the basic concepts including causality theory, and refer to the textbooks \cite{beem,oneillbook} for details. 

Let $(M,g)$ be a Lorentzian manifold. A nonzero vector $v\in TM$ is {\it timelike } [resp.  {\it lightlike} (or {\it null}), {\it spacelike}]
	(in $(M,g)$) if $g(v,v)<0$ [resp. $g(v,v)=0$, $g(v,v)>0$]. A vector is {\it causal} if it is either timelike or lightlike. The zero vector is spacelike by convention. (These are the possible {\it causal types} of the vector $v$.) Given $p\in M$, the set of timelike vectors $v\in T_pM$ has two connected components, each of which is an open (in $T_pM$) convex set called a {\it timecone}. Two timelike vectors $v,w \in T_pM$ are in the same timecone if and only if $g_p(v,w)<0$. A {\it time-orientation} on the Lorentzian manifold $(M,g)$ is a continuous choice of one of the two timecones at each $p \in M$, which we always denote as $\tau_p^+$ and refer to it a the {\it future} timecone --- while the opposite timecone is denoted as $\tau^-_p$ and referred to as the {\it past} timecone. (By a ``continuous choice'' in this context we mean that for any $p \in M$ there exist an open set $U\ni p$ and a continuous vector field $X\in \mathfrak{X}(U)$, which can actually be chosen to be smooth, such that $X_q \in \tau^+_q$ for each $q\in U$.) The Lorentzian manifold $(M,g)$ is said to be {\it time-orientable} if a time-orientation exists, and in that case we can easily show, since $M$ is connected, that there exist exactly two such time-orientations. If one of these has been fixed, $(M,g)$ is said to be {\it time-oriented}. It is easy to check, via a partition of unity argument, that $(M,g)$ is time-orientable if and only if there exists a smooth {\it timelike} vector field $X\colon M\rightarrow TM$, i.e., $X_p$ is timelike for each $p \in M$. Finally, a connected, time-oriented Lorentzian manifold $(M,g)$ is called a {\it spacetime}. 
 
	Now, let $(M,g)$ be a spacetime. Given $p \in M$ any lightlike vector $v\in T_pM$ will be in only one of the sets $\mathcal{C}_p^\pm := \overline{\tau^\pm_p}\setminus \{0_p\}$, the overhead bar indicating topological closure in $T_pM$. The set $\mathcal{C}^+_p$ [resp. $\mathcal{C}^-_p$] is called the {\it future causal cone} [resp. {\it past causal cone}]. A causal vector $v\in T_pM$ is thus {\it future-directed} (or {\it future-pointing}) if $v\in \mathcal{C}^+_p$. ({\it Past-directed/past-pointing} causal vectors are defined similarly as pertaining to the past causal cone.) We can extend the notion of causal type to regular piecewise smooth curves as follows. Given a regular piecewise smooth curve $\gamma\colon I\subset \mathbb{R} \rightarrow M$, we say it is {\it future-directed timelike}/{\it future-directed null}/{\it future-directed causal} provided that for every $t\in I$ the tangent vector $\gamma'(t)$ is of the corresponding causal type. The {\it past-directed} versions are defined analogously. Although general piecewise smooth curves may not have a definite causal character, geodesics do, and in particular {\it causal geodesics} are either timelike or lightlike everywhere. 

 A peculiar feature of Lorentzian foliations (and more generally, of semi-Riemannian foliations of index $\mu\neq 0$) is the \textit{transverse causal character} of vectors. 

\begin{definition}[Transverse causal character of vectors]\label{defi: transverse causal vectors}
Let $(\f,g_\intercal)$ be a semi-Riemannian foliation of index $\mu \neq 0$ on $M$. A vector $v\in TM$ is 
\begin{itemize}
    \item \textit{transversely timelike}, if $g_\intercal(v,v)<0$;
    \item \textit{transversely lightlike}, if $v \notin T\f$ and $g_\intercal(v,v)=0$;
    \item \textit{transversely causal}, if $v$ is either transversely timelike or transversely lightlike;
    \item \textit{transversely spacelike}, if either $v\in T\f$ or else $g_\intercal(v,v)>0$.
\end{itemize}
These \textit{transverse causal characters} extend to smooth curves and vector fields in the usual way. Specifically, if $\gamma\colon I\rightarrow M$ is a smooth curve, we say it is  \textit{transversely timelike} [resp. \textit{transversely lightlike, transversely causal, transversely spacelike}], if $\displaystyle \forall t\in I, \gamma^\prime(t)$ is transversely timelike [resp. transversely lightlike, transversely causal, transversely spacelike], and if 
 $X\in\mathfrak{X}(\mathcal{U}), \mathcal{U}\subset M$ is a smooth vector field (maybe only locally defined), $X$ is \textit{transversely timelike [resp. tranversely lightlike, transversely causal, transversely spacelike]} if $X_p$ is transversely timelike [resp. tranversely lightlike, transversely causal, transversely spacelike] for all $p\in \mathcal{U}$.

\end{definition}

 For the rest of this section, we fix a codimension $2\leq q <n$ Lorentzian foliation $(\mathcal{F}, g_\intercal)$ on the manifold $M^n$. (We shall often abuse language and refer to the triple $(M, \mathcal{F}, g_\intercal)$ itself as the foliation.)
 %Recall (Definition \ref{defi: transverse causal vectors}) that we have assigned  \textit{transverse causal types} to vectors $v\in TM$. Just as the ordinary causal types for vectors mentioned in the previous paragraph, these transverse types can be naturally extended to any smooth curve $\alpha: I\subset \mathbb{R} \rightarrow M$ [respectively vector field $X\in \mathfrak{X}(M)$] by declaring that it is transversely timelike/lightlike/spacelike when its tangent vector $\alpha'(t)$ [resp. $X_x$] is of the corresponding (fixed) type for every $t\in I$ [resp. every $x\in M$]. 
 
 Now, it will often be convenient to pick a \textit{Lorentzian} bundle-like metric $g$ associated with the Lorentzian transverse metric $g_\intercal$. Remember this is always possible, and a bundle-like metric $g$ associated with $g_\intercal$ is Lorentzian if and only if the leaves are all \textit{spacelike}, i.e., have Riemannian metric induced by $g$ (conf. Remark \ref{rmk: facts about bundle like}). Whenever we use a Lorentzian bundle-like metric, for any $v\in TM$ we let $v^\top$ and $v^\perp$ denote its vertical and horizontal parts, respectively, with respect to that particular metric, without explicit reference to the latter unless there is any risk of confusion. 
 
 If such a Lorentzian bundle-like metric is defined, we can relate the usual causality types of vectors with transverse causal types as follows. 
\begin{lemma}\label{lemma: comparing causal types}
  Let $g$ be a Lorentzian bundle-like metric on $M$ associated with $(\mathcal{F}, g_\intercal)$. 
  \begin{itemize}
      \item[i)] If $v\in TM$ is timelike with respect to $g$, then it is transversely timelike on $(M, \mathcal{F}, g_\intercal)$. 
      \item[ii)] If $v\in TM$ is lightlike with respect to $g$, then either $v$ is transversely timelike, or else it is transversely lightlike and $g$-horizontal (i.e., $g$-orthogonal to a leaf of $\f$) on $(M, \mathcal{F}, g_\intercal)$. In particular, $v$ is transversely causal. 
      \item[iii)] If $v$ is transversely spacelike on $(M, \mathcal{F}, g_\intercal)$, then it is spacelike with respect to $g$. 
  \end{itemize}
  \end{lemma}
\begin{proofcomment}
 The proof is almost immediate if we recall that if $g$ a Lorentzian bundle-like metric as stated, then, using the vertical/horizontal decomposition of vectors that $g$ yields, we can write   

\begin{align}\label{basic rel}
\begin{split}
g(v,v)&=g(v^\top, v^\top) + g(v^\perp, v^\perp)\\&=g(v^\top, v^\top)+g_\intercal(v,v)\\&\geq g_\intercal(v,v), \quad \forall v\in TM,
\end{split}
\end{align}
 since we always have $g(v^\top, v^\top)\geq 0$ because $g|_{T\f}$ is positive definite. 
\end{proofcomment}

\begin{remark}\label{rmk: transverse causal type}
Every converse of the statements $(i)-(iii)$ in Lemma \ref{lemma: comparing causal types} is false. As an elementary example, pick $M=\mathbb{R}^3$ with the Minkowski metric $g$ given in global Cartesian coordinates $(t,x,y)$ by the line element $-dt^2+ dx^2+ dy^2$. Consider the 1-dimensional foliation $\f$ given by the spacelike curves $t=\text{ const.}, x=\text{ const.}$. Then $g$ is bundle-like with respect to $\f$ and $g_\intercal =-dt^2 + dx^2$ is the associated Lorentzian transverse metric. For any $a,b \in \mathbb{R}$ let 
  $$v_{a,b} = a \cdot \partial_y + b\cdot \partial_x + \partial_t.$$
 If $b^2<1$ then $v_{a,b}$ is transversely timelike [resp. transversely lightlike], but it can be timelike, lightlike or spacelike depending on the value of $a^2$. Again, if $b^2=1$ then $v_{a,b}$ is transversely lightlike, but it is spacelike for $a^2>0$, lightlike for $a=0$. (Of course, Lemma \ref{lemma: comparing causal types}(i) implies that no transversely lightlike vector can be timelike.)
\end{remark}

\begin{lemma}\label{lemma: timewedge comps}
  If $v\in TM$ is transversely timelike on $(M, \mathcal{F}, g_\intercal)$ and $w\in TM$ is such that $g_\intercal(v,w) =0$, then $w$ is transversely spacelike. 
\end{lemma}

\begin{proof} 
  Suppose $v,w\in T_xM$. Let $g$ be a Lorentzian bundle-like metric for $\f$ associated with $g_\intercal$. Then $g(v^\perp, w^\perp) =0$ and $g(v^\perp,v^\perp)<0$. Since $g_\intercal$ restricted to $T_x\f^\perp\simeq \nu(\f)$ is a \textit{bona fide} Lorentzian metric  with respect to which $v^\perp$ is timelike and $g$-orthogonal to $w^\perp$, we conclude that $g(w^\perp, w^\perp) \geq 0$ with equality if and only if $w^\perp =0$. In the latter case $w \in T_x\f$, and otherwise $g_\intercal(w,w) > 0$. In any case $w$ is transversely spacelike as claimed. 
\end{proof}

\begin{proposition}\label{prop : timewedges}
    Let $(\mathcal{F}, g_\intercal)$ be a Lorentzian foliation on $M$. Given $x\in M$, the set  
    \begin{equation}
        \tau^\intercal(x) := \{v\in T_xM \, : \, \text{ $v$ is transversely timelike}\}
    \end{equation}
    is a non-empty open set and has exactly two connected components $\tau^\intercal_{\pm}(x)$, which are convex open cones in $T_xM$, i.e.,
    $$v,w \in \tau^\intercal_{\pm}(x) ,a,b >0 \Longrightarrow a \cdot v + b \cdot w \in \tau^\intercal_{\pm}(x).$$
    Each of these two components is called a \emph{timewedge} at $x$.
\end{proposition}
\begin{proof}
  The set $\tau^\intercal(x)$ is non-empty since in particular it contains the timelike vectors at $x$ with respect to some associated Lorentzian bundle-like metric. Openness follows from the continuity of the function $v\in T_xM \mapsto g_\intercal(v,v) \in \mathbb{R}$. 
  
  Now, fix some transversely timelike vector $u\in T_xM$. Let 
  $$\tau^\intercal(\pm u) := \{v \in \tau^\intercal(x) \, : \, \mp g_\intercal(u,v) >0\}.$$
  Lemma \ref{lemma: timewedge comps} implies the disjoint union $\tau^\intercal(x) = \tau^\intercal(+ u) \cup \tau^\intercal(- u)$. It is clear that 
  $$v,w \in \tau^\intercal(\pm u) ,a,b >0 \Longrightarrow a \cdot v + b \cdot w \in \tau^\intercal(\pm u).$$
  Thus, to conclude the proof it is sufficient to show that for any transversely timelike vectors $u,v,w$ we have
  \begin{equation}\label{cmon}
      \begin{array}{c}
           g_\intercal(u,v) <0  \\
           \text{ and } \\
           g_\intercal(v,w)<0
      \end{array}
      \Longrightarrow g_\intercal(u,w)<0.
  \end{equation}
  Indeed, just as in the proof of Lemma \ref{lemma: timewedge comps} we can pick any Lorentzian bundle-like metric $g$ for $\f$ associated with $g_\intercal$. Recall, restricted to $T_x\f^\perp\simeq \nu(\f)$ $g_\intercal$ is an actual Lorentzian metric. Then, \eqref{cmon} becomes just the analogous implication for timecones in a Lorentz vector space (conf. e.g., \cite[Lemma~5.29]{oneillbook}).  
\end{proof}

\begin{remark}\label{rmk: about timewedges}
 It is evident that timewedges as defined in Prop. \ref{prop : timewedges} are the transverse analogues of timecones in usual Lorentzian geometry. We saw that if an associated bundle-like Lorentzian metric $g$ is introduced, then for each $x\in M$ the $g$-horizontal subspace $T_x\f^\perp\subset T_xM$ becomes a Lorentz vector space upon restricting $g_x$ to horizontal vectors; if we denote by $\tau_{\pm}^{\mathrm{hor}}(x) \subset T_x\f^\perp$ the two timecones therein, then one can easily check the following facts.
 \begin{itemize}
     \item[a)] The timewedges are given by ``translating the transversal timecones along vertical directions'', i.e., 
     $$\tau^\intercal_{\pm}(x)= \tau_{\pm}^{\mathrm{hor}}(x)+ T_x\f.$$
     \item[b)] The boundary of $\tau^\intercal(x)$ consists of the disjoint union of the vertical subspace $T_x\f$ and the set $\Lambda^\intercal(x)$ of transversely lightlike vectors, which we refer to as the \textit{nullwedge}. 
     \item[c)] The nullwedge $\Lambda^\intercal(x)$ also has exactly two connected components $\Lambda^\intercal_{\pm}(x)$. The union of a timewedge $\tau^\intercal_{\pm}(x)$ with the component $\Lambda^\intercal_{\pm}(x)$ of the nullwedge contained in its boundary is called a \textit{causal wedge}. It is not difficult to check that a connected component of $\Lambda^\intercal(x)$ is in the boundary of a given timewedge if and only if for any $v\in \Lambda^\intercal(x)$ in that component and any $u\in \tau^\intercal(x)$ in that timewedge we have $g_\intercal(u,v)<0$. 
     %To see this, note that given $v\in \Lambda^\intercal(x)$, we can pick a sequence $v_k$ of vectors in the timewedge (whose boundary contains $v$) converging in $T_xM$ to $v$. Pick any $u$ in that timewedge. Since $g_\intercal(u,v_k)<0$ for each $k\in \mathbb{N}$ we have $g_\intercal(u,v)\leq 0$. But if $g_\intercal(u,v)=0$ Lemma \ref{lemma: timewedge comps} would imply that $v$ is transversely spacelike, in contradiction. So we conclude that $g_\intercal(u,v)<0$ as claimed.
 \end{itemize}
\end{remark}
\begin{definition}[Transverse time-orientation]\label{def: timeorientation}
    A \textit{transverse time-orientation} for the Lorentzian foliation $(M, \mathcal{F}, g_\intercal)$ is a function $\mathcal{T}$ that assigns to each $x\in M$ a timewedge $\mathcal{T}(x) \subset T_xM$ and varies smoothly on $M$ in the sense that for each open set $U\subset M$ we have a smooth transversely timelike vector field $X\in \mathfrak{X}(U)$ such that $X_y\in \mathcal{T}(y)$ $\forall y\in U$. If such a transverse time-orientation $\mathcal{T}$ for $(M, \mathcal{F}, g_\intercal)$ has been fixed, then $(M, \mathcal{F}, g_\intercal)$ is said to be \textit{time-oriented}. In this case, for each $x\in M$ the timewedge $\tau^\intercal_{+}(x):= \mathcal{T}(x)$ is called the \textit{future timewedge} (at $x$ with respect to $\mathcal{T}$), while the opposite timewedge is the \textit{past timewedge}. 
\end{definition}

\begin{remark}
    Just as in standard causal theory on Lorentzian manifolds, fixing a transverse time-orientation on
    the Lorentzian foliation $(M, \mathcal{F}, g_\intercal)$ amounts to classifying transversely causal vectors/curves/vector fields into either future-directed or past-directed. Thus, for example, we say that a transversely causal vector field $X\in\mathfrak{X}(M)$ is \textit{(transversely) future-directed} (resp. \textit{past-directed}), if $X_x$ is in the future (resp. past) causal wedge at $x$ (with its obvious meaning) for each $x\in M$. We analogously define (transversely) future/past-directed transversely timelike/lightlike/causal vectors/vector fields/piecewise smooth curves.  
\end{remark}
As a simple criterion for transverse time-orientability we have:
\begin{proposition}\label{prop: time orient crit}
 For the Lorentzian foliation $(M, \mathcal{F}, g_\intercal)$ the following statements are equivalent.
 \begin{itemize}
     \item[i)] $(M, \mathcal{F}, g_\intercal)$ is transversely time-orientable.
     \item[ii)] There exists a smooth, globally defined transversely timelike vector field $X\in \mathfrak{X}(M)$.
     \item[iii)] Given any Lorentzian bundle-like metric $g$ on $M$ associated with $g_\intercal$ the Lorentzian manifold $(M,g)$ is time-orientable (in the usual sense).
 \end{itemize}
\end{proposition}
\begin{proof} The implication $(ii) \Rightarrow (i)$ is immediate. $(i) \Rightarrow (ii)$ follows from a simple partition of unity argument. $(iii) \Rightarrow (ii)$ is also immediate from Lemma~\ref{lemma: comparing causal types}(i) when we recall that time-orientability on the Lorentzian manifold $(M,g)$ is equivalent to the existence of a globally defined $g$-timelike vector field on $M$. Finally, $(ii)\Rightarrow (iii)$ follows from the fact that given any transversely timelike vector field $X\in \mathfrak{X}(M)$, its $g$-horizontal part $X^\perp$ is timelike since $g$ is associated with $g_\intercal$.     
\end{proof}
It is evident from the previous result that a \textit{necessary} condition for the existence of a time-orientable tranverse Lorentzian metric on a codimension $q\geq2$ foliation is the existence of a global, everywhere-nonzero section of its normal bundle. If the foliation is Riemannian and the section can be chosen to be holonomy-invariant, then this condition is also sufficient. 
\begin{theorem}\label{ teo: we have to talk about riemannian}
 Let $\f$ be a codimension $q\geq2$ foliation on the manifold $M$, and suppose there exist a transverse Riemannian metric $h_\intercal$ on $(M,\f)$ and an everywhere-nonzero holonomy-invariant section of the normal bundle $\nu(\f)$. Then there exists a transverse Lorentzian metric $g_\intercal$ such that $(M, \mathcal{F}, g_\intercal)$ is a transversely time-orientable Lorentzian foliation.
\end{theorem}
\begin{proof}
  The proof here is entirely analogous to the proof of \cite[Lemma 5.36]{oneillbook}. Let $h$ be a Riemannian bundle-like metric on $(M,\f)$ associated with $h_\intercal$. We may identify $\nu(\f)$ with the $h$-normal bundle $T\f^{\perp_{h}}$, and thus to fix an everywhere-nonzero, foliate $h$-horizontal vector field $X\in \mathfrak{L}(\f)$, which we may in addition assume to have unit $h$-norm. Let $X^\flat$ be its metrically associated (with respect to $h$) $1$-form. Then 
  $$g:= -2 X^\flat\otimes X^\flat + h$$
  can be easily seen to be a Lorentzian bundle-like metric for $\f$, for which $X$ is unit timelike and then defines a transverse time-orientation. 
\end{proof}

It is well-known (see, e.g., \cite[Section 3.1]{beem} that given any connected (not necessarily time-orientable) Lorentzian manifold $(M,g)$ there exists a smooth double covering map $c\colon \widetilde{M}\rightarrow M$ such that the Lorentzian manifold $(\widetilde{M},c ^\ast g) $ is time-orientable, with $\widetilde{M}$ connected if and only if $(M,g)$ is not time-orientable. This has the following transverse analogue.

\begin{proposition}[Transversely time-orientable double cover]\label{prop: TO covering}
Let $(M,\f,g_\intercal)$ be a Lorentzian foliation. There exists a smooth double covering map $c\colon \widetilde{M}\rightarrow M$ such that the pullback Lorentzian foliation $(\widetilde{M},c^\ast \f, c ^\ast g_\intercal)$ is transversely time-orientable and $\widetilde{M}$ connected if and only if $(M,\f, g_\intercal)$ is not transversely time-orientable.
\end{proposition}
\begin{proofoutline}
 The proof is entirely analogous to the more standard construction of a \textit{transversely orientable} double cover described in detail, e.g., on \cite[Ch. 2, Sect. 5]{camacho}. Therefore we only outline it briefly here. 

 As a set, define 
 $$\widetilde{M}=\{(x,\tau) \, :\, \text{ $x\in M$, $\tau$ is a timewedge on $T_xM$}\},$$
 together with the map $c\colon \widetilde{M}\rightarrow M$ given by $c(x,\tau) :=x,\forall x\in M$. We shall call a subset $\mathcal{V}\subset \widetilde{M}$ a \textit{base set} if there exist an open set $V\subset M$ and a transversely timelike vector field $Z\in \mathfrak{X}(V)$ such that 
 $$\mathcal{V}= \{(x,\tau(Z_x)) \, : \, x\in V \},$$
 where $\tau(Z_x)\subset T_xM$ is the unique timewedge containing $Z_x$. It is straightforward to check that all the base sets together indeed comprise a base for a smooth manifold topology on $\widetilde{M}$ for which $c\colon \widetilde{M}\rightarrow M$ is the desired double covering map. On the pullback foliation $(\widetilde{M},c^\ast \f, c ^\ast g_\intercal)$ we define a transverse time-orientation as follows. Given a transversely timelike vector $\mathcal{Z}_{(x,\tau)}\in T_{(x,\tau)}\widetilde{M}$, 
 $$\mathcal{Z}_{(x,\tau)} \text{ is future-directed if and only if } dc_{(x,\tau)}(\mathcal{Z}_{(x,\tau)}) \in \tau.$$
One can thus check that this indeed is well-defined and the remaining properties just as in the proof of \cite[Ch. 2, Sect. 5, Thm. 4]{camacho}.
\end{proofoutline}

Let us consider some further examples of transversely time-orientable Lorentzian foliations.

\begin{example}[Lorentzian warped products and submersions]\label{exe: such}
   It is clear that Examples \ref{exe: warped products} and \ref{exe: SR submersions} are immediately adapted to the present context. Thus, if $(B^q,g_B)$ is a Lorentzian manifold, $(F^p,g_F)$ is a connected Riemannian manifold and $f\in C^\infty(B)$ a positive warping function, then the warped product $B\times_f F := (B\times F,g)$ given by 
   $$g= \pi_B^\ast g_B + (f\circ \pi_B)^2 \pi_F^\ast g_F$$
   is a Lorentzian manifold. The fibers $\{b\}\times F$ naturally define a foliation by spacelike fibers with respect to which $g$ is bundle-like, and hence defines a transversely Lorentzian foliation if $q\geq 2$. Note that the leaf space in this case is identified with $B$. Thus, this foliation is transversely time-orientable if and only if $(B,g_B)$ is time-orientable. In particular, \textit{any spacetime can be naturally viewed as the leaf space of a transversely time-oriented Lorentzian foliation.} 

    More generally, if $(M^n ,g)$ and $(B^q,h)$ are Lorentzian manifolds with $n>q$ and $\pi\colon  (M,g) \rightarrow (B,h)$ is a semi-Riemannian submersion, then the connected components of the fibers automatically define a semi-Riemannian foliation of index 1, and hence a Lorentzian foliation with bundle-like $g$ if $q\geq 2$. We shall call such a submersion a \textit{Lorentzian submersion}. If $(M,g)$ is connected, then so is $(B,h)$, and again, the associated Lorentzian foliation is transversely time-orientable if and only if $(B,h)$ is time-orientable. 
\end{example}

\begin{example} \label{exe: misneragain} Here is an example of a foliation that \textit{does not admit a transverse Riemannian metric, but does admit a (transversely time-orientable) transverse Lorentzian metric}. To describe it, we revisit Example~\ref{exe: twice misner} in a different guise. 
We now fix a number $\alpha>0$ and consider the linear mapping $h\colon \mathbb{R}^2\rightarrow\mathbb{R}^2$ given by the matrix $$
\left(\begin{array}{cc}
\cosh(\alpha)&\sinh(\alpha)\\ \sinh(\alpha)&\cosh(\alpha)\end{array}\right),$$ the boost with (fixed) parameter $\alpha$. Let $\varphi\colon \mathbb{Z}\simeq \pi_1(\mathbb{S}^1)\rightarrow Iso(\mathbb{R}^2_1)$ be the group homomorphism given by $\varphi(n)=h^n$ for all $n\in\mathbb{Z}$, and consider the foliation $(M, \f)$ given by the suspension of $h$ (cf. Example \ref{exe : SR suspensions }). That is, $M$ is the quotient of $\mathbb{R}\times\mathbb{R}^2$ by the $\mathbb{Z}-$action $(n, t, \textbf{x})\mapsto (t+n, h^{-n}(\textbf{x}))$, and the leaves of $\f$ are obtained from the leaves of $pr_2\colon \mathbb{R}\times\mathbb{R}^2\rightarrow\mathbb{R}^2$. The quotient of the connected open region $x+t>0$ of $\mathbb{R}^2$ by the iterations of $h$ is known as the \textit{Misner spacetime} (cf. \cite{hawking-ellis}, Section~5.8). It is well-known that when the action is taken over the whole $\mathbb{R}^2$ the quotient ceases to be a manifold altogether, and it is not even Hausdorff. The closure of the orbit of any point lying in the either $\mathbb{R}\cdot (1,1)$ or $\mathbb{R}\cdot (1,-1)$ contains the origin, and the action is neither free nor proper. In the foliation $(M, \f)$, the leaves of the form $[pr_2^{-1}(a, \pm a)]$ are not closed, for $a\neq 0$. In fact, these leaves are homeomorphic to $\mathbb{R}$ and their closures contain the leaf $[pr_2^{-1}(0, 0)]$, which is homeomorphic to $\mathbb{S}^1$, hence these closures are not manifolds. This is at odds with Molino's structural theory for Riemannian foliations, which asserts that the closures of leaves of any Riemannian foliation are always manifolds (\cite{molino}, Ch. 5). Nonetheless, this foliation is evidently Lorentzian, since the Minkowski metric on $\mathbb{R}^2$ can be pulled back to $\mathbb{R}\times\mathbb{R}^2$ in order to yield a Lorentzian transverse metric, and thus it is projected down to a Lorentzian transverse metric on $(M,\f)$.
\end{example}

\begin{example}
A codimension $q$ foliation $(M,\f)$ is \textit{transversely parallelizable} if there exist $q$ transverse vector fields $\xi_1, \ldots,\xi_q \in \mathfrak{l}(\f)$ that form a basis of the normal bundle $\nu(\f)$ at each point (see, e.g., \cite[p. 88]{mrcun}). Thus, for example, the so-called \textit{transverse frame bundle} of any (Riemannian) foliation with the lifted foliation is always transversely parallelizable \cite[Thm. 4.20]{mrcun}. If $(M,\f)$ is transversely parallelizable, then it admits a transverse Riemannian metric \cite[Prop.4.7]{mrcun}, and thus when $q\geq 2$ we can pick a projectable vector field $X_1 \in \mathfrak{L}(\f)$ associated with $\xi_1$ (say) and use Theorem \ref{ teo: we have to talk about riemannian} to conclude that \textit{any transversely parallelizable foliation of codimension $q\geq 2$ admits a time-orientable transverse Lorentzian metric}. 
\end{example}
\begin{example}[Commuting actions]\label{exe:commuting actions}
 Let $G$ and $H$ be two Lie groups acting on the smooth manifold $M$ in such a way that their actions, $\mu$ and $\nu$, respectively, commute (i.e., $\mu_\mathrm{g}\circ\nu_\mathrm{h}=\nu_\mathrm{h}\circ\mu_\mathrm{g}$ for any $\mathrm{g}\in G,\mathrm{h}\in H$). Suppose further that $G$ is compact, $\nu$ is locally free (i.e., all stabilizers $H_x$ at any point $x\in M$ are discrete subgroups of $H$) and the actions define two everywhere transverse homogeneous foliations, $\mathcal{G}$ and $\mathcal{H}$, on $M$, with $\mathrm{codim}(\mathcal{G})\geq 2$ (conf. Example \ref{exe: foliated actions}). Then it is well-known (via group averaging \cite[VI Thm. 2.1]{bredon}) that $M$ always admits a $G$-invariant Riemannian metric $\sigma$, say, which is thus bundle-like for $\mathcal{G}$. Furthermore, any non-trivial infinitesimal generator $\xi$ of the $H$-action will be projectable with respect to $\mathcal{G}$ (by the commutation of the actions), everywhere non-zero (otherwise at a zero of $\xi$ the stabilizer would contain a copy of $\mathbb{R}$ or $\mathbb{S}^1$) and transversal to the leaves of $\mathcal{G}$. Thus $(M,\mathcal{G})$ admits a time-orientable transverse Lorentz metric by Theorem \ref{ teo: we have to talk about riemannian}. These general assumptions, although seemingly contrived, are all illustrated on the so-called \textit{stationary and axisymmetric spacetimes} such as the region outside the ergosphere in the ``slowly rotating'' Kerr-Newmann family of solutions to the Einstein field equations in general relativity. Spacetimes in this class are equipped with an isometric action of the abelian group $SO(2)\times \mathbb{R}$ satisfying all the conditions laid out here. (Here, of course, $G=SO(2)$ and $H=\mathbb{R}$.)
\end{example}
We turn to the definition of transverse analogues of the causal structure on Lorentzian manifolds. From now on we assume that $(M, \mathcal{F}, g_\intercal)$ is a transversely time-oriented Lorentzian foliation. Whenever we introduce a Lorentzian bundle-like metric $g$ on $M$ associated with $g_\intercal$, we shall implicitly assume, taking Prop. \ref{prop: time orient crit} into account, that the time-orientation on $(M,g)$ is such that any future-directed timelike vector on $M$ is transversely future-directed on $(M, \mathcal{F}, g_\intercal)$ (the \textit{induced time-orientation}). 

We have often mentioned how wild the topology of the leaf space can be from a geometer's perspective. Accordingly, one ought to expect many distinct features in the transverse causal structure, and this will be indeed be borne out by some of the results below. For example, a key tool in regular semi-Riemannian geometry is the existence of convex normal neighborhoods, which in the Lorentzian context entail that local causality ``resembles'' that of flat Minkowski spacetime, or equivalently, the geometry of the tangent space. In the present case, not only such local structure is often entirely absent, but eventual accumulation of leaves may lead to a breakdown of any local regularity around a given leaf. The next example illustrates this via a geometrically elegant instance of a suspension (recall Examples \ref{example: suspensions} and~\ref{exe : SR suspensions }) in which the leaves nevertheless exhibit a highly chaotic recurrence pattern, and the leaf space is hence quite hard to describe locally. After we have introduced the machinery of transverse causality we shall introduce more contextualized --- if simpler and easier to analyze --- examples (see, for instance, Remarks \ref{rmk: so cool counterexamples} and \ref{nownow} below). 

%%%%%%%%%%%%%%%%%%% exemplo francisco aqui

\begin{example}
    Consider the matrix
    \[A:=\begin{bmatrix}2 & 1 \\ 1 & 1\end{bmatrix}\in\mathrm{SL}(2,\mathbb{Z}),\]
    whose eigenvalues are
    \[\lambda_1=\frac{3+\sqrt{5}}{2}\ \ \text{and}\ \ \lambda_2=\lambda_1^{-1}.\]
    Let $\Phi_A\colon \mathbb{T}^2\to \mathbb{T}^2$ be the hyperbolic diffeomorphism of the torus induced by $A$. Taking $B:=\mathbb{S}^1$, let $\mathcal{F}$ be the foliation obtained via the suspension of
    \[h_A\colon \pi_1(B)\simeq \mathbb{Z}\ni n \longmapsto \Phi_A^n\in \operatorname{Diff}(\mathbb{T}^2).\]
    The background manifold $M$ will be the mapping torus of $\Phi_A$: a twisted $\mathbb{T}^2$-bundle over $\mathbb{S}^1$. Explicitly, $M=\mathbb{R}\times \mathbb{T}^2/\sim$, where $(t,x)\sim(t+1,\Phi_A(x))$.

    We claim that $\mathcal{F}$ has a natural Lorentzian foliation structure. To see this, according to the general recipe in Example \ref{exe : SR suspensions }, it suffices to define a suitable Lorentzian metric on the torus $\mathbb{T}^2$ for which $h_A$ yields an isometric $\mathbb{Z}$-action. Indeed, note that $v_1=(\phi,1), v_2=(1,-\phi)\in \mathbb{R}^2$ are eigenvectors associated to $\lambda_1$ and $\lambda_2$, where $\phi$ is the golden ratio. On $\mathbb{R}^2$, consider the metric
    \[g_0:=\omega_1\otimes\omega_2+\omega_2\otimes\omega_1,\]
where $\omega_i$ is given by $\omega_i(v_j)=\delta_{ij}$. We have that $\hat{g}_0(v_1+v_2,v_1+v_2)=2$, $\hat{g}_0(v_1-v_2,v_1-v_2)=-2$, and $\hat{g}_0(v_1+v_2,v_1-v_2)=0$, showing that $\hat{g}_0$ is indeed Lorentzian. Since it has constant coefficients, it is $\mathbb{Z}^2$-invariant, hence descends to a metric $g_0$ on $\mathbb{T}^2$. Furthermore,
\begin{align*}
    \Phi_A^*(g_0) & = (\Phi_A^*\omega_1)\otimes(\Phi_A^*\omega_2)+(\Phi_A^*\omega_2)\otimes(\Phi_A^*\omega_1)\\
     & = (\lambda_1\omega_1)\otimes(\lambda_2\omega_2)+(\lambda_2\omega_2)\otimes(\lambda_1\omega_1)\\
     & = \lambda_1\lambda_1^{-1}\omega_1\otimes\omega_2+\lambda_1^{-1}\lambda_1\omega_2\otimes\omega_1\\
     & = g_0.
\end{align*}
Hence $g_0$ does induce a transverse (clearly transversely time-orientable) Lorentzian metric on $(M,\f)$.

    Now, by construction, the dynamics of $\mathcal{F}$ is the dynamics of $\Phi_A$, which is known to be chaotic: actually this is just the famous \textit{Arnold's cat map} (see, for instance \cite{arnold}). In particular, $\f$ has a dense leaf (in fact the set of points whose leaves are dense is dense in $M$), and a leaf of $\f$ is closed if, and only if, it is the leaf at a point with rational coordinates on the section $\mathbb{T}^2$. In particular, this means that the set of points contained in closed leaves is also dense in $M$, which illustrates a striking difference between Lorentzian and Riemannian foliations: the latter are known to stratify the ambient manifold by holonomy type (see, e.g., \cite[Section 6.2]{molino} for more details). 
\end{example}

%%%%%fim do exemplo do francisco %%%%%%%%%%

\begin{definition}[Leaf space chronological and causal relations]\label{def: chronology and causality relations}
Let $L_1, L_2\in \mathcal{F}$. We define the following relations between the leaves of $\f$.
\begin{enumerate}
    \item [(a)] The \textit{leaf space chronological relation}, denoted by $L_1\ll_{M/\f} L_2$, holds if there is a transversely timelike, future-directed piecewise smooth curve $\alpha\colon [a,b]\rightarrow M$ such that $\alpha(a)\in L_1$ and $\alpha(b)\in L_2$;
    \item [(b)] The \textit{leaf space causality relation}, denoted by $L_1<_{M/\f} L_2$, holds if there is a transversely causal future-directed piecewise smooth curve $\alpha\colon [a,b]\rightarrow M$ such that $\alpha(a)\in L_1$ and $\alpha(b)\in L_2$;
    \item [(c)] We also write $L_1\leq_{M/\f} L_2$ if either $L_1<_{M/\f} L_2$ or $L_1=L_2$.
\end{enumerate}
\end{definition}

It is immediate from these definitions that $$L_1 \ll_{M/\f} L_2 \Longrightarrow L_1<_{M/\f} L_2.$$ 
It is also very convenient to introduce the analogues of the standard chronological and causal futures and pasts on $M/\f$. 

\begin{definition}[Leaf space futures and pasts]\label{def: leaf future pasts}
    Let $\mathcal{A}\subset \mathcal{F}$ be a set of leaves. We define:
\begin{itemize}
\item $I^+_{M/\f}(\mathcal{A}):=\{L^\prime\in \mathcal{F} \, :\, L\ll_{M/\f} L^\prime \text{ for some $L\in \mathcal{A}$}\}$, the \textit{leaf space chronological future} of $\mathcal{A}$;
\item $I^-_{M/\f}(\mathcal{A}):=\{L^\prime\in \mathcal{F}\,:\, L^\prime \ll_{M/\f} L \text{ for some $L\in \mathcal{A}$}\}$, the \textit{leaf space chronological past} of $\mathcal{A}$;
\item $J^+_{M/\f}(\mathcal{A}):=\{L^\prime\in \mathcal{F}\,:\, L\leq_{M/\f} L^\prime \text{ for some $L\in \mathcal{A}$} \}$, the \textit{leaf space causal future} of $\mathcal{A}$;
\item $J^-_{M/\f}(\mathcal{A}):=\{L^\prime\in \mathcal{F}\,:\, L^\prime \leq_{M/\f} L \text{ for some $L\in \mathcal{A}$} \}$, the \textit{leaf space causal past} of $\mathcal{A}$.
\end{itemize}
If $\mathcal{A}=\{L\}$, i.e., the singleton set of a leaf $L$, then we write 
$I^\pm_{M/\f}(L)$ and $J^\pm_{M/\f}(L)$. 
\end{definition}

We emphasize that the leaf space chronological and causal relations as defined are relations on the leaf space $M/\f$, \textbf{not} on $M$. This is most natural when we recall our intended goal of ``doing Lorentzian geometry on leaf spaces''. However, two important issues arise. First, it is not immediately clear from these definitions that the leaf space chronology and causality relations are even transitive, as ought to be expected if they are to have something to do with chronological/causal relations on Lorentzian manifolds as usually understood.  

The second issue regards a simple but key example. Suppose we have a smooth (onto) submersion $\pi\colon M\rightarrow B$ so that each fiber $\pi^{-1}(b)$ is connected. In this case we can identify $B\simeq M/\f$ for the foliation $\f$ given by the fibers of $\pi$ (conf. Prop. \ref{prop: submersion leaf space}(i)). Let $h$ be a time-oriented Lorentzian metric on $B$. We know (Example \ref{exe: SR submersions}) that $g_\intercal:= \pi^\ast h$ is naturally a time-oriented Lorentz transverse metric on $\f$. We should expect that the leaf space causality on this particular $(M,\f,g_\intercal)$ ought to be entirely equivalent to the (standard) causality on $(B,h)$ if our ``transverse-geometry-is-the-geometry-on-the-quotient'' philosophy is to be upheld in this example. Indeed, on the one hand, given $b_1,b_2 \in B$ such that $L_1:= \pi^{-1}(b_1) \ll _{M/\f} \pi^{-1}(b_2)=:L_2$ we can pick a transversely future-directed timelike curve $\alpha\colon [0,1]\rightarrow M$ with $\alpha(0) \in L_1,\alpha(1) \in L_2$. We then have 
$$h((\pi \circ \alpha)'(t), (\pi \circ \alpha)'(t)) = g_\intercal(\alpha '(t),\alpha'(t))< 0 , \forall t\in [0,1] \Rightarrow b_1 \ll_h b_2. $$
On the other hand, however, if we only knew that $b_1 \ll_h b_2$ it would not be immediately clear that this would lift to the corresponding relation $L_1\ll_{M/\f} L_2$. Similar considerations apply when the leaves $L_1,L_2$ are connected via transversely causal curves. To address these (and later other) issues is the object of the next two key technical lemmas. 

\begin{lemma}[Global trivialization lemma]\label{lemma: global trivialize}
    Let $\alpha\colon [a,b] \rightarrow M$ be a continuous, injective curve whose image is entirely contained in a leaf $L\in \f$. There exists a neighborhood $\mathcal{N}\supset \alpha([a,b]) $ and a diffeomorphism 
    $$\phi\colon  D^{n-q}\times D^q \rightarrow \mathcal{N}$$
    (where $D^i \subset \mathbb{R}^i$ denotes the open $i$-dimensional unit disc) such that the pullback $\phi^\ast \f$ is the foliation whose leaves are of the form $\pi_2^{-1}(y)$, where $\pi_2\colon  D^{n-q}\times D^q \rightarrow D^q$ is the projection $\pi_2(x,y) =y$.
\end{lemma}
\begin{proofcomment}
    This is a classic result in foliation theory. See, e.g., \cite[Lemma IV.3.3, p. 69]{camacho} for a proof.
\end{proofcomment}

\begin{lemma}[Transversely causal waterfall lemma]\label{lemma: causal waterfall}
 Let $\alpha\colon  [a,b]\rightarrow M$ be a piecewise smooth, future-directed transversely causal [resp. transversely timelike] curve, and let $z\in M$ be any point on the same leaf $L\in \f$ as $\alpha(a)$. There exists a  piecewise smooth, future-directed transversely causal [resp. transversely timelike] curve $\beta\colon  [c,d]\rightarrow M$ such that $\beta(c) =z$ and $\beta(d)=\alpha(b)$.
\end{lemma}
\begin{proof}
 The idea of the proof is to connect $x$ to $\alpha(a)$ by a curve within their common leaf $L$ to be concatenated with $\beta$ and ``slant it forward'' a bit, inspiring the ``waterfall'' appellation, probably due to Candel and Conlon in \cite{candel}. It is actually a modification, to incorporate transverse causality and greater codimension, of the argument in \cite[lemma 3.3.7]{candel} (which treats only codimension 1). 

 Concretely, since $L$ is connected we can pick any continuous injective curve $\gamma\colon [0,1] \rightarrow M$ whose image is entirely contained within $L$ and has $x$ and $\alpha(a)$ as endpoints. We now use Lemma \ref{lemma: global trivialize} (notation included) to fix the neighborhood $\mathcal{N}\supset \gamma[0,1]$ and diffeomorphism $\phi\colon  D^{n-q}\times D^q \rightarrow \mathcal{N}$ with the property described therein. Then we know (example \ref{pullback-transv}) that $\widehat{g}^\intercal:= \phi^\ast g_\intercal$ is a transverse Lorentzian metric on $(D^{n-q}\times D^q, \phi^\ast \f)$, and we can also induce a time-orientation so that any future-directed transversely causal [resp. timelike] curve on $D^{n-q}\times D^q$ is mapped onto a future-directed transversely causal [resp. timelike] curve on $\mathcal{N}$, i.e., $\phi$ acts like a transverse isometry. Now, $\widehat{z}_0:=\phi^{-1}(z)$ and $\widehat{\alpha}(a):= \phi^{-1}(\alpha(a))$ are on the same plaque $L_0 = D^{n-q}\times \{y_0\}$ of $\phi^\ast \f$, and by continuity for some $t_0\in (a,b]$ we have $\alpha[a,t_0]\subset \mathcal{N}$ and the curve $\widehat{\alpha}:= \phi^{-1}\circ \alpha|_{[a,t_0]}$ is future-directed $\widehat{g}^\intercal$-causal [resp. timelike]. We conclude it is sufficient to prove the lemma for $M= D^{n-q}\times D^q$ and the foliation $\f =\{D^{n-q}\times \{y\}\, :\, y \in D^q\}$ endowed with some transverse Lorentzian metric $\widehat{g}^\intercal$, and we do so for the rest of the proof. 

 We use adapted Cartesian coordinates $x^1,\ldots, x^{p:=n-q},y^1,\ldots, y^q$, and whenever we write a pair such as $(\xi_1,\xi_2)$ to describe a point in $D^p\times D^q$ it is understood that $\xi_1\in D^p$ and $\xi_2\in  D^q$. Observe that $\widehat{g}^\intercal$ being transverse means that its coefficients with respect to the coordinate basis $\{\partial_{x^i},\partial_{y^j}\}$ only depend on the $y^j$-coordinates, and even then only coefficients of the form $g_{y^jy^k}:= \widehat{g}^\intercal(\partial_{y^j},\partial_{y^k})$ are non-zero. Also, reparametrizing we take $a=0$. 

 Write $\widehat{\alpha}(t)=(\widehat{x}(t),\widehat{y}(t))$ and $\widehat{z}=(x_0,y_0)$ and define
 $$\sigma(t)=\left(\frac{\widehat{x}(t_0)-x_0}{t_0}\cdot t + x_0,\widehat{y}(t)\right), \quad 0\leq t\leq t_0.$$
 We then see that $\sigma(0)=\widehat{z}_0$ (recall that $\widehat{y}(0)=y_0$) and $\sigma(t_0)=\widehat{\alpha}(t_0)$. Moreover, 
 $$\widehat{g}^\intercal(\sigma'(t),\sigma'(t)) = \sum_{i,j=1}^q g_{y^iy^j}(\widehat{y}(t))(\widehat{y}'(t))^{i}(\widehat{y}'(t))^{j} \equiv \widehat{g}^\intercal(\widehat{\alpha}'(t),\widehat{\alpha}'(t)),$$
 and thus $\sigma$ is seen to be (future-directed and) transversely causal [resp. timelike]. By concatenating the image by $\phi$ of this curve with $\alpha|_{[t_0,b]}$ with a convenient reparametrization we obtain the desired curve $\beta$. 
\end{proof}

\begin{remark}\label{rmk: evident in waterfall}
It is clear that a simple modification the proof of Lemma~\ref{lemma: causal waterfall} can be made to show the following variant: given a future-directed (piecewise smooth) transversely causal/timelike curve $\alpha\colon [a,b] \rightarrow M$ with the \textit{future endpoint} $\alpha(b)$ on the same leaf $L$ as some $z\in M$, there exists a a future-directed (piecewise smooth) transversely causal/timelike curve starting at $\alpha(a)$ and ending at $z$. This will be used without further comments below. 
\end{remark}

The first fruit of the transversely causal waterfall lemma is the following nice property of submersions.

\begin{corollary}\label{cor: submersoes no saco}
 Let $\pi\colon M\rightarrow B$ be an onto submersion whose fibers are all connected, and consider the foliation $\f$ whose leaves are these fibers. Let $h$ be a Lorentzian metric on $B$ so that $(B,h)$ is time-oriented, and consider the Lorentzian foliation $(M,\f, g_\intercal:= \pi^\ast h)$ with the induced time-orientation. Then for all $x,y \in B$, if we write $L^x=\pi^{-1}(x)$ and $L^y=\pi^{-1}(y)$ respectively, we have
 $$L^x\ll_{M/\f} L^y \text{ [resp. $L^x<_{M/\f} L^y$]} \Longleftrightarrow x\ll_h y \text{ [resp. $x<_h y$]}.$$
\end{corollary}
\begin{proof} We have already seen that the implication $\Rightarrow$ holds. We prove the converse for the chronological relation since the proof for the causal relation is entirely analogous. Suppose then $x\ll_h y$ and fix a future-directed timelike curve $\alpha\colon [a,b]\rightarrow B$ such that $\alpha(a)=x$, $\alpha(b)=y$.  

Cover $\alpha([a,b])\subset B$ by a finite collection $U_1,\ldots, U_k$ of open sets of $B$ which are domains of local sections $\sigma_i\colon  U_i\subset B \rightarrow M$ of $\pi$, i.e., $\pi\circ \sigma_i = \mathrm{Id}_{U_i}$, for all $i\in \{1,\ldots, k\}$. Let $\{t_0=a< t_1< \cdots, t_k=b\}$ be a partition of $[a,b]$ with $\alpha([t_{i-1},t_i])\subset U_i$ for each $i\in \{1,\ldots, k\}$. Again, for each such $i$ one can define $$\widehat{\alpha}_i\colon  t\in [t_{i-1},t_i] \mapsto (\sigma_i\circ \alpha) (t) \in M.$$
We easily check that $\widehat{\alpha}_i$ is a lift of $\alpha|_{[t_{i-1},t_i]}$ through $\pi$ and that it is a transversely timelike curve segment with respect to $(M,\f, g_\intercal)$. In particular, $\widehat{\alpha}_{i}(t_i)$ and $\widehat{\alpha}_{i+1}(t_i)$ will be on the same leaf, while $\widehat{\alpha}_{1}(a)\in L^x$ and $\widehat{\alpha}_{k}(b)\in L^y$. Applying Lemma \ref{lemma: causal waterfall} to each segment in sequence, we can obtain a single piecewise smooth future-directed transversely timelike curve $\beta\colon [a,b]\rightarrow M$ with $\beta(a)\in L^x$ and $\beta(b)\in L^x$ as desired; we conclude that $L^x\ll_{M/\f} L^y$.
\end{proof}
\begin{example}\label{exe: not every sub}
  \textit{The hypothesis that the fibers are connected cannot be dropped in Corollary \ref{cor: submersoes no saco}}. To see this take the submersion $\pi\colon (t,x,y) \in \mathcal{U}\subset \mathbb{R}^3 \mapsto (t,x) \in \mathbb{R}^2$ and the Minkowski metric $h=-dt^2 + dx^2$ on $\mathbb{R}^2$ with its standard time-orientation, where 
  $$\mathcal{U}:= \mathbb{R}^3\setminus \{(t,x,0)\, : \, 0\leq t\leq 1\}.$$
  Consider
  \begin{align*} 
  &L_1=\{(0,0,y) \, : \, y< 0\}, L_a^-=\{(a,0,y) \, : \, y< 0\}\\
  &\text{and }L_a^+ = \{(a,0,y) \, : \, y>0\},\end{align*}
     with $0< a\leq 1$. Note that $\pi(L_1)= (0,0) \ll _h\pi(L_a^\pm) =(a,0)$. $L_1,L^\pm_a$ are leaves of the foliation given by the connected components of the fibers, and $\pi^\ast h$ is a Lorentzian transverse metric. Clearly, we have $L_1 \ll_{M/\f} L_a^-$, but we claim that $L_1 \nleq _{M/\f} L_a^+$. Indeed, any future-directed transversely causal curve $\alpha(s)= (t(s),x(s),y(s))$ ($0\leq s \leq 1$) in $\mathcal{U}$ starting at $L_1$ at $s=0$ would have to cross the plane $y=0$, say at $s_0 \in (0,1)$, to reach any point on $L^+_a$. But since $\alpha$ is future-directed, 
     $$t'(s) >0 \Rightarrow 0=t(0) < t(s_0) (\notin [0,1]) < t(1) \Rightarrow a\leq 1< t(1),$$
     and hence $\alpha(1) \notin L_a^+$. 
    % but it would land at the region $y>0$ at points with $t>1$, since 
     %$$-t'(s)^2 + x'(s)^2 \leq 0 \Rightarrow 1<x(1) \leq t(1).$$  
\end{example}

The next consequences are the following key structural properties of leaf space causality. 
\begin{theorem}\label{teo: structural properties leaf causal}
The leaf space chronological and causal relations on the time-oriented Lorentzian foliation $(M,\f,g_\intercal)$ have the following properties. 
\begin{enumerate}
    \item[a)] Both $\ll_{M/\f}$ and $<_{M/\f}$ are transitive relations. Moreover, they are dense preorders, i.e., for any leaves $L,L'\in \f$, if $L\ll_{M/\f} L'$, then there exists $L''\in \f$ such that $L\ll_{M/\f} L''\ll_{M/\f} L'$ and analogously for $<_{M/\f}$. Finally, $I^\pm_\f(L),J^\pm_\f(L)$ are non-empty for any $L\in \f$. 
    
    \item[b)] (Openness of chronology) The relation $\ll_{M/\f}$ is open (with respect to the product topology) in $M/\f\times M/\mathcal{F}$.
    
    \item[c)] (Push-up property) If $L_1, L_2, L_3\in\mathcal{F}$ are such that either $L_1\ll_{M/\f} L_2$ and $L_2\leq_{M/\f} L_3$ or else $L_1\leq_{M/\f} L_2$ and $L_2\ll_{M/\f} L_3$, then $L_1\ll_{M/\f} L_3$. 

    \item[d)] Let $\mathcal{A}$ be a set of leaves of $\f$. Regarding the leaf space chronological and causal futures, with respect to the quotient topology on $M/\f$ we have:
      \begin{enumerate}
      \item[d.i)] $$\overline{J^{\pm}_{M/\f}(\mathcal{A})} = \overline{I^{\pm}_{M/\f}(\mathcal{A})},$$
      \item[d.ii)] $$\mathrm{int} \, J^{\pm}_{M/\f}(\mathcal{A}) = I^{\pm}_{M/\f}(\mathcal{A}).$$
  \end{enumerate}
  In particular, 
  $$\partial J^{\pm}_{M/\f}(\mathcal{A}) = \partial I^{\pm}_{M/\f}(\mathcal{A}).$$

\end{enumerate}
\end{theorem}
\begin{proof}
%It will be convenient for some steps to fix a bundle-like Lorentzian metric $g$ associated with $g_\intercal$ (with induced time-orientation), and all mentioned future/past sets (in the usual sense) are to be taken with respect to the causal structure given by $g$. 
\begin{enumerate}
\item[(a)] The proofs are analogous for both relations, so we discuss it only for $\ll_{M/\f}$. Given $L,L',L'' \in \f$ such that $L\ll_{M/\f} L'$ and $L'\ll_{M/\f} L''$ we can pick future-directed transversely timelike curves $\alpha\colon [a,b]\rightarrow M$ and $\beta\colon [c,d] \rightarrow M$ with $\alpha(a)\in L$, $\alpha(b),\beta(c)\in L'$ and $\beta(d)\in L''$. Using the transversely causal waterfall lemma \ref{lemma: causal waterfall}, we can obtain a future-directed transversely timelike curve $\gamma\colon [b,b']\rightarrow M$ with $\gamma(b)=\alpha(b)$ and $\gamma(b')=\beta(d)$. Concatenating $\alpha$ and $\gamma$ we end up with a future-directed transversely timelike curve from $L$ to $L''$, i.e, $L\ll_{M/\f} L''$. The other claims in this item follow immediately from the definitions. 
    \item [(b)] Let $L_1, L_2\in\mathcal{F}$ be such that $L_1\ll_{M/\f} L_2$. Since the projection $\pi_\f\colon M\rightarrow M/\f$ is open, we only need to find saturated neighborhoods $\mathcal{U}_i$ of $L_i, i=1,2$ in $M$ respectively such that $L\ll_{M/\f} L^\prime$ for all $L\subset\mathcal{U}_1$ and $L^\prime\subset\mathcal{U}_2$. Since $L_1\ll_{M/\f} L_2$ there is a future-directed transversely timelike curve $\alpha\colon [0,1]\rightarrow M$ such that $\alpha(0)\in L_1$ and $\alpha(1)\in L_2$. Choose some submersion $\pi\colon  U\subset M\rightarrow \mathbb{R}^q$ with connected fibers locally defining $\f$ such that $\alpha(0) \in U$. There exists a unique Lorentzian metric $h$ on $\pi(U) \subset \mathbb{R}^q$ such that $\pi^\ast h = g_\intercal$, and it is naturally time-oriented with the induced time-orientation. Pick $0<\varepsilon<1$ small enough so that $\alpha[0,\varepsilon] \subset U$. Now, $\pi(\alpha(0))\ll_h\pi(\alpha(\varepsilon))$, and (the standard $h$-past) $I^-_h(\pi(\alpha(\varepsilon)))$ contains $\pi(\alpha(0))$ and is open. Choose therefore 
    $$\mathcal{U}_1:= \pi_\f^{-1}(\pi_\f( \pi^{-1}(I^-_h(\pi(\alpha(\varepsilon)))))),$$ 
    in other words, $\mathcal{U}_1$ is the saturation of the open set $\pi^{-1}(I^-_h(\alpha(\varepsilon)))\subset M$, which intersects $L_1$, so $L_1 \subset \mathcal{U}_1$. Now, if $L\subset \mathcal{U}_1$, then we can pick some $x\in L\cap \pi^{-1}(I^-_h(\pi(\alpha(\varepsilon))))$. Thus, $\pi(x) \ll_h \pi(\alpha(\varepsilon))$, and Corollary \ref{cor: submersoes no saco} applied to $\pi$ implies the existence of a future-directed transversely timelike curve $\gamma\colon [0,1] \rightarrow U$ with $\gamma(0) \in L, \gamma(1) \in L_{\alpha(\varepsilon)}$. Therefore $L\ll_{M/\f} L_{\alpha(\varepsilon)}$. An entirely analogous construction at $\alpha(1)$ will imply the existence of a saturated open set $\mathcal{U}_2\supset L_2$ such that for any $L' \subset \mathcal{U}_2$ we have $ L_{\alpha(1-\varepsilon)}\ll_{M/\f} L'$. We conclude by transitivity (recall (a)) that 
    $$L \subset \mathcal{U}_1, L' \subset \mathcal{U}_2 \Longrightarrow L \ll_{M/\f} L'$$
    as desired. 
    
\item[(c)] We shall prove only the first case, since the second one is analogous. Consider $\beta\colon [a, b]\rightarrow M$ a future-directed transversely causal curve such that $\beta(a)\in L_2$ and $\beta(b)\in L_3$. Let 
$$K:=\{s\in [a,b] \,:\, L_1\ll_{M/\f} L_{\beta(t)}, \forall t\in[a,s]\}.$$
Note that $K$ is not empty, since $a\in K$ by assumption. Define $s_0:=\sup(K)$. The openness of leaf space chronology we saw implies that $a<s_0$. Suppose $s_0 <b$ and pick just as in the proof of the previous item some submersion $\pi\colon  U\subset M\rightarrow \mathbb{R}^q$ with connected fibers locally defining $\f$ such that $\beta(s_0) \in U$. By continuity we take $\varepsilon>0$ such that $a<s_0-\varepsilon< s_0+\varepsilon<b$ and $\beta[s_0-\varepsilon, s_0+\varepsilon]\subset U$. For some $s' \in [s_0-\varepsilon, s_0]$ we can fix a future-directed transversely timelike curve $\alpha\colon [0,1]\rightarrow M$ such that $\alpha(0)\in L_1$ and $\alpha(1)=\beta(s')$. Again by continuity we can pick $0<\delta <1$ such that $\alpha[1-\delta,1] \subset U$. Consider again the unique Lorentzian metric $h$ on $\pi(U) \subset \mathbb{R}^q$ such that $\pi^\ast h = g_\intercal$, and it is naturally time-oriented with the induced time-orientation. Then we have by the standard push-up property in Lorentzian causality
$$\pi(\alpha(1-\delta))\ll_h \pi(\alpha(s')) <_h \pi(\beta(s_0+ \varepsilon)) \Longrightarrow \pi(\alpha(1-\delta))\ll_h \pi(\beta(s_0+ \varepsilon)).$$
Thus, by Corollary \ref{cor: submersoes no saco} applied to $\pi$ we have that $L_{\alpha(1-\delta)}\ll_{M/\f} L_{\beta(s_0+\varepsilon)} $ and by transitivity $L_1\ll_{M/\f} L_{\beta(s_0+\varepsilon)} $, i.e., $s_0+\varepsilon\in K$, a contradiction. Therefore we must have $s_0=b$.

\item[d)] We deal only with futures, since the arguments are wholly analogous for pasts.\\

\noindent $(i)$. Since $I^{+}_{M/\f}(\mathcal{A})\subset J^{+}_{M/\f}(\mathcal{A})$, we immediately get $\overline{I^{+}_{M/\f}(\mathcal{A})} \subset \overline{J^{+}_{M/\f}(\mathcal{A})}$. Now, given $L\in \overline{J^{+}_{M/\f}(\mathcal{A})}$, and given any $\mathcal{O} \ni L$ open in $M/\f$, pick $L'\in \mathcal{O}\cap J^{+}_{M/\f}(\mathcal{A})$. Let $L_0\in \mathcal{A}$ such that $L_0\leq_{M/\f} L'$. Now, density near gives the existence of some $L''\in I^{+}_{M/\f}(L') \cap \mathcal{O}$, and the push-up property yields 
$$L_0\leq_{M/\f} L'\ll_{M/\f} L'' \Rightarrow L_0 \ll_{M/\f} L'',$$
i.e., $\mathcal{O} \cap I^{+}_{M/\f}(\mathcal{A}) \neq \emptyset$. We conclude that $L \in \overline{I^{+}_{M/\f}(\mathcal{A})}$, and hence that $\overline{J^{+}_{M/\f}(\mathcal{A})} \subset \overline{I^{+}_{M/\f}(\mathcal{A})}$.\\

\noindent $(ii)$. Again, since $I^{+}_{M/\f}(\mathcal{A})\subset J^{+}_{M/\f}(\mathcal{A})$ and $I^{+}_{M/\f}(\mathcal{A})$ is open, we always have $I^{+}_{M/\f}(\mathcal{A})\subset \mathrm{int}\, J^{+}_{M/\f}(\mathcal{A})$. On the other hand, given $L \in \mathrm{int} \, J^{+}_{M/\f}(\mathcal{A})$, density near $L$ allows us to pick $L'\in I^{-}_{M/\f}(L)\cap \mathrm{int} \, J^{+}_{M/\f}(\mathcal{A})$, so there exists $L_0\in A$ for which
$$L_0\leq_{M/\f} L'\ll_{M/\f} L \Rightarrow L_0\ll_{M/\f} L$$
by push-up; thus $L\in I^{+}_{M/\f}(\mathcal{A})$, concluding the proof. \vspace{-1em}
\end{enumerate}
\end{proof}

 It is as well to deal with sets in the leaf space, but for operational convenience --- also consistent with the ``transverse approach'' adopted here --- we will need to define the (saturated) inverse images by the standard projection $\pi_\f\colon  M\rightarrow M/\f$ of leaf space futures and pasts. 

\begin{definition}[Transverse futures and pasts]\label{def: transverse future pasts}
    Let $A\subset M$ be a saturated set, so that for some set of leaves  $\mathcal{A}\subset \mathcal{F}$ we have $A=\pi^{-1}_\f(\mathcal{A})$. We define:
\begin{itemize}
\item $I^\pm_\intercal(A):= \pi_\f^{-1}\left(I^\pm_{M/\f}(\mathcal{A})\right)$, the \textit{transverse chronological future/past} of $A$;
\item $J^\pm_\intercal(A):= \pi_\f^{-1}\left(J^\pm_{M/\f}(\mathcal{A})\right)$, the \textit{transverse causal future/past} of $A$.
\end{itemize}
Again, if $\mathcal{A}=\{L\}$, then we write 
$I^\pm_\intercal(L)$ and $J^\pm_\intercal(L)$. 
\end{definition}

\begin{remark}\label{rmk: on transverse future pasts}
  Let $A\subset M$ be a saturated set. It is clear from the above definitions that $I^\pm_\intercal(A) \subset J^\pm_\intercal(A)$ and these are also saturated sets. Indeed, $x\in I^+_\intercal(A)$ (say) if and only if there exists a piecewise smooth, future-directed, transversely timelike curve $\alpha\colon [a,b] \rightarrow M$ with $\alpha(a)\in L$ for some leaf $L\subset A$ and $\alpha(b)\in L_x$, so by the causal waterfall lemma  $L_x\subset I^+_\intercal(A)$, with analogous statements for the other transverse chronological/causal future/pasts.   When one introduces a bundle-like Lorentzian metric $g$ associated with $g_\intercal$ we also have (recall Lemma \ref{lemma: comparing causal types})
  $$I^\pm_g(A) \subset I^\pm_\intercal(A) \text{ and }J^\pm_g(A) \subset J^\pm_\intercal(A),$$
  where $I^\pm_g,J^\pm_g$ are the standard future/pasts defined via $g$-causal curves. By saturation, we also have more generally 
  \begin{equation}\label{not the same uncle}    
  \pi^{-1}_\f(\pi_\f(I^\pm_g(A)) \subset I^\pm_\intercal(A) \text{ and }\pi^{-1}_\f(\pi_\f(J^\pm_g(A)) \subset J^\pm_\intercal(A).
  \end{equation}
\end{remark}

\begin{example}
The inclusions \eqref{not the same uncle} are strict in general. For a simple example, let $$K=\{(t, x, y)\in\mathbb{R}^3 \,:\, y\leq 4, |t| \leq 1, |x|\leq 4\}\cup \{(-2,0,0)\}$$ 
and take $M:=\mathbb{R}^3\setminus K$ foliated by the fibers of the submersion $(t, x, y)\mapsto (t,x)$ and endowed with the bundle-like Lorentzian metric $g=-dt^2 + dx^2+dy^2$. Consider the leaf 
$$L=\{(-2,0,y)\, : \, y<0\}.$$
%point $p=(-2,0, 0)$. We see that $$I^+_g(p)=\{(t, x, y)\in M: -2<t<-1, t^2>x^2+y^2\}$$ so that 
We see that its saturation is $$\pi^{-1}(\pi(I^+_g(L)))=\{(t, x, y)\in M \,:\, t<-1,\  t+2>|x|\},$$ while $$I^+_\intercal(L)=\{(t, x, y)\in M\,:\, t+2>|x|\}.$$
\end{example}

Theorem \ref{teo: structural properties leaf causal} has the following immediate consequence for transverse futures and pasts.

\begin{corollary}\label{cor: structural properties transverse causal}
The transverse chronological and transverse causal pasts and futures of any saturated set $A\subset M$ for a time-oriented Lorentzian foliation $(M,\f,g_\intercal)$ have the following properties. 
\begin{enumerate}
    \item [a)] $I^\pm_\intercal(A)$ are open in $M$.

    \item[b)] With respect to the manifold topology on $M$, we have $$\overline{J^{\pm}_\intercal(A)} = \overline{I^{\pm}_\intercal(A)},$$
      and
      $$\mathrm{int} \, J^{\pm}_\intercal(A) = I^{\pm}_\intercal(A).$$
  In particular, 
  $$\partial J^{\pm}_\intercal(A) = \partial I^{\pm}_\intercal(A).$$
\end{enumerate}
\end{corollary}

We now show that these boundaries $\partial I^{\pm}_\intercal(A)$ are actually $C^0$ hypersurfaces of $M$, just like achronal boundaries in the usual causal theory on spacetimes. In fact, this conclusion will follow immediately from the combination of two features we will see they share with achronal boundaries: achronality and edgelessness. We first introduce transverse analogues of these concepts.

\begin{definition}
    We say that a set $A\subset M$ is \textit{transversely achronal} (on $(M,\f, g_\intercal)$) if no two points of $A$ can be connected by a transversely timelike curve.
    %, or equivalently, if $A\cap I^+_\intercal(A)=\emptyset$.
\end{definition}
Let $g$ be any Lorentzian bundle-like metric associated with $g_\intercal$. Since any $g$-timelike curve is in particular $g_\intercal$-transversely timelike, it is immediate that 
$$A\subset M \text{ transversely achronal on $(M,\f,g_\intercal)$} \Longrightarrow A \text{ is achronal on $(M,g)$}.$$
\begin{proposition}
For any saturated set $A\subset M$ we have that $\partial I^{\pm}_\intercal(A)$ are transversely achronal.
\end{proposition}
\begin{proof}
We only show that for $\partial I^+_\intercal(A)$ since the other case is analogous. If this was not transversely achronal we would have $x,y\in \partial I^{\pm}_\intercal(A)$ with $x\in I^-_\intercal(L_y)$. Since the latter set is open, it would then intersect $I^+_\intercal(A)$, i.e., there would exist $x' \in I^+_\intercal(A)\cap I^{-}_\intercal(L_y)$. We conclude (using the causal waterfall lemma) that $y\in \partial I^+_\intercal(A)\cap I^+_\intercal(A)$, an absurd.  
\end{proof}

\begin{definition}\label{defi: transverse closure}
 Given $A\subset M$, we define the \textit{transverse closure} of $A$, denoted by $\overline{A}^\intercal$, as the closure of $A$ with respect to the topology $\tau_\f$ of saturated open subsets of $M$. If $A=\overline{A}^\intercal$, then $A$ is said to be \emph{transversely closed}. 
\end{definition}
Since $M\setminus \overline{A}^\intercal \in \tau_\f$ by definition, this set is also open and saturated (with respect to the manifold topology) in $M$, whence $\overline{A}^\intercal$ is always closed and saturated in $M$. In fact, it is easy to see that $\overline{A}^\intercal$ is precisely the smallest closed (in the manifold topology) and saturated set containing $A$. Therefore, we also have
\begin{itemize}
    \item  $\overline{A}^\intercal=\overline{\pi_\f^{-1}(\pi_\f(A))}^\intercal$ (so in taking transverse closures we may consider only saturated sets) and
    \item  $\pi_\f(\overline{A}^\intercal) = \overline{\pi_\f(A)}$, where the latter closure is taken on $M/\f$. 
\end{itemize}
\begin{definition}\label{defi: transverse edge}
    Let $A\subset M$ be a transversely achronal set. We define the \textit{transverse edge of} $A$, denoted by $\mathrm{edge}_\intercal(A)$ by the set of points $x\in\overline{A}^\intercal$ such that for every saturated neighborhood $\mathcal{U}$ of $x$ there exists a future-directed transversely timelike curve $\gamma$ contained in $\mathcal{U}$ which connects $I^-_\intercal(L_x,\mathcal{U})$ to $I^+_\intercal(L_x,\mathcal{U})$ without meeting $A$.
\end{definition}
It is immediate from this definition that 
\begin{itemize}
    \item If $A$ is transversely achronal, and $g$ is any Lorentzian bundle-like metric associated with $g_\intercal$, then $\mathrm{edge}_g (A)\subset \mathrm{edge}_\intercal(A)$, and 
    \item for any saturated set $A\subset M$ we have $\mathrm{edge}_\intercal(\partial I^\pm _\intercal(A)) = \emptyset$. This already shows that $\partial I^\pm _\intercal(A)$ are closed and edgeless sets, and therefore $C^0$ hypersurfaces as claimed.
\end{itemize}

\begin{example}
Consider $M=\mathbb{R}^3$ with global coordinates $(t,x,y)$, with the foliation $\f$ given by $(t,x)=constant$ endowed with Lorentzian transverse metric $g_\intercal=-dt^2+dx^2$, and associated Lorentzian bundle-like metric $g=-dt^2+dx^2+dy^2$ on $M$. Consider the sets \begin{align*}
A&:=\{t=0, x^2+y^2>0\},\\
B&:=\{t=sin(y)\}.
\end{align*}

$A$ is both achronal in $(M,g)$ and transversely achronal in $(M,\f,g_\intercal)$, but while $\mathrm{edge}(A)=\{(0,0,0)\}$ we have $\mathrm{edge}_\intercal(A)=\{t=0, x=0\}$. As for $B$, it is achronal but not transversely achronal, since $(0,0,0)$ and $(1,0,\pi/2)$ which are in $B$ can be joined by the curve $s\mapsto (s, 0, (\pi/2)s)$ for $s\in[0,1]$ which is not timelike but it is transversely timelike.
\end{example}

\begin{lemma}\label{edgeless is saturated}
If $A\subset M$ is transversely achronal with $\mathrm{edge}_\intercal(A)=\emptyset$, then $A$ is transversely closed.    
\end{lemma}
\begin{proof}
We first show that $\overline{A}^\intercal$ is also transversely achronal. Indeed, suppose not, and pick $x,y\in \overline{A}^\intercal$ connected by some transversely timelike curve segment. But then $L_x,L_y \in \pi_\f(\overline{A}^\intercal)=\overline{\pi_\f(A)}$ and $L_x \ll_{M/\f}L_y$. Since the leaf space chronological relation is open, there exist $L',L'' \in \pi_\f(A)$ such that $L'\ll_{M/\f} L''$, and thus by the causal waterfall lemma \ref{lemma: causal waterfall}, we can construct a future-directed transversely timelike curve connecting a point $x'\in L'\cap A$ to a point $y \in L''\cap A$, contradicting the transverse achronality of $A$. 

Thus, to complete the proof it suffices to show that $\overline{A}^\intercal\setminus A \subset \mathrm{edge}_\intercal(A)$. But given $x\in \overline{A}^\intercal\setminus A$, let $\mathcal{U}$ be any saturated neighborhood of $x$ and $x^+, x^-$ points respectively in $I^+_\intercal(L_x,\mathcal{U})$ and $I^-_\intercal(L_x,\mathcal{U})$. We can view $(\mathcal{U},\f|_{\mathcal{U}}, g_\intercal|_{\mathcal{U}})$ as a Lorentzian foliation by itself, and we do so for the rest of the proof. Using the causal waterfall lemma \ref{lemma: causal waterfall}, we can construct a future-directed transversely timelike curve $\alpha\colon [0,1] \rightarrow \mathcal{U}$ with $\alpha(0)=x^-$ and $\alpha(1)=x^+$ and passing through $x$. Since $\alpha$ intersects $\overline{A}^\intercal$ at $x$ and we saw this set is transversely achronal, it cannot intersect it anywhere else, and in particular, it cannot intersect $A$. Therefore $x\in \mathrm{edge}_\intercal(A)$ as desired. 
\end{proof}
\begin{corollary}
    For any saturated set $A\subset M$, the sets $\partial I^\pm _\intercal(A)\subset M$ are transversely closed (and in particular closed and saturated).
\end{corollary}

\begin{theorem}
Let $A\subset M$ be a transversely achronal set. If $\mathrm{edge}_\intercal(A)=\emptyset$, then $A$ is a transversely closed (hence saturated) $C^0$ hypersurface. 
\end{theorem}

\begin{proof}
%$(\implies)$ 
Let $g$ be a Lorentzian bundle-like metric associated with $g_\intercal$. Since in this case $A$ is achronal and edgeless in $(M,g)$, then $A$ is a closed $C^0$ hypersurface by standard results \cite[Cor. 14.26]{oneillbook}. It is transversely closed by Lemma \ref{edgeless is saturated}. 
\end{proof}

\section{The transverse causal ladder}\label{section: Lorentzfol2}

Another important aspect of the geometry of spacetimes, especially in physical applications, is their position on the so-called \textit{causal ladder} (or \textit{causal hierarchy}) (for a comprehensive discussion on this, see \cite{minguzzi-sanchez}). Thus, for example, a spacetime $(M,g)$ is said to be {\it chronological} [resp. {\it causal}] if there is no timelike [resp. causal] curve segment $\gamma\colon [a,b]\rightarrow M$ with $\gamma(a)=\gamma(b)$. Clearly if $(M,g)$ is causal, then it is chronological, but the converse is false. The physical importance of this classification is that while non-causal spacetimes are widely believed to be of scant physical interest insofar as these describe universes in which two-way communication with one's past is possible, it is often unclear just how strong the physics actually constrains the causal structure. Therefore, a thorough understanding of these causal possibilities is desirable. We thus extend the usual causal conditions and obtain a --- for now partial --- \textit{transverse causal ladder}, and investigate its relation to the standard causality defined via an associated bundle-like time-oriented Lorentzian metric. 

We again fix throughout this section a codimension $q\geq2$ time-oriented Lorentzian foliation $(M, \mathcal{F}, g_\intercal)$. As made clear in the previous section, there is a particularly nice interplay between causal transverse notions, defined on $M$, and a causal structure defined directly on the leaf space $M/\f$, which we will keep exploring here. 

We first introduce a (purely foliation-theoretic) notion that will recur a number of times. 
\begin{definition}[Tranverse compactness]\label{defi:transverse compactness}
Given a foliation $(M,\f)$, a subset $K\subset M$ is \textit{transversely compact (with respect to $\f$)} if any covering of $K$ by $\f$-saturated open sets admits a finite subcovering. 
\end{definition}
Of course, any compact subset of $M$ is transversely compact, but the converse is not true in general. Indeed, a transversely compact set need be neither closed nor bounded (with respect to a natural background distance function). For example, if the foliation is given by the fibers of the projection $\pi\colon (x,y,z)\in \mathbb{R}^3\mapsto (x,y) \in \mathbb{R}^2$ the set 
$$\{(x,y,z) \in \mathbb{R}^3 \, : \, x^2+y^2\leq 1\} \setminus \{(0,0,0)\}$$
is transversely compact. More generally, it is an easy exercise left to the reader to check the following characterization.
\begin{proposition}\label{prop: transversely compact charac}
For a subset $K\subset M$ and a foliation $\f$ of $M$ the following statements are equivalent.
\begin{itemize}
    \item[i)] $K$ is transversely compact.
    \item[ii)] $\pi_\f(K)$ is compact in the leaf space $M/\f$.
    \item[iii)] The saturation $\pi^{-1}_\f(\pi_\f(K))$ of $K$ is transversely compact. 
\end{itemize}
\end{proposition}
%\qcd
We now turn to the transverse causal ladder. 
\begin{definition}\label{defi: transverse ladder}
We say that $(M, \mathcal{F}, g_\intercal)$ is
\begin{itemize}
\item \textit{transversely totally vicious} if $L\cap I^+_\intercal(L) \neq \emptyset$ (equivalently $L\in I^{+}_{M/\f}(L)$), $\forall L\in\mathcal{F}$;
\item \textit{transversely non-totally vicious} if there is at least one $L\in\f$ such that $L\cap I^+_\intercal(L) = \emptyset$;
    \item \textit{transversely chronological} if $L\cap I^+_\intercal(L) =\emptyset$ (equivalently $L\notin I^{+}_{M/\f}(L)$), $\forall L\in\mathcal{F}$;
    \item \textit{transversely causal} if $L\nless_{M/\f} L$ $ \forall L\in\mathcal{F}$;
    \item \textit{transversely future [resp. past] distinguishing} if $\forall L,L'\in \f$
   % $$I^+_\intercal(L) =I^+_\intercal(L') \text{ [resp. $I^-_\intercal(L) =I^-_\intercal(L')$]} \Longrightarrow L=L',$$
    $$L\leq_{M/\f} L' \text{ and } L \in \overline{J^+_{M/\f}(L')}\text{ [resp. $L'\in \overline{J^-_{M/\f}(L)}$]} \Longrightarrow L=L',$$
    and \textit{transversely distinguishing} if it is both transversely future distinguishing and transversely past distinguishing;
    \item \textit{transversely strongly causal} if $\forall L\in \mathcal{F}$ and every saturated open neighbourhood $\mathcal{U}\supset L$ in $M$, there exists an open neighborhood $L\subset \mathcal{V}\subset \mathcal{U}$ such that if $\gamma\colon [a,b]\rightarrow M$ is any transversely causal curve with $\alpha(a), \alpha(b)\in\mathcal{V}$ then $\alpha([a,b])\subset\mathcal{U}$;
    %\item  \textit{transversely stably causal} if $M$ admits a transverse time-function $f$, \textit{i.e.}, a function $f\colon M\rightarrow\mathbb{R}$ which is strictly increasing along any future-directed transversely timelike curve;
    \item \textit{transversely globally hyperbolic} if it is transversely strongly causal and if for all $L, L'\in\mathcal{F}$ the \textit{causal prism} $J^+_\intercal(L)\cap J^-_\intercal(L')$ is transversely compact in $M$. 
\end{itemize}
\end{definition}
\begin{remark}\label{rmk: cute fact}
It should be clear that transverse causality, as something naturally pertaining to the quotient $M/\f$, can be very ``bad'' while the causality on $(M,g)$ (with $g$ bundle-like time-oriented Lorentzian metric) is ``very good''\footnote{This should not be surprising, as there are well-known precedents from isometric group actions in Lorentzian geometry: quotients of globally hyperbolic spacetimes (say) by suitable isometries can even be totally vicious. Perhaps the simplest example arises by identifying points via their time coordinate $(t,x) \sim (t+ 1,x)$ on the (globally hyperbolic) $2d$ Minkowski spacetime along the ``time axis'' to form a totally vicious flat cylinder.}. For instance, consider, on $3d$ Minkowski spacetime $(\mathbb{R}^3,-dt^2+dx^2+dy^2)$, the foliation given by the integral curves of the spacelike Killing vector field $X= \partial _t + 2\partial_x $. Upon the isometric identification $(t,x+ 2, y) \sim (t,x,y)$ the resulting spacetime $(M,g)$ is still globally hyperbolic, but the induced foliation $\f$ becomes one by spacelike helices each one of which is intersected by timelike curves infinitely many times, and hence this $(M,\f,g_\intercal)$ is transversely totally vicious.  
\end{remark}

Since the standard proof \cite[Lemma 14.10]{oneillbook} that compact spacetimes are never chronological only uses openness of futures that also apply for leaf space chronology, we immediately have (recall Proposition \ref{prop: transversely compact charac}):
\begin{corollary}
$(M,\f,g_\intercal)$ is not transversely chronological if $M$ is transversely compact. 
\end{corollary}

It is immediate from the definitions that transversely globally hyperbolic $\implies$ transversely strongly causal, and that transversely causal $\implies$ transversely chronological $\implies$ transversely non-totally vicious. It should also be clear than none of converses of these implications holds in general, since we can easily adapt from spacetime examples where they fail. However, the implications transversely strongly causal $\implies $ transversely distinguishing $\implies $ transversely causal requires further discussion. 

First we check that our definition of transverse strongly causality does yield some nice properties. 
\begin{proposition}\label{prop: strong causality characterization}
  The following statements regarding the time-oriented Lorentzian foliation $(M,\f,g_\intercal)$ are equivalent.   
  \begin{enumerate}
      \item[i)] $(M,\f,g_\intercal)$ is transversely strongly causal.
      \item[ii)] For any $L\in \f$, and any open set $\widehat{U} \ni L$ in $M/\f$ there exists another open set $L\in \widehat{V}\subset \widehat{U}$ of $M/\f$ such that $\forall L_1,L_2\in \widehat{V}$ and for any $L'\in \f$
      $$L_1 <_{M/\f} L'<_{M/\f} L_2 \Longrightarrow L'\in \widehat{U}.$$
      \item[iii)] The sets of the form $I^+_{M/\f}(L_1)\cap I^-_{M/\f}(L_2)$ for every $L_i\in \f$ are a basis for the quotient topology on $M/\f$, i.e., the quotient topology of the leaf space coincides with the \emph{Alexandrov topology} arising from the chronological relation. 
  \end{enumerate}
  
\end{proposition}
\begin{proof}
   The equivalence $(i)\Leftrightarrow (ii)$ is essentially rephrasing leaf space causality into transverse causality and vice-versa.
   
   Let us show $(ii)\Rightarrow (iii)$. Assuming that $(ii)$ holds and given any $L\in \f$, and any open set $\widehat{U} \ni L$ in $M/\f$, and the open set $L\in \widehat{V}\subset \widehat{U}$ of $M/\f$ such that $\forall L_1,L_2\in \widehat{V}$ and for any $L'\in \f$
      $$L_1 <_{M/\f} L'<_{M/\f} L_2 \Longrightarrow L'\in \widehat{U},$$
      we can indeed find by the density of chronology two $L_1,L_2\in \widehat{V}$ with $L_1\ll_{M/\f} L\ll_{M/\f} L_2$, so that $I^+_{M/\f}(L_1)\cap I^-_{M/\f}(L_2) \subset \widehat{U}$.
      
      The proof of $(iii)\Rightarrow (ii)$ is straightforward. 
\end{proof}

\begin{proposition}\label{prop: cute fact}
  If $(M,\f,g_\intercal)$ is 
  \begin{enumerate}
      \item[a)] either transversely past or transversely future distinguishing, then it is transversely causal;
      \item[b)] transversely strongly causal and the (quotient) topology of the leaf space is $T_1$, then it is transversely distinguishing.
      %causal. 
      %\item[c)] transversely strongly causal and the (quotient) topology of the leaf space is Hausdorff, then it is transversely distinguishing. 
  \end{enumerate}
\end{proposition}

\begin{proof}
    For $(a)$, suppose that for some leaf $L\in \f$ we have $L<_{M/\f} L$, and let $\alpha\colon [0,1] \rightarrow M$ be a future-directed transversely causal curve with $\alpha(a),\alpha(b)\in L$. Then since $\alpha$ must leave $L$, it will intersect some other leaf $L'\neq L$. But then $L<_{M/\f} L'<_{M/\f} L$, so have both $L'<_{M/\f}L$ and $L'\in \overline{J^+_{M/\f}(L)}$ for distinct $L,L'$,
    %and it follows from transitivity and push-up (Theorem \ref{teo: structural properties leaf causal}) that 
    %$$I^\pm_{M/\f}(L) = I^\pm_{M/\f}(L') \Rightarrow I^\pm_\intercal(L) = I^\pm_\intercal(L'),$$
    and similarly for transverse pasts, i.e., $(M,\f,g_\intercal)$ is neither transversely past nor transversely future distinguishing.

    Let us now prove $(b)$. Suppose the leaf space topology is $T_1$, and that $(M,\f,g_\intercal)$ is not, say, transversely future-distinguishing, the past-distinguishing case being entirely analogous. Then there exist $L\neq L'\in \f$ such that $L<_{M/\f}L'$
    and $L\in \overline{J^+_{M/\f}(L')}$. By the $T_1$ property, there exists an open set $\mathcal{U}\ni L$ in $M/\f$ not containing $L'$. Let $L \subset \mathcal{V}\subset \mathcal{U}$ be any open set and pick $L''\in \mathcal{V}\cap J^+_{M/\f}(L')$. Thus, by transitivity, we have that $L<_{M/\f}L'<_{M/\f}L''$. We conclude that $(M,\f,g_\intercal)$ is not transversely strongly causal (conf. Prop. \ref{prop: strong causality characterization}(ii)). 
    %If $L<_{M/\f} L$ for some $L\in \f$, then there exists some leaf $L'\neq L$ such that $L<_{M/\f} < L' <_{M/\f} L$. By the $T_1$ property we have some open set $\widehat{U}\ni L$ not containing $L'$, which according to Proposition \ref{prop: strong causality characterization}(ii) means $(M,\f,g_\intercal)$ cannot be transversely strongly causal.
   % 
    %Finally, for $(c)$, we show that the transverse strong causality together with the Hausdorff character of the leaf space topology imply the \textit{future} transverse distinguishing property, since the past case is analogous. It is actually a straightforward adaptation of the standard proof in Lorentzian causal theory (see, e.g., \cite{minguzzi-sanchez}). Suppose that there are $L\neq L'\in \f$ with $I^+_{M/\f}(L)=I^+_{M/\f} (L')$. From the Hausdorff property we can assume there are disjoint open sets $\widehat{U}\ni L$, $\widehat{V} \ni L'$ of $M/\f$. By transverse strong causality in the form of Proposition \ref{prop: strong causality characterization}(iii) we have some $L_2,L_1 \in \f$ such that $L \in I^+_{M/\f}(L_1)\cap I^-_{M/\f}(L_2)\subset \widehat{U}$. But then $L_2 \in I^+_{M/\f}(L)=I^+_{M/\f} (L')$, that is $L' \in I^-_{M/\f}(L_2)$. Since $\widehat{V} \cap I^-_{M/\f}(L_2) $ is open in $M/\f$ and contains $L'$ pick any $L'' \in \f$ contained in the set $\widehat{V} \cap I^-_{M/\f}(L_2) \cap I^+_{M/\f}(L')$, whence we conclude that $L_1\ll_{M/\f} L\ll_{M/\f} L''\ll_{M/\f} L_2$ and therefore $L'' \in \widehat{U}\cap \widehat{V}$, a contradiction. 
\end{proof}

\begin{remark}\label{rmk: so cool counterexamples}
\textit{Without the topological assumption in item $(b)$ of Prop. \ref{prop: cute fact} the conclusion therein may fail.} To see that the implication in $(b)$ fails if $M/\f$ is not $T_1$, start with 3-dimensional Minkowski spacetime $(\mathbb{R}^3,-dt^2+dx^2+dy^2)$ and consider the 1-dimensional homogeneous foliation $\hat{\mathcal{G}}$ given by the orbits of the isometric $\mathbb{R}$-action 

$$\Phi: (\lambda, (t,x,y)) \in \mathbb{R}\times \mathbb{R}^3 \mapsto (t+\lambda ,x+a\lambda, y+b\lambda) \in \mathbb{R}^3$$
where $a,b \in \mathbb{R}\setminus \mathbb{Q}$ are fixed, rationally independent and chosen such that the $\mathbb{R}$-orbits are spacelike. Notice that, since $\mathbb{R}$ acts by translations, $\hat{\mathcal{G}}$ is invariant by the natural $\mathbb{Z}^3$-action on $\mathbb{R}^3$, and thus projects down to a foliation $\mathcal{G}$ of the flat Lorentzian torus $\mathbb{T}^3=\mathbb{R}^3/\mathbb{Z}^3$. By construction, every leaf of $\mathcal{G}$ is dense in $\mathbb{T}^3$ (indeed this is just a $3$-d version of the Kronecker foliation), and the induced Lorentzian flat metric $g$ is bundle-like for $\mathcal{G}$ by construction. On the one hand, $(\mathbb{T}^3,g)$ is a compact spacetime, and hence has a closed $g$-timelike curve \cite[Lemma 14.10]{oneillbook}, which is then a closed transversely timelike curve (by Lemma \ref{lemma: comparing causal types}) with respect to the associated Lorentzian transverse metric $g_\intercal$. In other words, $(\mathbb{T}^3,\f,g_\intercal)$ is not transversely chronological, hence not transversely causal of past/future distinguishing either. Now, on the one hand the leaf space has the indiscrete topology, which is not even $T_0$, let alone $T_1$. On the other hand, this topology is trivially the same as the Alexandrov topology in Prop. \ref{prop: strong causality characterization} (which is evidently always coarser than the quotient topology from the openness of leaf space chronological futures and pasts). Therefore $(\mathbb{T}^3,\f,g_\intercal)$ is transversely strongly causal. 
\end{remark}

One of the main technical reasons why strong causality is important in causal theory is that it ensures \textit{non-partial imprisonment} of causal curves in compact sets, i.e., inextendible causal curves that enter some compact set eventually leave it and never return. We now show that there is a transverse analogue of that statement. 
\begin{definition}\label{def: inextendibility}
Let $\gamma\colon [a,b)\rightarrow M$ be a continuous curve. We say that a leaf $L\in\mathcal{F}$ is a \textit{transverse limit to the right} for $\gamma$ if for any saturated open set $\mathcal{U}\supset L$ there exists some $s_0\in[a,b)$ such that $s\geq s_0\implies \gamma(s)\in \mathcal{U}$. We say that $\gamma$ is \textit{transversely right-inextendible} if there are no transverse limits to the right for $\gamma$. We analogously define versions to the left for a curve defined on $(a, b]$. Finally, a curve $\gamma\colon (a,b)\rightarrow M$ is \textit{transversely inextendible} if for some (and hence for any) $c\in (a,b)$ we have both that $\gamma|_{[c,b)}$ is transversely right-inextendible and $\gamma|_{(a,c]}$ is transversely left-inextendible. 
\end{definition}
\begin{remark}\label{nownow}
Note that any curve that is transversely right-inextendible is in particular right-inextendible (in the usual sense), but the converse is not true. Indeed, it is clear that the definition of transverse (right- or left-)inextendibility is equivalent to the requirement that the projection of the curve onto the leaf space $M/\f$ is (right- or left-)inextendible in the usual sense therein. But a curve on $M$ may be inextendible and yet have an extendible projection. Furthermore, since the leaf space need not be Hausdorff, a curve may well have more than one transverse limit to the right and/or to the left. As a simple concrete example, consider the manifold $M=\mathbb{R}^3\setminus \{(0,0,0)\}$ with the foliation $\f$  given by the fibers of the projection $\pi\colon  (t,x,y)\in M\mapsto (t,x)\in \mathbb{R}^2$ with transverse Lorentzian metric on $(M,\f)$ given by $g_\intercal =-dt^2+dx^2$ time-oriented by $\partial_t$, and the leaves
$$L_\pm := \{(0,0,y)\in M \, : \, \pm y >0\}.$$
The curve $\alpha\colon  s \in [-1,0) \mapsto (s,0, \sin (1/t)) \in M$ clearly has both $L_\pm$ as transverse limits to the right, though it is right-inextendible. The existence of these two transverse right-limits is possible because the topology on leaf space is not Hausdorff. Indeed, any two leaves other than $L_\pm$ are pairwise Hausdorff-related, while the saturated open set 
$$\mathcal{U}_+= \{(t,x,y) \in M\, :\, x^2+ t^2 <1\} \setminus \{(0,0,y) \in M\, :\, y<0\}$$
contains $L_+$ but not $L_-$, and if we ``turn it upside down'' we obtain another that contains $L_-$ but not $L_+$. But clearly there are no disjoint saturated open sets that contain $L_+$ and $L_-$ separately. We conclude that the leaf space is $T_1$ but not Hausdorff as claimed (compare with Proposition \ref{prop: submersion leaf space}).
\end{remark}

\begin{theorem}\label{teo: transverse disprisonment}
Suppose $(M, \mathcal{F}, g_\intercal)$ is transversely strongly causal and let $\gamma\colon [a,b)\rightarrow M$ be a future-directed, transversely right-inextendible, transversely causal curve such that $\gamma(a)\in K$, where $K\subset M$ is transversely compact. Then there exists some $s_0\in (a,b)$ such that $s\geq s_0\implies \gamma(s)\notin K$.
\end{theorem}
\begin{proof}
    Suppose the conclusion is false. Then either $\gamma$ is entirely contained in $K$ or it persistently returns to $K$. In either case, we have a sequence $(s_i)_{i\in\mathbb{N}}\subset [a,b)$ such that $s_i\rightarrow b$ and $(\gamma(s_i))_{i\in\mathbb{N}}\subset K$. Since $K$ is transversely compact, $\pi_\f(K)$ is compact in $M/\f$ and there exists a subnet $(s_\lambda)_{\lambda \in \Lambda}$ such that $\pi_\f\circ \gamma(s_\lambda)\rightarrow L_0$ for some leaf $L_0\in \f$, where $\Lambda$ is some partially ordered, directed index set. Let $\mathcal{U}\supset L_0$ be any saturated open set. Using transverse strong causality in the form of Proposition \ref{prop: strong causality characterization}(ii), we choose some open set $L_0\in \widehat{V}\subset \pi_\f(\mathcal{U})$ in $M/\f$ such that $\forall L_1,L_2\in \widehat{V}$ and for any $L'\in \f$
      $$L_1 <_{M/\f} L'<_{M/\f} L_2 \Longrightarrow L'\in \pi_\f(\mathcal{U}).$$ 
      Since for some $\lambda_0\in \Lambda$ we do have $ \pi_\f\circ \gamma(s_{\lambda})\in \widehat{V}$ whenever $\lambda \geq \lambda_0$, given $s\in (a,b)$ with $s\geq s_{\lambda_0}$, there exists $\lambda_1\geq \lambda_0$ such that $s_{\lambda_1}> s$. Hence 
     $$ \pi_\f\circ \gamma(s_{\lambda_0})\leq_{M/\f} \pi_\f\circ \gamma(s)\leq_{M/\f} \pi_\f\circ \gamma(s_{\lambda_1}) \Longrightarrow \pi_\f\circ \gamma(s)\in \pi_\f(\mathcal{U}),$$ 
     that is $\gamma[s_{\lambda_0},b) \subset \pi^{-1}_\f(\pi_\f(\mathcal{U}))\equiv \mathcal{U}$ since the latter set is saturated. We conclude that $L_0$ is transverse limit to the right for $\gamma$, contrary to our assumption that $\gamma$ is transversely right-inextendible. 
\end{proof}

\subsection{Transverse Lorentzian causality versus standard causality}\label{extrinsic}

 A question that not only will be of practical importance for us in the final section of this paper, but that has indeed independent interest is this: given a time-oriented Lorentzian foliation $(M,\f,g_\intercal)$, and a Lorentzian bundle-like metric $g$ associated with $g_\intercal$ and with compatible time-orientation, how exactly does the transverse causal structure of $(M,\f,g_\intercal)$ is related to the standard causal structure of the spacetime $(M,g)$? 
 
We shall pursue this issue in more detail in this subsection. Recall that as we impose $\mathrm{ind}(g)=1$, we necessarily have that all the leaves of $\mathcal{F}$ are spacelike submanifolds with respect to $g$. The issue at hand has been investigated in particular situations. For example, the causality of Lorentzian warped products is well-known. 
\begin{proposition}\cite[Props. 3.62, 3.62 and 3.68]{beem}\cite[Prop. 2.14]{minguzzi-warped}\label{prop:causality warped products}
    Let $B\times _f F$ be a Lorentzian warped product as in Example \ref{exe: such}. Then 
    \begin{itemize}
        \item[i)] $B\times _f F$ is chronological [resp. causal, distinguishing, strongly causal] if and only if $(B,g_B)$ is a chronological [resp. causal, distinguishing, strongly causal] spacetime;
        \item[ii)] $B\times_f F$ is globally hyperbolic if and only if $(B,g_B)$ is a globally hyperbolic spacetime and the Riemannian manifold $(F,g_F)$ is complete.
    \end{itemize}
\end{proposition}

G. Walschap \cite{walschap} has shown that these results basically generalize to the more general case of Lorentzian submersions $\pi\colon  (M,g)\rightarrow (N,h)$ [conf. Example \ref{exe: such}]. Of course, it is immediate that $(N,h)$ chronological [resp. causal] implies that $(M,g)$ is also chronological [resp. causal]. Relative strong causality too follows without further ado. However, for the globally hyperbolic context, Walschap felt the need of imposing additional geometric constraints on the fibers apart from completeness: that the submersion has bounded geometry \cite[Def. 3.3]{walschap}. 

\begin{proposition}\cite[Thms. 3.2 and 3.6]{walschap}\label{prop: causality lorentzian submersions}
    Let $\pi\colon  (M,g)\rightarrow (N,h)$ be a Lorentzian submersion. The following statements hold. 
    \begin{itemize}
        \item[i)] If $(N,h)$ is strongly causal, then so is $(M,g)$.
        \item[ii)] If $(N,h)$ is globally hyperbolic, all fibers of $\pi$ are complete and $\pi$ has bounded geometry, then $(M,g)$ is also globally hyperbolic.
    \end{itemize}
\end{proposition}
In discussing our generalizations of these results to Lorentzian foliations we too shall need to impose further conditions in appropriate places. The case of the chronological/causal/distinguishing conditions, however, are simple enough. 
\begin{proposition}\label{prop: tranverse chron implies chron}
If $(M,\mathcal{F}, g_\intercal)$ is transversely chronological [respectively transversely causal, transversely future/past distinguishing], then $(M,g)$ is chronological [resp. causal, future/past distinguishing]. Moreover, all the leaves of $\f$ are achronal [resp. acausal] immersed spacelike submanifolds in $(M,g)$. 
\end{proposition}
\begin{proof}
If $\gamma$ is a timelike [resp. causal] curve starting and ending on the same leaf (perhaps at the same point), then we already know it is transversely timelike [resp. transversely causal] curve by Lemma \ref{lemma: comparing causal types}. 

The future/past distinguishing case is slightly more elaborate, and again we only do the future case. Suppose $(M,\mathcal{F}, g_\intercal)$ is transversely future-distinguishing, and let $x,y \in M$ such that 
$$x\leq_g y \text{ and } x\in \overline{J^+_g(y)}.$$
We claim that 
$$L_x\leq_{M/\f}L_y \text{ and } L_x\in \overline{J^+_{M/\f}(L_y)}.$$
If this is the case, then the transverse future-distinguishing condition would imply $L_x =L_y = L$. Thus, if we had $x\neq y$, then $L<_{M/\f}L$, violating Proposition \ref{prop: cute fact}(a). We would then conclude that $x=y$, which in turn would mean that $(M,g)$ is future-distinguishing, as desired. 

We thus proceed to prove the claim. Now, we immediately have $L_x\leq_{M/\f}L_y$. Let $\mathcal{U}\ni L_x$ be an open set $M/\f$, and let $\widetilde{U}=\pi_{}^{-1}(\mathcal{U})$ be the associated saturated open subset of $M$. Since $x\in \widetilde{U}\cap \overline{J^+_g(y)} $. we can pick $x'\in \widetilde{U}\cap J^+_g(y)$. But then, since $\widetilde{U}$ is saturated, $L_{x'}\subset \widetilde{U}$; furthermore,
$$L_{x'}\in J^+_{M/\f}(L_y) \text{ and } L_{x'} \in \mathcal{U},$$
and we conclude that $L_x\in \overline{J^+_{M/\f}(L_y)}$, so the proof is complete. 
\end{proof}

\begin{comment}
\begin{example}\label{exe: distinguishing nao vai}
\color{teal}Não está claro para mim porque é transversalmente  distinguidor. De fato, fico  com a impressão que o espaço de folhas é justamente o exemplo clássico de espaço-tempo que é causal mas não distinguidor.\color{black}
 \textcolor{red}{[VERDADE, REALMENTE NÃO FUNCIONA ESTE EXEMPLO. VOU TENTAR PENSAR EM OUTRO.]}\textit{We may have a transversely distinguishing $(M, \mathcal{F}, g_\intercal)$ and a good topology on the leaf space with $(M,g)$ still not distinguishing.} To see this, start with the spacetime $(\mathbb{R}^3, (\cosh \, t -1)^2(dx^2-dt^2) -dt\,dx +dy^2)$ and the foliation $\widetilde{\f}$ given by the fibers of $\pi\colon (t,x,y) \in \mathbb{R}^3\mapsto (t,x) \in \mathbb{R}^2$. Now, identify $(t,x,y) \sim (t,x+m,y+m^\prime), \forall m, m^\prime\in\mathbb{Z}$, which is clearly isometric, to obtain a spacetime with topology $\mathbb{R}\times \mathbb{T}^2$ foliated by spacelike circles, hence with leaf space homeomorphic to $\mathbb{R}\times\mathbb{S}^1$, with a Hausdorff topology. This spacetime is not causal because all curves with $t=0$ and $y= \text{ const.}$ are closed null curves. But if we delete some circle $x= \text{ const.}$ for $t=0$ it becomes causal but not distinguishing. Nonetheless, it is transversely distinguishing with respect to the associated transverse Lorentzian metric $(\cosh \, t -1)^2(dx^2-dt^2) -dt\,dx$ 
\end{example}
\end{comment}
As we saw in the Remark \ref{rmk: so cool counterexamples}, if the leaf space topology is bad, then transverse strong causality by itself can do little to control the causality of the ambient spacetime. But if leaf space topology is good enough, then that issue disappears. 
\begin{proposition}\label{prop:strong causality does happen}
If $(M, \mathcal{F}, g_\intercal)$ is transversely strongly causal and if $M/\mathcal{F}$ is Hausdorff, then $(M, g)$ is strongly causal.
\end{proposition}
\begin{proof}
Suppose that $(M,g)$ is not strongly causal. In that case there exist a point $x_0\in M$, an open set $x_0\subset \mathcal{U}\subset M$ and a sequence $(\gamma_k\colon [0,1]\rightarrow M)_{k\in\mathbb{N}}$ of future-directed causal curves such that $\gamma_k(0),\gamma_k(1)\rightarrow x_0$ but $\gamma_k(1/2)\notin \mathcal{U}, \forall k$. Shrinking $\mathcal{U}$ and reparametrizing the curves if necessary we can assume that $\overline{\mathcal{U}}$ is compact and contained in some normal convex set, with $\gamma_k(1/2)\in \partial \mathcal{U}, \forall k$. Thus, up to passing to a subsequence we can assume that $\gamma_k(1/2)\rightarrow y_0 \in \partial \mathcal{U}$ and $x_0 <_g y_0$ \cite[Lemma 2.14]{oneillbook}, which in turn implies $L_{x_0} <_\f L_{y_0}$. However, recall that since topology of the leaf space is in particular $T_1$, $(M, \mathcal{F}, g_\intercal)$ is transversely causal by Proposition \ref{prop: cute fact}(b); thus we certainly have $L_{x_0}\neq L_{y_0}$ and by the Hausdorff property on the leaf space there exist open, saturated, disjoint sets $\widetilde{V}\supset L_{x_0}$ and $\widetilde{W}\supset L_{y_0}$. By the transverse strong causality there must exist some open set $L_{x_0}\subset \widetilde{O}\subset \widetilde{V}$ such that any transversely causal curve segment with endpoints in $\widetilde{O}$ is all within $\widetilde{V}$. But for large $k\in \mathbb{N}$ we clearly have $\gamma_k(0),\gamma_k(1)\in \widetilde{O}$, and hence $\gamma_k[0,1]\subset \widetilde{V}$, while $\gamma_k(1/2)\in \widetilde{W}$, a contradiction.
\end{proof}

We now turn to investigation of global hyperbolicity of bundle-like metrics, which will be of special importance for us in the next section. We shall see that it arises at the price of extra assumptions on the topology of the leaves. 
\begin{lemma}\label{inextransvinex}
    Let $\gamma\colon [a, b)\rightarrow M$ be a right-inextedible causal curve. If $(M,g)$ is strongly causal and the leaves of $\mathcal{F}$ are compact, then $\gamma$ is transversely right-inextendible.
\end{lemma}
\begin{proof}
Suppose $\gamma$ is not transversely right-inextendible. Therefore, there is $L\in \mathcal{F}$ such that $L$ is a transverse limit to the right for $\gamma$. Let $h$ be an auxiliary complete Riemannian metric on $M$ and let $d_h$ denote its associated metric distance. Since $L$ is a transverse limit to the right for $\gamma$, given any $\varepsilon>0$, the compact set $$\{x\in M\, : \, d_h(x,L)\leq \varepsilon\}$$ 
contains the final portion of $\gamma$, which cannot happen since $\gamma$ is right-inextendible and $(M, g)$ is strongly causal.
\end{proof}

\begin{theorem}\label{teo: global hyperbolic does happen}
Suppose $(M, \mathcal{F}, g_\intercal)$ is transversely globally hyperbolic, $M/\f$ is Hausdorff and $\mathcal{F}$ has compact leaves. Then for any bundle-like Lorentz metric $g$ associated with $g_\intercal$, $(M,g)$ is globally hyperbolic.
\end{theorem}
\begin{proof}
Suppose that $(M,g)$ is not globally hyperbolic. Since Prop. \ref{prop:strong causality does happen} yields that $(M,g)$ is strongly causal, there must exist some $x, y\in M$ such that the causal diamond $J(x,y):=J^+(x)\cap J^-(y)$ is not compact. But in that case we can find some future-directed, right-inextendible causal curve $\gamma\colon [0, b)\rightarrow M$ entirely contained in $J(x,y)$. Now, on the one hand, by Lemma \ref{inextransvinex} $\gamma$ is transversely right-inextendible. On the other hand, by the transverse global hyperbolicity hypothesis, the causal prism $J_\intercal(L_x,L_y):= J_\intercal^+(L_x)\cap L^-_\intercal(L_y)$ is transversely compact, and since $\gamma$ is also transversely causal we have that $\gamma[0,b)\subset J(x,y)\subset J_\intercal(L_x,L_y)$, in violation of Theorem \ref{teo: transverse disprisonment}. 
\end{proof}

\begin{remark}
By deleting a point from $\mathbb{R}^2\times \mathbb{T}^2$ with a suitable flat Lorentzian metric we easily find a counterexample to the previous result (with a Hausdorff leaf space) when we drop compactness of fibers. A topological restriction on the space of leaves is also clearly necessary: the second example in Remark \ref{rmk: so cool counterexamples} is actually transversely globally hyperbolic, while the associated spacetime is not even chronological.    
\end{remark}

\section{A transverse diameter theorem}\label{mainsec}

Recall that the classic Bonnet--Myers theorem in Riemannian geometry establishes that if $(M^n,h)$ is a complete Riemannian manifold such that $\ric_h \geq (n-1)C$, for some positive constant $C>0$, then $(M,h)$ has finite diameter 
$$\diam(M,h) \leq \pi/\sqrt{C},$$
so in particular $M$ is compact. Moreover, since the geometric assumptions on $(M,h)$ also apply to its universal cover, the latter is also compact, so $\pi_1(M)$ is finite. There is also a well-known \textit{transverse} version of the Bonnet-Myers theorem for Riemannian foliations due to Hebda \cite{hebda}.

Now, there is an also well-known \textit{Lorentzian} analogue of the Bonnet-Myers result due to Beem and Ehrlich, the so-called \textit{timelike diameter theorem}:
\begin{theorem}\cite[Thm. 11.9]{beem}\label{teo: lorentzian diameter}
 Let $(M^n,g)$ be a globally hyperbolic spacetime and suppose $\ric_g(v,v)\geq (n-1)C$ for any unit timelike $v\in TM$ and some positive constant $C>0$. Then its \emph{timelike} diameter is finite. More precisely,
 $$\diam(M,g) \leq \pi/\sqrt{C}.$$
\end{theorem}
Although Theorem \ref{teo: lorentzian diameter} is formally similar to its Riemannian counterpart, its geometric interpretation is strikingly different. Since our goal in this section is to establish a transverse analogue for Lorentzian foliation of the timelike diameter theorem, it is worthwhile to review that interpretation here. We do so only briefly and refer to \cite[Chs. 4 and 11]{beem} for further details. 

Let $(M,g)$ be a spacetime. Given a piecewise $C^1$ curve segment $\alpha\colon [a,b]\rightarrow M$, its \textit{Lorentzian length} is given by
$$\ell_g(\alpha):=\int_a^b\sqrt{|g(\alpha'(s),\alpha'(s))|}ds.$$
The Lorentzian length is particularly significant for a timelike $\alpha$: in general relativity and other geometric theories of gravity, it measures the \textit{proper time} between the events $\alpha(a)$ and $\alpha(b)$ as reckoned by an observer whose worldline is described by $\alpha$. 

Using the Lorentzian length of causal curves one also defines the so-called \textit{Lorentzian distance function} $d_g\colon  M\times M\rightarrow [0,+\infty]$ as follows. For $x,y \in M$, we put 
   \begin{multline*} 
   d_g(x,y):= \sup \{\ell_g(\alpha) \, ;\,  \alpha\colon [a,b]\rightarrow M\\
\text{future-directed causal curve, $\alpha(a)=x,\alpha(b)=y$}\}, \end{multline*}
provided $x<_gy$. If $y\notin J^+_g(x)$ then we put $d_g(x,y)=0$. Despite its name, the Lorentzian distance function is \textit{not} a distance in the metric space sense, as in the Riemannian context\footnote{Because of its peculiar features and the physical interpretation of the Lorentzian length of timelike curves as a proper time, $d_g$ is sometimes also called \textit{time-separation function}.}. In particular, it is almost never the case that $d_g(x,y)=d_g(y,x)$ and the Lorentzian distance between two points can perfectly well be infinite-valued. (For example we have $d_g(x,x)=+\infty$ for \textit{every} $x\in M$ if $(M,g)$ is totally vicious.) 
In addition, we have that $d_g(x,y) >0$ if and only if $x\ll_g y$.  

The \textit{timelike diameter} of the spacetime $(M,g)$ is thus defined as 
$$\diam(M,g) := \sup \{d_g(x,y) \, :\, x,y \in M\}.$$
Observe that we always have $\diam(M,g)>0$, but if $\diam(M,g) <+\infty$ then \textit{every} timelike geodesic has finite Lorentzian length, and is therefore incomplete. In addition, $M$ can \textit{never} be compact in this case, because the compactness of $M$ implies the existence of some $x\in M$ with $d_g(x,x)=\infty$. These features are of course in stark contrast with their Riemannian counterparts. 

\subsection{Contrasting Ricci and transverse Ricci bounds}\label{section: constrasting ricci and transverse ricci}
As announced before, in order to obtain a transverse timelike diameter result, we shall have to impose a suitable bound on the \textit{transverse} Ricci tensor of a given time-orientable Lorentzian foliation along transversely timelike directions\footnote{Observe that since the transverse Ricci tensor is $\f$-basic, this is equivalent to imposing bounds along the $g$-timelike directions for any associated Lorentzian bundle-like metric $g$.}. From the outset, we ought to discuss the logical relation between such a bound and analogous bounds on the (full) Ricci tensor. In this subsection we establish via explicit examples the following logical independence result.

\begin{proposition}\label{prop: ricciversustransversericci}
  There exist examples of time-orientable Lorentzian foliations $(M,\f,g_\intercal)$ with associated timelike bundle-metric $g$ such that, respectively, 
  \begin{itemize}
      \item[1)] there exist positive constants $A,C>0$ such that for any $g$-unit timelike vector $v\in TM$ we have $\ric_g(v,v) \geq C$ but $\ric_\intercal(v,v)\leq -A$;
      \item[2)] there exists a positive contant $C>0$ such that for any $g$-unit timelike vector $v\in TM$ we have $\ric_\intercal(v,v) \geq C$, but there exists a $g$-timelike vector $v_0\in TM$ such that $\ric_g(v_0,v_0)<0$. 
  \end{itemize}
\end{proposition}
The two classes of examples we shall build here use warped products. Specifically, we shall take two distinct convenient choices of a $p$-dimensional connected Riemannian manifold $(F,g_F)$, a $q$-dimensional ($q\geq 2$) spacetime $(B,g_B)$ and a warping function $f\in C^\infty(B)$. We then shall take $M=B\times F$ with metric 
$$g= \pi^\ast g_B + (f\circ \pi_B)^2 \pi_F^\ast g_F,$$
which is as we know bundle-like for the foliation $\f$ whose leaves are the fibers $\{b\}\times F$. 

Let $T\in TM\cong TB\oplus TF$ be a unit $g$-timelike vector. We decompose it in vertical ($T^V$) and horizontal ($T^H$) parts and use the standard formulas given in \cite[Cor. 7.43]{oneillbook}: 
\begin{align}\label{oneill formula for ricci}
	\ric_g(T,T) &= \ric^F( T^V,T^V) - g(T^V,T^V) f^{\#} \\
		   \notag  &\quad + \ric^B (T^H,T^H) - \frac{p}{f} \mathrm{Hess}^B_f(T^H,T^H),
\end{align}
where quantities indicated by superscripts $F$ and $B$ on the upper right correspond to lifts to $(M,g)$ of the corresponding quantities on $F$ and $B$, respectively, and 
$$f^{\#}:= \left(\frac{\Delta_B f}{f}+(p-1)\frac{g_B(\nabla_Bf,\nabla _Bf)}{f^2}\right)\circ \pi_B.$$

\begin{example}
In our example establishing the situation in Prop. \ref{prop: ricciversustransversericci}(1), we fix $(F,g_F)$ to be any connected compact manifold such that $p\geq 2$ and 
\begin{equation}\label{boundonF}
    \ric_F\geq 0, 
\end{equation}
such as for example a round sphere or some compact quotient of the Euclidean space. As $(B,g_G)$, however, we make a more careful choice. Take $B=(0,L)\times \mathbb{S}^3$ where $L>0$ is to be specified later on, with the \textit{de Sitter} metric in the form 
\begin{equation}\label{eqdesitter}
g_B = -dt^2 + a(t)^2\cdot \omega_3,    
\end{equation}
where $\omega_3$ denotes the standard round metric on the $3$-sphere and $a(t):=\cosh\, t$. Time-orientation is taken so that $\partial_t$ is future-directed. Thus, \cite[Thm. 3.68]{beem} ensures that $(B,g_B)$ is globally hyperbolic. Moreover, a direct calculation shows that given any vector $v\in TB$ we have 
\begin{equation}\label{usualdesitterbound}
\ric_B(v,v)= 3\cdot g_B(v,v).
\end{equation}
In particular, then, for the $g$-unit timelike vector $T\in TM$ we have 
$$\ric_\intercal(T,T) \leq -3.$$
Now, on $B$ we choose the smooth warping function 
$$f(t,x) := e^{\phi(t)}, \quad \phi(t) = \log \left(t/2\right).$$
Therefore, a straightforward computation gives (omitting composition on the right with $\pi_B$ for notational simplicity)
\begin{equation}\label{f ast}
  f^{\#} = -\left[\frac{(p-2)}{4t^2} + \frac{3 a(t)\dot{a}(t)}{2t}\right]. 
\end{equation}
If we write $T^H= \lambda \cdot \partial_t +x$ where $\partial_t$ and $x$ indicate both vectors on $B$ with $x$ on the sphere and their horizontal lifts (again by notational simplicity), another straightforward calculation yields
\begin{equation}\label{hessianonb}
   \frac{\mathrm{Hess}^B_f(T^H,T^H)}{f} = - \frac{\lambda^2}{4t^2} -\frac{\omega_3(x,x)}{2t}.
\end{equation}
Therefore, plugging eqs. \eqref{boundonF}, \eqref{usualdesitterbound}, \eqref{f ast} and \eqref{hessianonb} into \eqref{oneill formula for ricci} we have
$$\ric_g(T,T) \geq \left(\frac{p}{4t^2} - 3\right) \lambda^2,$$
and 
$$g(T,T) =-1 \Longrightarrow \lambda^2 \geq 1,$$
so provided we choose $L\leq \sqrt{p}/4$ we end up with $\ric_g(T,T) \geq 1,$ thus obtaining the desired conclusion.
\end{example}
\begin{example}\label{bounded-diameter}
We now choose $(F^p, g_F)$ as some compact quotient of the $p$-dimensional hyperbolic space $\mathbb{H}^p$ by some discrete subgroup of isometries, so it has constant negative curvature $k=-1$. Take $B=(1,e)\times \mathbb{R}^{q-1}$ with metric 
$$g_B= -dt^2 + \delta^{q-1}$$
where $\delta^{q-1}$ is the Euclidean (flat) metric. On $B$ we consider the warping function 
$$f=f(t) = \log \, t, \quad 1<t<e.$$

Again, $(B,g_B)$ is globally hyperbolic, and thus so is the full warped product $(M,g)$. Moreover, since the $\mathbb{R}^{q-1}$ is flat and appears via a direct product, it will give no contribution to the either the Ricci tensor or its transverse version. Thus, let $E=\frac{\partial}{\partial t}+X$, for (the horizontal lift of]) some $X\in\mathfrak{X}(F)$. Then $g\left(E,E\right)=-1+g_F(X,X)\log^2(t)$, and therefore $E$ is timelike iff $g_F(X,X)<\frac{1}{\log^2(t)}$. In this case, on the one hand we have 
$$\ric_\intercal(E,E)=\ric_B\left(\frac{\partial}{\partial t},\frac{\partial}{\partial t}\right)=\frac{p}{t^2\log(t)}\geq \frac{p}{e^2}>0.$$ 
On the other hand, the Ricci tensor of the warped product is given by (see \cite[Cor. 12.10]{oneillbook})\begin{align*}
    \ric\left(E,E\right)&=-p\frac{f^{\prime\prime}}{f}+\left[(p-1)\left(\frac{f^\prime}{f}\right)^2+(p-1)\frac{k}{f^2}+\frac{f^{\prime\prime}}{f}\right]g_F(X,X)\\
&=\frac{p}{t^2\log(t)}+\left[\frac{(p-1)}{t^2\log^2(t)}+\frac{(p-1)k}{\log^2(t)}-\frac{1}{t^2\log(t)}\right]g_F(X,X)\\
&=\frac{p\,\log(t)+((p-1)(1+kt^2)-\log(t))g_F(X,X)}{t^2\log^2(t)}.
\end{align*}
Now, recall that $k=-1$, and since $1/\log^2 \, t> 1$ for $t\in (1,e)$, we can pick $X$ so that for some $x_0\in N$ we have $g_F(X(x_0),X(x_0))) =1$. For such $x_0$ and for some $t_0$ near enough $e$, $E(z_0)$ is timelike and $\ric_g(E(z_0),E(z_0))<0$ where $z_0=(t_0,0,x_0)$. 
\end{example}

\subsection{Transverse diameter}\label{finalsec}
After these preliminaries, we now turn to our announced transverse version of the Lorentzian timelike diameter theorem. In a first moment we discuss a few ideas and results which are valid on any semi-Riemannian foliation, and later on we specialize to Lorentzian foliations. So let $(M,\mathcal{F}, g_\intercal)$ be any semi-Riemannian foliation. The \textit{transverse length} of a piecewise $C^1$ curve segment $\alpha\colon [a,b]\rightarrow M$ is given by
$$\ell_\intercal(\alpha):=\int_a^b\sqrt{|g_\intercal(\alpha'(s),\alpha'(s))|}ds.$$

\begin{proposition}\label{prop: tocomparelengths} Let $(M,\mathcal{F}, g_\intercal)$ be a codimension $q$  semi-Riemannian foliation of index $\mu$. Given piecewise $C^1$ curve segments $\alpha,\beta\colon [a,b]\rightarrow M$, the following properties of the transverse length hold.  \begin{enumerate}\label{length-facts}
\item If $\alpha[a,b] \subset L$ for some leaf $L\in \f$, then $\ell_\intercal(\alpha)=0$.
\item If $\pi_\mathcal{F}\circ\alpha=\pi_\mathcal{F}\circ\beta$, then $\ell_\intercal(\alpha)=\ell_\intercal(\beta)$.
\item If $g$ is a bundle-like metric associated with $g_\intercal$ and $\alpha$ is $g$-causal, then $\ell_g(\alpha)\leq\ell_\intercal(\alpha)$, and equality holds if and only if $\alpha$ is $g$-horizontal.
\end{enumerate}
\end{proposition}

\begin{proof} Items (1) and (3) follow easily from the definitions and the identity \eqref{basic rel}. Let us prove (2). Using Prop. \ref{ prop: basic 2}, we can pick a Haefliger cocycle $(U_i,\pi_i, \gamma_{ij})_{i\in I}$ on $M$ together with semi-Riemannian metrics $h_i$ of the same index $\mu$ defined on $\pi_i(U_i)\subset \mathbb{R}^q$ for each $i\in I$ such that 
       $$\gamma_{ij}^\ast (h_i|_{\pi_i(U_i\cap U_j)}) = h_j|_{\pi_j(U_i\cap U_J)} \quad \forall i,j \in I,$$
   that is, the transition maps $\gamma_{ij}$ are isometries. By shrinking the $U_i$ if needed we can assume that each $\pi_i$ has connected fibers; observe in addition that $$\pi_i^\ast h_i = g_\intercal$$ on $U_i$ for each $i\in I$.
   
   Fix $t_0 \in [a,b]$, and let $i_0 \in I$ such that $\alpha(t_0)\in U_{i_0}$. We claim that given any $x\in U_{i_0}$, there exists some number $\varepsilon = \varepsilon(x,t_0)>0$ and a piecewise $C^1$ curve segment $\widehat{\alpha}_{t_0,x}\colon [t_1,t_2]\rightarrow U_{i_0} $ with $t_1<t_2$, $t_0 \in [t_1,t_2] \subset [a,b] $ and $|t-t_0|< \varepsilon, \forall t\in [t_1,t_2]$ for which 
   $$\widehat{\alpha}_{t_0,x}(t_0) =x,\, \ell_\intercal(\widehat{\alpha}_{t_0,x}) = \ell_\intercal(\alpha|_{[t_1,t_2]}) \text{ and } \pi_\f\circ \widehat{\alpha}_{t_0,x} = \pi_\f \circ \alpha|_{[t_1,t_2]}.$$
   Indeed, by continuity of $\alpha$ we first pick $\varepsilon>0$ be such that $\alpha(t) \in U_{i_0}$ whenever $t\in [a,b]$ is such that $|t-t_0|<\varepsilon$. Then let $\sigma_0\colon  V\subset \pi_{i_0}(U_{i_0}) \subset \mathbb{R}^q \rightarrow U_{i_0}$ be a smooth local section of $\pi_i$ with $\pi_{i_0}(\alpha(t_0)) \in V$, $\sigma_0(\pi_{i_0}(\alpha(t_0))) =x$ and 
   $$\pi_{i_0}\circ \sigma_0 = \mathrm{Id}_{V}.$$
   By shrinking $\varepsilon >0$ if necessary we can assume the existence of $\varepsilon$-close-to-$t_0$ numbers $t_1<t_2$ in $[a,b]$ such that $t_0\in [t_1,t_2]$and $\pi_{i_0}\circ \alpha[t_1,t_2]\subset V$, and thus define 
   $$\widehat{\alpha}_{t_0,x}(t):= \sigma_0\circ \pi_{i_0}\circ \alpha(t), \quad \forall t\in [t_1,t_2].$$ 
   Now, $$\sigma_0^\ast g_\intercal = \sigma_0^\ast \pi_{i_0}^\ast h_{i_0}=(\pi_{i_0}\circ \sigma_0)^\ast h_{i_0} \equiv h_{i_0},$$
   whence we conclude that
   $$\ell_\intercal(\widehat{\alpha}_{t_0,x})= \ell_{h_{i_0}}(\pi_{i_0}\circ \alpha|_{[t_1,t_2]}) = \ell_\intercal(\alpha|_{[t_1,t_2]}).$$
   Furthermore,
   $$\pi_{i_0}\circ \widehat{\alpha}_{t_0,x} = \pi_{i_0}\circ \alpha|_{[t_1,t_2]} \Longrightarrow \pi_\f\circ \widehat{\alpha}_{t_0,x} = \pi_\f \circ \alpha|_{[t_1,t_2]}.$$ 
   This proves the claim. 

   Now, recall we are assuming in particular that 
   $$L_{\beta(t_0)} = L_{\alpha(t_0)}.$$
   Pick any continuous curve $\gamma\colon [0,1] \rightarrow M$ whose image is contained in this common leaf and such that $\gamma(0) =\alpha(t_0), \gamma(1)=\beta(t_0)$, and let $U_{i_0}, U_{i_1}, \dots, U_{i_k}$ of the simple sets in our Haefliger cocycle whose union covers $\gamma[0,1]$. We also can assume that $U_{i_{j-1}}\cap U_{i_j}\neq \emptyset$ for each $j\in \{1,\ldots, k\}$, $\alpha(t_0) \in U_{i_0}$ and $\beta(t_0) \in U_{j_0}:= U_{i_k}$. By applying our claim repeatedly along points $x_j=\gamma(s_j)\in U_{i_{j-1}}\cap U_{i_j}$ we conclude that there exist a number $\varepsilon_{t_0}>0$, and a piecewise $C^1$ curve segment $\widehat{\alpha}_{t_0}\colon [s_1,s_2]\rightarrow U_{j_0} $ with $s_1<s_2$, $t_0 \in [s_1,s_2] \subset [a,b] $ and $|s-t_0|< \varepsilon_{t_0}, \forall s\in [s_1,s_2]$ for which 
   \begin{align*}
   &\widehat{\alpha}_{t_0}(t_0) =\beta(t_0),\ \  \beta[s_1,s_2] \subset U_{j_0},\\
   & \ell_\intercal(\widehat{\alpha}_{t_0}) = \ell_\intercal(\alpha|_{[s_1,s_2]}) \text{ and } \pi_\f\circ \widehat{\alpha}_{t_0} = \pi_\f \circ \alpha|_{[s_1,s_2]}.\end{align*}
   But note that in that case we also have 
   \begin{equation}\label{a single step}
    \pi_\f\circ \widehat{\alpha}_{t_0}=\pi_\f \circ \beta|_{[s_1,s_2]}. 
   \end{equation}
   Since $\widehat{\alpha}_{t_0}$ is in the same simple set $U_{j_0}$ as $\beta|_{[s_1,s_2]}$ we can consider as in Proposition \ref{prop: submersion leaf space} the uniquely defined map $\overline{\pi_{j_0}}\colon  \pi_\f(U_{j_0}) \rightarrow \mathbb{R}^q$ that makes the diagram 
     $$
     \xymatrix{
 & U_{j_0}  \ar[dr]^-{\pi_{j_0}} \ar[dl]_{\pi_{\f}} & \\
\pi_\f(U_{j_0}) \ar[rr]_{\overline{\pi_{j_0}}} & & \mathbb{R}^q}
$$
commute. By applying $\overline{\pi_{j_0}}$ to equation \eqref{a single step} we end up with 
$$\pi_{j_0}\circ \widehat{\alpha}_{t_0}=\pi_{j_0} \circ \beta|_{[s_1,s_2]} \Longrightarrow \ell_\intercal (\widehat{\alpha}_{t_0})= \ell_\intercal(\beta|_{[s_1,s_2]}),$$
whence finally 
$$\ell_\intercal(\beta|_{[s_1,s_2]})=\ell_\intercal(\alpha|_{[s_1,s_2]}).$$
Since the fiducial point $t_0 \in [a,b]$ was taken arbitrarily, we have in effect established (by taking a Lebesgue number for the open cover $(t-\varepsilon_t,t+\varepsilon_t)_{t\in [a,b]}$ of the compact interval $[a,b]$ thus obtained) the existence of a partition $s_0=a< \cdots < s_m = b$ of the interval $[a,b]$ such that 
$$\ell_\intercal(\beta|_{[s_{j-1},s_j]})=\ell_\intercal(\alpha|_{[s_{j-1},s_j]}), \quad \forall j \in \{1,\ldots, m\},$$
and thus $\ell_\intercal(\beta)=\ell_\intercal(\alpha).$
\end{proof}

We now establish the presence of focal points of leaves along certain sufficiently long horizontal geodesics. Let $(M,\mathcal{F}, g_\intercal)$ be a semi-Riemannian foliation. We say that $v\in T_pM$ is \textit{transversely cospacelike} if either $(M,\mathcal{F}, g_\intercal)$ is a Riemannian foliation, or else $(M,\mathcal{F}, g_\intercal)$ has index 1 and $v$ is transversely timelike. A cospacelike vector $v\in TM$ is a \textit{$g_\intercal$-unit} vector if $|g_\intercal(v,v)|=1$.

\begin{theorem}\label{theorem: transverse focal points}
Let $(M^n,\mathcal{F}, g_\intercal)$ be a codimension $q\geq 2$ semi-Riemannian foliation of index $\nu\in\{0,1\}$, and suppose $g$ is an associated bundle-like semi-Riemannian metric of index $\mu\geq\nu$ on $M$. If there exists some constant $C>0$ such that $\ric_\intercal(v,v)\geq (q-1)C$ on any $g_\intercal$-unit cospacelike vector $v\in 
TM$, then any unit speed, $g$-horizontal cospacelike $g$-geodesic segment $\gamma\colon [a,b]\rightarrow M$ with $\ell_\intercal(\gamma)\geq \pi/\sqrt{C}$ admits a focal point of the leaf $L_{\gamma(a)}$. 
\end{theorem}
\begin{proof}
We will adapt the construction in the proof of \cite[Proposition 5]{hebda} to our setting. Endow $M$ with an auxiliary Riemannian metric $h$ and let $\nu\gamma\to [a,b]$ denote the normal bundle of $\gamma$ with respect to $h$. Then, since $[a,b]$ is compact, there exists some $\varepsilon>0$ such that $\exp_h^\perp\colon \nu_\varepsilon\gamma\to M$ is a local diffeomorphism\footnote{In order to simplify the structure of the foliation around the geodesic segment as a whole, it is more convenient to deal with a tubular neighborhood around it rather than a mere collection of simple open sets. That is, such a tubular neighborhood acts here as an extended analogue of a simple open set.} onto its image $U=\mathrm{Tub}_\varepsilon(\gamma)$, where $\nu_\varepsilon\gamma=\{v\in\nu\gamma \,:\, \Vert v\Vert_h<\varepsilon\}$. By further shrinking $\varepsilon$ if necessary, we can assume that $\f|_V$ is a simple foliation (i.e., its leaves are the fibers of a submersion) on each open set $V\subset U$ onto which $\exp_h^\perp$ restricts to a diffeomorphism: this is possible since $\gamma([a,b])$ can be covered by finitely many $\f$-simple open sets, and then we can take $\varepsilon$ small enough so that $U$ is contained in their union. Hence, if we consider $\mathcal{G}:=(\exp_h^\perp)^*(\f|_U)$, then by construction $\mathcal{G}$ is a simple foliation, given by the fibers of a submersion $\pi\colon \nu_\varepsilon\gamma\to B$, where $B$ is locally isometric to the total transversal $S_\Gamma$ (recall discussion in Subsection~(\ref{subsection: foliations})) for some Haefliger cocycle $\Gamma$ defining $\f$.

We can now endow $(\nu_\varepsilon\gamma,\mathcal{G})$ with the bundle-like semi-Riemannian metric $(\exp_h^\perp)^*g$ of index $\mu$ and drop the auxiliary metric $h$. Notice that $(\exp_h^\perp)^*g$ projects down to $B$ and it is either Riemannian or Lorentzian, according if $\nu$ is respectively $0$ or $1$. Hence $\pi\colon \nu_\varepsilon\gamma\to B$ becomes a semi-Riemannian submersion. By construction, it now suffices to show that there is a focal point of the leaf $J$ of $\mathcal{G}$ containing $(a,0)$ over the zero section $s$ of $\nu\gamma$, since $J$ is mapped on $\f_{\gamma(a)}$ and $s$ is mapped over $\gamma$ under $\exp_h^\perp$. Since $\pi\colon \nu_\varepsilon\gamma\to B$ is a semi-Riemmanian submersion and $s$ is a geodesic perpendicular to its fibers, $\pi\circ s$ is a geodesic on $B$ with the same causal character of $\gamma$, which is not null. Also, $J$-focal points over $s$ correspond directly to $\pi(J)$-conjugate points in $B$ over $\pi\circ s$ (see \cite[Theorem 4]{oneill2}: its Riemannian statement notwithstanding, it works exactly the same in a semi-Riemannian ambiance for horizontal geodesics).

Notice that, since $B$ is locally isometric to local quotients of $\f$, it follows from what we saw in Remark \ref{remark:basicLCconnectiononquotients} that $B$ satisfies $\ric_B\geq(q-1)C>0$ for all unit cospacelike vectors, hence, since
$$\ell_B(\pi\circ s)=\ell_{(\exp_h^\perp)^*g}(s)=\ell_g(\gamma)=\ell_\intercal(\gamma)\geq \pi/\sqrt{C},$$
it follows that there exists the required conjugate point in $B$, by \cite[Lemma~10.23]{oneillbook}.
\end{proof}

We define the \textit{transverse Lorentzian distance function} on $(M,\f,g_\intercal)$ as follows: for $L,L'\in\f$, we put  $d_\intercal(L,L')=0$ if $L'\not\subset J^+_\intercal(L)$ and, otherwise, if $L'\subset J^+_\intercal(L)$, we define
$$d_\intercal(L,L')=\sup\{\ell_\intercal(\gamma)\},$$
where the supremum ranges over the set of all piecewise $C^1$, future-directed, transversely causal curves starting on $L$ and ending on $L'$. Again, despite the name $d_\intercal$ is not a distance function in the usual sense; in fact, in view of Prop. \ref{prop: tocomparelengths} it should be thought of as the Lorentzian distance function of $M/\f$. Then, the \textit{transverse timelike diameter} of $(M,\f,g_\intercal)$ is
$$\diam_\intercal(M,\f,g_\intercal):=\sup\{d_\intercal(L,L')\, :\, L,L'\in\f\}.$$

\begin{remark}
It is immediate that, if $g$ is a bundle-like metric compatible with $g_\intercal$, then
$\diam(M,g)\leq\diam_\intercal(M,\f,g_\intercal)$.
\end{remark}

We are finally in position to state and prove the main results regarding the timelike diameter. We actually obtain two such results: one for the timelike diameter of $(M,g)$, and another one for the transverse timelike diameter of $(M,\f,g_\intercal)$, which as we pointed out can be seen as a timelike diameter for $M/\f$. Let us start with the timelike diameter of $(M,g)$. For that it will be useful to recall a notion first introduced by Galloway in \cite{galloway1986} (see also \cite{gallowayfcc2,treude}): a subset $C$ in a spacetime $(M,g)$ is \textit{future causally complete} (FCC) if for any $x\in J^+(C)$ the set $J^-(x)\cap C$ has compact closure in $C$. \textit{Past causally complete} (PCC) sets are defined time-dually. 

Observe that any FCC/PCC is necessarily closed in $M$. To see this, assume $C$ is (say) FCC and pick a sequence $(x_k) \subset C$ converging in $M$ to a point $x \in M$. We wish to show that $x\in C$. To that end, fix any $x' \in I^+(x)$. Since $I^{-}(x')$ is open in $M$ and contains $x$, eventually $x_k \in I^-(x')$. By the compactness of the closure $\overline{J^-(x')\cap C}^C$ in $C$, $(x_k)$ converges, up to passing to a subsequence, to a point $x'' \in C$. But then clearly $x=x''$, so $x\in C$ as desired. Moreover, it is not difficult to check that any compact subset in $M$ is both FCC and PCC. 
Finally, observe that when $(M,g)$ is globally hyperbolic, the sets $J^\pm(x)$ are closed in $M$ for all $x\in M$; thus, the sets of the form $J^-(x)\cap C$ are actually compact \textit{in $M$} for an FCC set $C$ and $x\in J^+(C)$, with an obvious time-dual statement when $C$ is PCC. 
\begin{theorem}\label{diameter}
Let $(M, g)$ be a globally hyperbolic spacetime. Let $\mathcal{F}$ be a codimension $q\geq2$ foliation on $M$ whose leaves are all future causally complete, spacelike submanifolds with respect to $g$, and such that $g$ is bundle-like with respect to $\f$. Suppose that $\ric_\intercal\geq (q-1)C>0$ on $g_\intercal$-unit transversely timelike vectors. Then
$$\diam(M, g)\leq\frac{\pi}{\sqrt{C}}.$$
\end{theorem}

\begin{proofcomment}
At this point, the proof is fairly simple. Suppose that there is a timelike curve $\alpha\colon [a,b]\rightarrow M$ with $\ell_g(\alpha)>\frac{\pi}{\sqrt{C}}$. Since $\alpha(b) \in J^+(L_{\alpha(a)})$ and the leaf $L_{\alpha(a)}$ through $\alpha(a)$ is FCC, given that $(M,g)$ is globally hyperbolic, the set $J^-(\alpha(b))\cap L_{\alpha(a)}$ is compact in $M$ by the above remarks. By standard causal-theoretic techniques (see, e.g., \cite[Section 3.2.4]{treude}) we obtain a future-directed timelike geodesic segment $\gamma$ starting at, and normal to $L_{\alpha(a)}$ (hence horizontal) and ending at $\alpha(b)$, whose Lorentzian length realizes the maximum Lorentzian distance between $\alpha(b)$ and $L_{\alpha(a)}$. Thus, $\ell_g(\gamma)\geq\ell_g(\alpha)$. But then Theorem \ref{theorem: transverse focal points} yields a contradiction, since $\gamma$ cannot contain focal points to $L_{\alpha(a)}$ before $\alpha(b)$.
\end{proofcomment}

The following technical fact will be useful and has independent interest. 
\begin{lemma}\label{amended}
Let $\alpha\colon [a,b]\rightarrow M$ be a regular, nowhere vertical curve without self-intersections. Then, there exists a smooth distribution $\mathcal{D}\subset TM$ such that $\mathbb{R}\alpha^\prime\subset \mathcal{D}$ and $TM=\mathcal{D}\oplus T\f$.
\end{lemma}
\begin{proof}
First, observe that since $\alpha$ is regular and has no self-intersections, $\alpha([a,b])$ is a 1-dimensional compact, properly embedded submanifold with boundary of $M$, and hence there exists a smooth vector field $Z\in \mathfrak{X}(M)$ such that $Z(\alpha(t)) = \alpha'(t), \forall t \in [a,b]$. Now, being non-vertical is an open condition, so there exists an open neighborhood $U \supset \alpha([a,b])$ in $M$ so that $Z$ is everywhere non-vertical in $U$. 

Pick a finite open covering $\{V_1, \ldots , V_k\}$ of $\alpha([a,b])$ in $M$ and Riemannian metrics $h_1, \ldots, h_k$ on the $V_1, \ldots, V_k$, respectively, with respect to which $Z|_{V_i}$ is $h_i$-orthogonal to $T\f|_{V_i}$, and $V_i\subset U$ for each $i\in \{1,\ldots, k\}$. This can be done, e.g., by choosing each $V_i$ to be the domain of a local frame of vertical sections $F_i= \{\sigma_1^i, \ldots, \sigma_{\text{rank}\, T\f}^i\}$ spanning $T\f$, and picking $h_i$ so that $Z|_{V_i}\cup F_i$ can be completed to an $h_i$-orthonormal frame of $TM|_{V_i}$. Finally, let $W \subset M$ be an open set such that 
$$\alpha([a,b]) \subset W\subset \overline{W} \subset \cup_{i=1}^kV_i.$$
Choose any Riemannian metric $h_0$ on $M\setminus \overline{W},$
and a smooth partition of unity $$\{\eta_0, \eta_1, \ldots, \eta_k\} \subset C^\infty(M)$$
subordinate to the open covering $\{V_0:= M\setminus \overline{W}, V_1, \ldots, V_k\}$ of $M$. 

By construction, $\alpha'$ is everywhere orthogonal to $T\f$ with respect to the Riemannian metric 
$$h:= \sum_{i=0}^k\eta_i h_i,$$
and thus the distribution  
$$\mathcal{D}:= (T\f)^{\perp_h}$$
possesses the desired properties. 
\end{proof}

We are ready to give a timelike diameter theorem in its purely transverse version.

\begin{theorem}\label{transverse-diameter}
Let $(M,\mathcal{F}, g_\intercal)$ be a codimension $q\geq 2$ transversely globally hyperbolic Lorentzian foliation with compact leaves, such that the leaf space $M/\f$ is Hausdorff. If $\ric_\intercal\geq (q-1)C$ on $g_\intercal$-unit timelike vectors, for some constant $C>0$, then $\diam_\intercal(M,\f, g_\intercal)\leq\frac{\pi}{\sqrt{C}}$.
\end{theorem}

\begin{proof}
Suppose there is a transversely timelike curve $\alpha\colon [a,b]\rightarrow M$ with $\ell_\intercal(\alpha)>\frac{\pi}{\sqrt{C}}$. Since the foliation is transversely globally hyperbolic, $\alpha$ never returns to a leaf. By Lemma \ref{amended} there is a distribution $\mathcal{D}\subset TM$ such that $\mathbb{R}\alpha^\prime\subset\mathcal{D}$ and $TM=\mathcal{D}\oplus T\mathcal{F}$. So, if $v\in TM$ we write $v=v_\mathcal{D}+v_\mathcal{F}$. Now, fix any Riemannian metric $h$ on $M$ and define $g(v,v):=g_\intercal(v_\mathcal{D},v_\mathcal{D})+h(v_\mathcal{F},v_\mathcal{F})$. It follows that $g$ is a Lorentzian bundle-like metric associated with $g_\intercal$ such that $\mathcal{D}$ is precisely its horizontal distribution. In particular, $\alpha$ is a $g$-horizontal curve. Hence, by Lemma \ref{length-facts} we have that $\ell_g(\alpha)=\ell_\intercal(\alpha)$. Also, since $g(\alpha^\prime,\alpha^\prime)=g_\intercal(\alpha^\prime,\alpha^\prime)<0$, $\alpha$ is a $g$-timelike curve, so if we suppose, without loss of generality, that $\alpha$ is future-directed, we have that $\alpha(b)\in I^+(L_{\alpha(a)})$. On the other hand, by Theorem \ref{teo: global hyperbolic does happen} we conclude that $(M,g)$ is globally hyperbolic. Also, the compactness of the leaves implies trivially that they are future causally complete, so we have by standard techniques the existence of a timelike geodesic segment $\gamma\colon [0,1]\rightarrow M$ with $\gamma(0)\in L_{\alpha(a)}$, $\gamma(1)=\alpha(b)$ maximizing the Lorentzian $g$-arclength between $L_{\alpha(a)}$ and $\alpha(b)$. We thus have $$
\ell_g(\gamma)\geq \ell_g(\alpha)=\ell_\intercal(\alpha)>\frac{\pi}{\sqrt{C}}.
$$ We also know that $\gamma$ is orthogonal to $L_{\alpha(a)}$ and by item $(iii)$ of Proposition~\ref{bundlelike equivalences} $\gamma$ is a horizontal geodesic. Therefore if we apply Theorem \ref{theorem: transverse focal points} to $\gamma$ we conclude that $\gamma$ has at least one focal point of $L_{\alpha(a)}$, which contradicts its maximality; thus we must have $\ell_\intercal(\alpha)\leq \frac{\pi}{\sqrt{C}}$.
\end{proof}

\begin{comment}
\begin{corollary}\label{corollary:transverse diameter for compact foliations}
Let $(M, \mathcal{F}, g_\intercal)$ be a transverse globally hyperbolic Lorentzian foliation. If the leaves of $\mathcal{F}$ are all compact and if $\ric_\intercal\geq (k-1)C>0$ on unit timelike vectors, then $\diam(M, g)\leq\frac{\pi}{\sqrt{C}}$.
\end{corollary}
\begin{proof}
    The compactness hypothesis trivially implies the leaves are future causally complete. By Proposition (\ref{teo: global hyperbolic does happen}), any Lorentzian bundle-like metric $g$ on $M$ is globally hyperbolic. Hence, Theorem (\ref{diameter}) applies.
\end{proof}
\end{comment}
\subsection{Lorentzian orbifold causality}

Let $\mathcal{O}$ be a connected $q$-dimensional ($q\geq2$) orbifold. As it was seen in Section~\ref{subsection: orbifolds}, $\mathcal{O}$ can be seen as the leaf space of a foliation, and if it is equipped with a time-oriented Lorentzian orbifold metric $g_\mathcal{O}$, then the natural --- and crucial --- question to be asked is whether the causal structure it is endowed with as a leaf space coincides with the one arising from $g_\mathcal{O}$. 

It is not our goal here to give a detailed account of causality on orbifolds, nor is it necessary to do so. Firstly, because our whole endeavor in developing the transverse language in this paper has been precisely to obviate that need! Secondly, apart from technicalities, the causality on the Lorentzian orbifold $(\mathcal{O},g_{\mathcal{O}})$ is essentially the same as that of manifolds as outlined at the beginning of Section~\ref{section: Lorentz fol1}. Thus, for example, we say that a smooth regular curve $\gamma\colon  I\rightarrow\mathcal{O}$ is causal if $g_\mathcal{O}(\gamma^\prime(s),\gamma^\prime(s))\leq 0, \forall s\in I$, where the the meaning of the latter expression and the ``smooth'' appellation are to be understood in terms of orbifolds charts (see, e.g., \cite{alex4} and \cite[Sect. 1.4]{caramello3}). Thus, the meaning of a Lorentzian orbifold being time-oriented, chronological/causal, strongly causal or globally hyperbolic is clear. 

The (affirmative) answer to the issue raised here can now be formulated as follows.
\begin{proposition}\label{keyfinal}
Let $(\mathcal{O}^q, g_\mathcal{O})$ ($q\geq 2$) be a connected time-oriented Lorentzian orbifold. Then, there exists a time-oriented Lorentzian foliation $(M^n, \f, g_\intercal)$ whose leaves are all compact and with finite holonomy such that $M/\f$ is orbifold-diffeomorphic to $\mathcal{O}$ and $g_\intercal=\pi_\f^*g_\mathcal{O}$ with a compatible time-orientation. Given a future-directed timelike [resp. causal] curve $\gamma\colon [0,1]\rightarrow \mathcal{O}$, there exists a future-directed transversely timelike [resp, causal] lift $\tilde{\gamma}\colon [0,1]\rightarrow M$ of $\gamma$ through $\pi_\f$ such that 
$$\ell_\intercal(\tilde{\gamma}) = \ell_{g_{\mathcal{O}}}(\gamma).$$
In particular, given $z_1, z_2\in \mathcal{O}$ corresponding respectively to the leaves $L_1$ and $L_2$ in $M/\f$, one has $z_1\leq_{g_\mathcal{O}}z_2 \iff L_1\leq_{M/\f}L_2$ and $z_1\ll_{g_\mathcal{O}}z_2 \iff L_1\ll_{M/\f}L_2$. 
\end{proposition}

\begin{proof}
We have seen in Prop. \ref{proposition: every orbifold is an isometric quotient} that $\mathcal{O}$ can always be realized as the orbit space of a locally free smooth action of a connected compact Lie group $G$ on a smooth manifold $M^n$, so we take $(M,\f)$ to be such a homogeneous foliation. The projection $\pi_\f\colon M\rightarrow M/G\cong\mathcal{O}$ is then smooth as a map between orbifolds, so we can consider the pullback metric $g_\intercal := \pi_\f^\ast g_{\mathcal{O}}$.

There are certain convenient orbifold charts on $\mathcal{O}=M/G$, analogous to the ones described in Example \ref{example: when leaf spaces are orbifolds}, which we describe very briefly here (see \cite[Thms. 3.49 and 3.57]{alex} for details). First of all we fix a $G$-invariant Riemannian metric $h$ on $M$ with associated exponential map $\exp^h$. Given any $x\in M$, we pick an open $h_x$-ball $B_{\varepsilon}(0) \subset (T_x(G.x))^{\perp _h}$  so that the $q$-dimensional submanifold $S_x=\exp^h_{\perp}(B_\varepsilon(0))\ni x$ is a $G_x$-invariant slice of the action at $x$. Thus, $\mathrm{Tub}(G.x):= G\dot S_x$ is an open tubular neighbourhood of the orbit $G.x$ and hence $\mathrm{Tub}(G.x)/G\simeq S_x/G_x$. Then, $(B_\varepsilon(0),G_x, \pi_\f \circ \exp^h_{\perp})$ is the desired orbifold chart at $\pi_\f(x)$.

Using these charts, it is straightforward to check that $g_\intercal := \pi_\f^\ast g_{\mathcal{O}}$ is a smooth transverse Lorentzian metric on $(M,\f)$, with transverse time-orientation induced by the time-orientation of $g_{\mathcal{O}}$.

Now, given any timelike/causal curve $\gamma\colon [0,1] \rightarrow \mathcal{O}$, 
there exists, by standard result on group actions \cite[Thm. II.6.2]{bredon}, a lift $\tilde{\gamma}\colon [0,1] \rightarrow M$ through $\pi_\f$ of $\gamma$, i.e., such that $\pi_\f\circ \Tilde{\gamma} =\gamma$. Again using the above-described charts we can show that this lift is (piecewise) smooth in the usual sense and, by construction,
$$g_\intercal(\Tilde{\gamma}', \Tilde{\gamma}')= g_{\mathcal{O}}(\gamma ', \gamma '),$$
whence the asserted claims follow. 
%We shall address only the first one since the reasoning is entirely analogous. Note that the $\impliedby$ direction is trivial, since the existence of a future-directed transversely causal curve between $L_1$ and $L_2$ in $M$ implies, by composition with $\pi_\f$ and up to identification, the existence of a future-directed causal curve between $p$ and $q$. For the converse direction, we start with a future-directed causal curve $\gamma:[0,1]\rightarrow \mathcal{O}$ connecting $p$ to $q$. Suppose $\mathcal{O}$ is orientable, then $(M, \f)$ is given by the bundle of oriented normal frames foliated by the orbits of $SO(q)$ (if $\mathcal{O}$ is not orientable, the reasoning follows \textit{mutatis mutandis}). To obtain a lift $\Tilde{\gamma}\colon [0,1]\rightarrow M$ of $\gamma$ we make use of the $SO(q)-$space structure on $M$ (see \cite{bredon}, Thm II.6.2). Since $g_\intercal(\Tilde{\gamma}^\prime, \Tilde{\gamma}^\prime)=g_\mathcal{O}(\gamma^\prime, \gamma^\prime)\leq 0$, it follows that $\Tilde{\gamma}$ is transversely causal and its transverse length is equal to the $g_\mathcal{O}-$length of $\gamma$.
\end{proof}

%and  by the compacity of $im(\gamma)$ we may assume that it lies entirely inside of an orbifold chart $U\subset\mathcal{O}$, which is homeomorphic to $\Tilde{U}/G$, for some open neighborhood $\Tilde{U}\subset\mathbb{R}^q$ and some finite group $G$. Lift $\gamma$ (see \cite{bredon}, Thm II.6.2) to obtain $\Tilde{\gamma}:[0,1]\rightarrow \Tilde{U}$. On the other hand, the open neighborhood $U$ corresponds to a saturated open neighborhood $U^\prime\subset M$ $$\xymatrix{\Tilde{U}\ar@{-->}[rr]\ar[dr]&&U^\prime\ar[ld]^-{\pi_\f}\\&U}$$

Finally, our main result for orbifolds is now immediate from Theorem~\ref{transverse-diameter} and the previous proposition (recall also Remark \ref{remark:basicLCconnectiononquotients}).
%Corollary \ref{corollary:transverse diameter for compact foliations} (and Proposition \ref{proposition: every orbifold is an isometric quotient}).

\begin{corollary}\label{finalorbi}
Let $(\mathcal{O}, g)$ be an $n$-dimensional, globally hyperbolic Lorentzian orbifold. If $\ric_\mathcal{O}\geq (n-1)C>0$ for unit timelike vectors, then $\diam(\mathcal{O}, g)\leq\frac{\pi}{\sqrt{C}}$.
\end{corollary}

\begin{remark}
When a foliation has compact leaves and its leaf space is Hausdorff, then all its leaves have finite holonomy \cite[Thm. 1.4]{millett}. This means that the current formulation of Theorem \ref{transverse-diameter} is a result about Lorentzian orbifolds --- precisely given as Corollary \ref{finalorbi} --- and nothing else, since the associated leaf spaces are exactly of this kind. Notice, however, that the leaf compactness assumption was used only from Lemma \ref{inextransvinex} onwards, and Lorentzian submersion examples suggest that this is not a necessary condition in all circumstances. Thus, if a future development of this theory manages to circumvent this hypothesis, then it is possible that Theorem \ref{transverse-diameter} may be applied to a class of leaf spaces strictly larger than just orbifolds. 
\end{remark}

\section*{Acknowledgements}
The authors wish to thank Vincent Grandjean for many useful discussions during the preparation of this work. We also thank the anonymous referee for their careful examination and for pointing out some additional references for this work. IPCS is partially supported by the project PID2020-118452GBI00 of the Spanish government. This study was financed in part by the Coordenação de Aperfeiçoamento de Pessoal de Nível Superior - Brasil (CAPES) - Finance Code 001.

\address{Departamento de Matemática\\
Universidade Federal de Santa Catarina\\
88040-900, Florianópolis -- SC, Brazil\\
\email{francisco.caramello@ufsc.br}\\
\email{henrique.martins@posgrad.ufsc.br}\\
\email{pontual.ivan@ufsc.br}\\
\received{July 9, 2025}\\
\accepted{March 4, 2026}}

\end{document}